\documentclass[11pt,a4paper]{article}
\pdfoutput=1
\usepackage[margin=1in]{geometry}
\usepackage[utf8]{inputenc}
\usepackage[dvipsnames]{xcolor}
\usepackage[bookmarksnumbered,linktocpage,hypertexnames=false,colorlinks=true,linkcolor=NavyBlue,urlcolor=NavyBlue,citecolor=ForestGreen,anchorcolor=green,breaklinks=true,pagebackref=true,pdfusetitle]{hyperref}
\usepackage{bbm,braket,microtype,amsmath,amssymb,amsthm,amsfonts,dsfont,dynkin-diagrams,mathtools,colortbl,booktabs,algorithm,algpseudocode,graphicx,enumitem,xspace,float,mathdots,ellipsis,caption,subcaption,mleftright,multirow,adjustbox,afterpage,lmodern,footnotebackref,pdfpages,bm,authblk}
\usepackage[T1]{fontenc}
\usepackage{textcomp}
\usepackage[english]{babel}
\usepackage[capitalize,nameinlink]{cleveref}
\usepackage{tikz}
\usepackage{xfrac}
\usepackage{booktabs,longtable,colortbl}

\newtheorem*{theorem*}{Theorem}
\newtheorem{theorem}{Theorem}[section]\crefname{theorem}{Theorem}{Theorems}
\newtheorem{lemma}[theorem]{Lemma}\crefname{lemma}{Lemma}{Lemmas}
\crefname{claim}{Claim}{Claims}
\newtheorem{proposition}[theorem]{Proposition}\crefname{proposition}{Proposition}{Propositions}
\crefname{observation}{Observation}{Observations}
\newtheorem{corollary}[theorem]{Corollary}\crefname{corollary}{Corollary}{Corollaries}
\crefname{conjecture}{Conjecture}{Conjecture}
\theoremstyle{definition}
\newtheorem{definition}[theorem]{Definition}\crefname{definition}{Definition}{Definitions}
\crefname{problem}{Problem}{Problems}
\newtheorem{remark}[theorem]{Remark}\crefname{remark}{Remark}{Remarks}
\newtheorem{example}[theorem]{Example}\crefname{example}{Example}{Examples}
\crefname{condition}{Condition}{Conditions}
\numberwithin{equation}{section}
\numberwithin{table}{section}

\usepackage{aliascnt}
\makeatletter
\let\c@algorithm\relax
\makeatother
\newaliascnt{algorithm}{theorem}

\makeatletter
\let\c@table\relax
\makeatother
\newaliascnt{table}{theorem}

\makeatletter
\let\c@figure\relax
\makeatother
\newaliascnt{figure}{theorem}

\makeatletter
\newenvironment{algorithmbreak}[1]{%
    \def\@captype{algorithm}%
    \par\noindent\hrulefill
    \vspace{-0.8em}
    \caption{#1}
    \vspace{-0.8em}
    \par\noindent\hrulefill\\}%
    {
    \vspace{-0.8em}
\par\noindent\hrulefill
}
\makeatother

\AddToHook{env/lemma/begin}{\crefalias{theorem}{lemma}}
\AddToHook{env/claim/begin}{\crefalias{theorem}{claim}}
\AddToHook{env/proposition/begin}{\crefalias{theorem}{proposition}}
\AddToHook{env/observation/begin}{\crefalias{theorem}{observation}}
\AddToHook{env/corollary/begin}{\crefalias{theorem}{corollary}}
\AddToHook{env/conjecture/begin}{\crefalias{theorem}{conjecture}}
\AddToHook{env/definition/begin}{\crefalias{theorem}{definition}}
\AddToHook{env/problem/begin}{\crefalias{theorem}{problem}}
\AddToHook{env/remark/begin}{\crefalias{theorem}{remark}}
\AddToHook{env/condition/begin}{\crefalias{theorem}{condition}}
\AddToHook{env/equation/begin}{\crefalias{theorem}{equation}}

\AddToHook{env/algorithm/begin}{\crefalias{theorem}{algorithm}}
\AddToHook{env/table/begin}{\crefalias{theorem}{table}}

\def\foo#1\endfoo{}
\newcolumntype{C}{@{}>{\foo}c<{\endfoo}}

\DeclareMathOperator{\aff}{aff}

\DeclareMathOperator{\codim}{codim}
\DeclareMathOperator{\conv}{conv}
\DeclareMathOperator{\diag}{diag}
\DeclareMathOperator{\Disc}{Disc}
\DeclareMathOperator{\GL}{GL}
\DeclareMathOperator{\Herm}{Herm}

\DeclareMathOperator{\maxrank}{maxrank}

\DeclareMathOperator{\rank}{rank}

\DeclareMathOperator{\SL}{SL}
\DeclareMathOperator{\spec}{spec}

\DeclareMathOperator{\supp}{supp}

\DeclareMathOperator{\Tr}{Tr}
\DeclareMathOperator{\U}{U}

\newcommand{\R}{\mathbb{R}}
\newcommand{\C}{\mathbb{C}}
\newcommand{\Z}{\mathbb{Z}}
\newcommand{\Q}{\mathbb{Q}}
\newcommand{\N}{\mathbb{N}}
\newcommand{\F}{\mathbb{F}}
\newcommand{\OO}{\mathcal{O}}

\def\<#1>{\left\langle\ignorespaces#1\unskip\right\rangle}

\newcommand{\ot}{\otimes}

\newcommand{\eps}{\varepsilon}

\DeclarePairedDelimiter\norm{\lVert}{\rVert}

\DeclareMathOperator{\tensorrank}{R}

\newcommand{\weylchamber}{\mathcal D}

\newcommand{\variety}{\mathbb X}

\newcommand{\polysystem}{F^v(h)} %
\newcommand{\tensorpolysystem}{F} %
\newcommand{\cost}{\textnormal{c}} %
\newcommand{\Tdet}{\mathsf{D}}
\newcommand{\Tnurmiev}{\mathsf{T}}
\newcommand{\familyparam}{\alpha}
\newcommand{\TW}{\mathsf{W}}

\newcommand{\weylgroup}{\mathcal{W}}
\newcommand{\ineqs}{\mathcal H}

\newcommand{\tensor}{T} %
\newcommand{\sep}{\mid} %
\newcommand{\lowertriangular}{\mathchoice
    {\tikz[scale=1]{\draw (0,0) -- (4mm,0mm) -- (0,4mm) -- cycle}}
    {\tikz[scale=1]{\draw (0,0) -- (4mm,0mm) -- (0,4mm) -- cycle}}
    {\hspace{1pt}\tikz[scale=0.40]{\draw (0,0) -- (4mm,0mm) -- (0,4mm) -- cycle}}
    {\hspace{0.5pt}\tikz[scale=0.33]{\draw (0,0) -- (4mm,0mm) -- (0,4mm) -- cycle}}
}
\newcommand{\uppertriangular}{\mathchoice
    {\tikz[scale=1]{\draw (0,4mm) -- (4mm,4mm) -- (4mm,0mm) -- cycle}}
    {\tikz[scale=1]{\draw (0,4mm) -- (4mm,4mm) -- (4mm,0mm) -- cycle}}
    {\tikz[scale=0.40]{\draw (0,4mm) -- (4mm,4mm) -- (4mm,0mm) -- cycle}}
    {\tikz[scale=0.33]{\draw (0,4mm) -- (4mm,4mm) -- (4mm,0mm) -- cycle}}
}
\newcommand{\astack}[2]{\mathbin{{\smash{\begin{array}{@{}c@{}}\phantom{#2}\\[-8pt]#1\\[-8pt]#2\end{array}}}}}

\newcommand{\G}{\textnormal{GL}}
\newcommand{\Sl}{\textnormal{SL}}
\newcommand{\K}{\textnormal{U}}

\newcommand{\DeltaB}{\Delta_{\textnormal{B}}}
\newcommand{\Drepr}{{\textnormal{repr}}}
\newcommand{\Dgeom}{{\textnormal{spec}}}
\newcommand{\Dsupp}{{\textnormal{supp}}}
\newcommand{\HWV}{{\textnormal{HWV}}}
\newcommand{\Tgeneric}{T_{\textnormal{r}}}
\newcommand{\ISO}{{\textnormal{ISO}}}
\newcommand{\unit}[1]{\mathsf{U}_{\smash{#1}}}

\newcommand{\MM}{\mathsf{M}}

\newcommand{\degenleq}{\trianglelefteq}
\newcommand{\degengeq}{\trianglerighteq}
\newcommand{\gdim}{n} %
\newcommand{\pdim}{d} %

\newcommand{\nurmievT}[1]{\textnormal{T}_{#1}}

\newcommand{\grayzero}{{\color{gray}0}}
\newcommand{\kronpol}[3]{\Delta(\C^{#1} \ot \C^{#2} \ot \C^{#3})}
\newcommand{\kronpolrepr}[3]{\Delta^\Drepr(\C^{#1} \ot \C^{#2} \ot \C^{#3})}
\newcommand{\kronpolgeom}[3]{\Delta^\Dgeom(\C^{#1} \ot \C^{#2} \ot \C^{#3})}

\newcommand{\cm}{\bullet}
\newcommand{\cl}{\circ}
\definecolor{lightergray}{rgb}{0.9,0.9,0.9}
\definecolor{lighterergray}{rgb}{0.95,0.95,0.95}
\allowdisplaybreaks[4]
\usepackage[normalem]{ulem}

\begin{document}

\title{Computing moment polytopes --- with a focus on tensors, entanglement and matrix multiplication}
\author[1,2]{Maxim van den Berg}
\author[3]{Matthias Christandl}
\author[1]{Vladimir Lysikov}
\author[3]{Harold Nieuwboer}
\author[1]{Michael Walter}
\author[2]{Jeroen Zuiddam}

\affil[1]{Ruhr University Bochum, Bochum, Germany}
\affil[2]{University of Amsterdam, Amsterdam, Netherlands}
\affil[3]{University of Copenhagen, Copenhagen, Denmark}

\date{}
\maketitle

\begin{abstract}

Tensors are fundamental in mathematics, computer science, and physics.
Their study through algebraic geometry and representation theory has proved very fruitful in the context of algebraic complexity theory and quantum information.
In particular, moment polytopes have been understood to play a key role.
In quantum information, moment polytopes (also known as entanglement polytopes) provide a framework for the single-particle quantum marginal problem and offer a geometric characterization of entanglement.
In algebraic complexity, they underpin quantum functionals that capture asymptotic tensor relations. %
More recently, moment polytopes have also become foundational to the emerging field of scaling algorithms in computer science and optimization.

Despite their fundamental role and interest from many angles, much is still unknown about these polytopes, and in particular for tensors beyond $\C^2\ot\C^2\ot\C^2$ and $\C^2\ot\C^2\ot\C^2\ot\C^2$ only sporadically have they been computed.
We give a new algorithm for computing moment polytopes of tensors (and in fact moment polytopes for the general class of reductive algebraic groups) based on a mathematical description by Franz (J.~Lie Theory 2002).

This algorithm enables us to compute moment polytopes of tensors of dimension an order of magnitude larger than previous methods, allowing us to compute with certainty, for the first time, all moment polytopes of tensors in $\C^3\ot\C^3\ot\C^3$, and with high probability those in $\C^4\ot\C^4\ot\C^4$ (which includes the $2\times 2$ matrix multiplication tensor). We discuss how these explicit moment polytopes have led to several new theoretical directions and results.

\vspace{1em}

\end{abstract}

\vspace{3em}
\setcounter{tocdepth}{3}
\tableofcontents
\newpage

\section{Introduction}
\label{section:introduction}
Tensors play a central role in various areas of computer science, mathematics and physics, such as algebraic complexity theory, quantum information theory, and additive combinatorics \cite{burgisser1996algebraic,blaser2013fast,MR1804183,walterEntanglementPolytopes2013,MR3645110,taoCapsetBlogPost2016,ellenberg2017large}.
Fundamental open problems about tensors are strongly tied to questions in computational complexity.
A well-known such problem is to determine the matrix multiplication exponent, which corresponds to the asymptotic rank of the matrix multiplication tensor, and which has been studied intensively for decades from many angles (computational, geometric, algebraic) \cite{strassen1969gaussian, DBLP:conf/focs/Strassen86, MR3081636, burgisserGeometricComplexity2011, MR3376667,MR3631613, alman2024asymmetryyieldsfastermatrix, wigderson2022asymptotic}.
A strong conjecture in this area is Strassen's asymptotic rank  conjecture, which has long been known to be intimately linked to the matrix multiplication exponent, and a recent burst of results has developed a range of strong connection between this conjecture and other problems in computational complexity theory \cite{10.1145/3618260.3649656, 10.1145/3618260.3649620, björklund2024chromaticnumber19999ntime, kaski2024universalsequencetensorsasymptotic}.
In quantum information theory, tensors are the natural formalism to study multipartite entangled quantum states, their applications, and relations under local operations, leading to fundamental problems like the quantum marginal problem \cite{christandlmitchison,daftuar2005quantum,klyachkoQuantumMarginalProblem2004,klyachko2006quantum,walterEntanglementPolytopes2013},
the $N$-representability problem \cite{klyachko2006quantum,altunbulakPauliPrincipleRevisited2008,SchillingChristandlGross2012, SchillingAltunbulakKnechtLopesWhitfieldChristandlGrossReiher2018},
and (asymptotic) entanglement transformations / entanglement distillation
\cite{MR1804183,hayashi2003errorExponents,chitambar2008tripartiteEntanglementTransformations,christandlUniversalPointsAsymptotic2021,jensenAsymptoticSpectrumLOCC2020,bugar2025errorExponentsEntanglement}.

Whereas matrices are understood through simple invariants like their rank, tensors have such intricate structure and relations that understanding them (and aforementioned problems) requires information of a richer quality. The \emph{moment polytope} is a mathematical object that collects such fundamental ``rank-like'' information about the tensor, in a precise sense that allows several different characterizations. Going back decades to fundamental work in representation theory, symplectic geometry and invariant theory \cite{mumfordGeometricInvariantTheory1994, nessStratificationNullCone1984, kempfLengthVectorsRepresentation1979, brion1987momentMapImage}, the relevance of moment polytopes has recently become apparent in several areas:
\begin{itemize}
\item in algebraic complexity theory as potential ``obstructions'' in geometric complexity theory (GCT) \cite{burgisserGeometricComplexity2011} (through understanding inclusions and separations between moment polytopes),
\item as the basis for the construction of elements in Strassen's asymptotic spectrum called the \emph{quantum functionals} \cite{DBLP:conf/focs/Strassen86, strassen1988asymptotic, christandlUniversalPointsAsymptotic2021} (the subject of Strassen's duality theorem for asymptotic rank and the matrix multiplication exponent),
\item in quantum information as entanglement polytopes, which describe reachable quantum marginals \cite{walterEntanglementPolytopes2013}, and in particular the solution to the single-particle quantum marginal and $N$-representability problems \cite{christandlmitchison,daftuar2005quantum,klyachkoQuantumMarginalProblem2004,klyachko2006quantum,altunbulakPauliPrincipleRevisited2008}, and which underpin asymptotic entanglement theory through the quantum functionals \cite{christandlUniversalPointsAsymptotic2021,jensenAsymptoticSpectrumLOCC2020},
\item in optimization through a class of optimization algorithms called scaling algorithms \cite{garg2019operatorScaling,cole2018operatorScaling,burgisser2018alternatingMinimization,burgisser2018tensorScaling,burgisserTheoryNoncommutativeOptimization2019,hirai2023interiorPointManifolds}.
\end{itemize}

Despite their fundamental role and tremendous interest from many mathematical and computational angles, much is still unknown about moment polytopes. In particular, they are notoriously hard to compute. For tensors beyond $\C^2 \otimes \C^2 \otimes \C^2$ \cite{han2004compatibleConditionsEntanglement,sawicki2013threeQubits} and $\C^2\ot\C^2\ot\C^2\ot\C^2$ \cite{walterEntanglementPolytopes2013} only sporadically have they been determined \cite{franz2002,vergneInequalitiesMomentCone2017,buloisAlgorithmComputeKronecker2025a}.
Moreover, very little is known about the inclusions and separation between them and about their ``operational'' meaning, which is particularly relevant for aforementioned applications.

Based on a characterization of moment polytopes by Franz \cite{franz2002}, we introduce a new algorithm to compute moment polytopes of tensors and more general group representations.
Our algorithm computes for the first time all moment polytopes of tensors in $\C^3 \ot \C^3 \ot \C^3$ with certainty in seconds and those in $\C^4 \ot \C^4 \ot \C^4$ with high probability.
Hence, our algorithm reaches an order of magnitude beyond the previously known results in the three-partite case (from dimension $2\cdot2\cdot2=8$ to $4\cdot4\cdot4=64$).
In particular, this allows us (for the first time) to compute the moment polytopes of several tensors of interest, such as the $2\times2$ matrix multiplication tensor.

The algorithm offers rigorous guarantees and is optimized for practical computation.
The basic algorithm comes in three variants:
a deterministic algorithm that computes the moment polytope of a tensor;
a faster randomized variant that is always correct when successful, but may fail to come to a solution with bounded probability;
and a randomized algorithm that is correct with bounded probability.
Additionally, we describe a practically oriented probabilistic variation of the algorithm without rigorous guarantees,
alongside two corresponding probabilistic verification algorithms, the first of which is always correct when successful but may fail to come to a solution,
and the second of which is correct with bounded probability.
(Moreover, we can boost the above three success probabilities arbitrarily.)
A crucial part of the algorithm is covered by Gr\"obner bases algorithms, which we apply over various fields.

A new tool in the ``moment polytope toolbox'', our algorithms, and in particular the resulting concrete description of all moment polytopes in these shapes, are a starting point for proving new structural results on moment polytopes (indeed we will discuss several).

\paragraph{New results.} We summarize here our main results, which we elaborate on in the rest of the introduction:
\begin{itemize}
    \item We develop a new algorithm for the computation of moment polytopes, highly optimized for practicality.
    We describe several (deterministic, probabilistic, heuristic) variations as well as a verification algorithm, and prove rigorous guarantees on correctness and success probabilities.
    Our algorithm is very general and applies to moment polytopes for general reductive groups acting by linear maps on finite-dimensional vector spaces.
    \item Using these algorithms and using an orbit classification due to Nurmiev \cite{nurmievOrbitsInvariantsCubic2000}, we compute exactly all moment polytopes of tensors in $\C^3 \ot \C^3 \ot \C^3$.
    In particular, we reveal all inclusions among them.
        The full descriptions of the computed polytopes are given in \cref{table:333 vertex data} and available online in \cite{vandenBerg2025momentPolytopesGithub}.
    \item We compute with high probability the moment polytopes of tensors in $\C^4 \ot \C^4 \ot \C^4$. In particular we compute it for the $2 \times 2$ matrix multiplication tensor, and note that it is not maximal. We give the computed vertices in \cref{table:444 vertex data}.
    \item The newly computed moment polytopes allow direct computation of the quantum functionals, which gives obstructions for asymptotic restrictions and entanglement transformations.
    \item As an indication of further applications, our computed polytopes led (in subsequent work) to several new structural results regarding matrix multiplication moment polytopes, tensor networks (matrix product states), moment polytope inclusion and explicit non-free tensors, which we discuss in \cref{subsec:intro-outlook} below.
\end{itemize}

\subsection{Tensors, their moment polytopes, and applications}

We give here a brief overview of the context and background of this work on tensors and moment polytopes in various areas.

\paragraph{Tensors as quantum states.}
In quantum information theory, tensors describe pure quantum states in multipartite finite-dimensional quantum systems.
For example, a tensor $T \in \C^a \ot \C^b \ot \C^c$ of norm one is a pure quantum state of the three local (also called \emph{marginal}) systems $\C^a$, $\C^b$ and $\C^c$.
The tensor $T$ describes the global state, including entanglement between the local systems.
The local state in each system is described by a Hermitian linear operator on the local system.
These linear operators are called \emph{marginal} density matrices (akin to marginal distributions in probability theory), and they can be defined as taking the partial trace of $TT^* \in \C^{a\times a} \ot \C^{b\times b} \ot \C^{c \times c}$ with respect to the other two systems. For example, the marginal density matrix of the first system is a Hermitian matrix of shape $a \times a$.

A central goal in quantum information theory is to establish \emph{entanglement monotones}.
These are functions that do not increase under operations that are ``local''. %
Examples of such operations are: acting by unitary matrices on the local systems (LU), LU operations with classical communication (LOCC), and LOCC operations with nonzero success probability (SLOCC).
Mathematically, SLOCC operations correspond to applying matrices $A,B$ and $C$ on the local systems, via $S = (A \ot B \ot C)T$.

Moment polytopes are fundamental entanglement monotones, and hence are also called entanglement polytopes~\cite{walterEntanglementPolytopes2013}.
They succinctly describe constraints on the set of states that a tensor $T$ can be
transformed into by SLOCC operations and taking limits.
That is, we act by invertible matrices $A$, $B$ and $C$ on $T$ but allow also the limits of tensors obtained in this way. We denote the set of all such tensors with $\overline{\G \cdot T}$, where $\G \cdot T \coloneqq \{(A\ot B \ot C)T \mid (A,B,C) \in \GL_a \times \GL_b \times \GL_c\}$, $\GL_n$ denotes the invertible $n \times n$ matrices, and the bar indicates that we include limit points.
The eigenvalues of the three Hermitian marginal density matrices of a tensor in $\C^a\ot\C^b\ot\C^c$ are real,
and when the tensor has unit norm they also sum to one (hence form a probability distribution).
Because we can diagonalize using LU operations, these eigenvalues classify the marginal density matrices.
Denote with $r_\textnormal{A}(T)$ the eigenvalues of the first marginal density matrix, sorted from large to small.
Similarly define $r_\textnormal{B}(T)$ and $r_\textnormal{C}(T)$.
Then we can define the \emph{moment polytope} %
\begin{align}
\label{eq:intro moment polytope geometric}
    \Delta(T) =
    \Big\{ \big( r_\textnormal{A}(S), r_\textnormal{B}(S), r_\textnormal{C}(S) \big)
    \ \Big|\
        S \in \overline{\GL \cdot T}, \norm{S} = 1
    \Big\}
    \subseteq \R^{a} \times \R^{b} \times \R^{c},
\end{align}
which, surprisingly, indeed is a (convex, compact) polytope.

One particular motivation for studying moment polytopes comes from matrix product states (MPS) \cite{ciracMatrixProductStates2021,christandlTensorNetworkRepresentations2020a,acuavivaMinimalCanonicalForm2023}.
Consider a system with $k$ sites arranged on a circle, and give each pair of adjacent systems one maximally entangled pair of dimension $n$; call the resulting tensor $T$.
Then MPS on $k$ sites with bond dimension at most $n$ are exactly those tensors which can be obtained from $T$ by SLOCC.
Thus, the moment polytope of $T$ characterizes the collections of one-body marginal density operators that can be realized (or approximately arbitrarily closely) using MPS of the given bond dimension.

One may also define the set of $(r_\textnormal{A}(S), r_\textnormal{B}(S), r_\textnormal{C}(S))$ with $S$ ranging over all states in $\C^a \ot \C^b \ot \C^c$.
The quantum marginal problem is to decide, given such a triple, whether it is inside this set or not. This set is also a polytope, known as the \emph{Kronecker polytope} due to its close connection to Kronecker coefficients \cite{christandlmitchison, klyachko2006quantum,christandl2007nonzeroKroneckerCoefficients}.
Moreover, it is known that a generic tensor has this polytope as their moment polytope \cite{brion1987momentMapImage}.

As moment polytopes characterize the jointly achievable marginals of a quantum state, they can be used to witness many-particle entanglement from single-particle data~\cite{walterEntanglementPolytopes2013}. This test has been used in experiments~\cite{aguilarExperimentalDeterminationMultipartite2015,zhaoExperimentalDetectionEntanglement2017}, and is relevant in the understanding of Pauli’s principle~\cite{altunbulakPauliPrincipleRevisited2008, SchillingChristandlGross2012, SchillingAltunbulakKnechtLopesWhitfieldChristandlGrossReiher2018}.
Up until now, such experiments have been limited to qubits, since the polytopes of larger dimensional tensors were not known; with our algorithm, we have been able to extend knowledge of these polytopes to three qutrits and further.
The test could readily be used experimentally in order to witness new types of entanglement.

\paragraph{Algebraic complexity theory.}

In algebraic complexity theory, tensors correspond to bilinear computational problems \cite{burgisser1996algebraic}.
Examples include, for every $n$, the \emph{matrix multiplication tensor} $\MM_{n} \in \C^{n^2}\ot\C^{n^2}\ot\C^{n^2}$ describing the multiplication of two $n \times n$ matrices, and for every $a,b$ the \emph{polynomial multiplication tensor} $\mathsf{P}_{a,b} \in \C^{a} \ot \C^{b} \ot \C^{a+b-1}$
describing the multiplication of two univariate polynomials of degrees $a-1$ and $b-1$, and the \emph{unit tensor} $\unit{r} \coloneqq  \sum_{i=1}^r e_i \ot e_i \ot e_i \in \C^{r}\ot\C^r\ot\C^r$ describing element-wise multiplication of two vectors of length $r$.
The complexity of these problems can be measured by the number of required bilinear multiplications between the two inputs, which is called the \emph{rank} of the tensor.
We write $T \leq S$ and say $T$ is a \emph{restriction} of $S$ whenever $T = (A \ot B \ot C)S$ for some matrices $A$, $B$ and $C$ of suitable sizes.
Whenever there exists a restriction $T \leq S$, this means we can compute $T$ using as many bilinear multiplications between inputs as is required for $S$.
Therefore, the rank $\tensorrank(T)$ of $T$ is the smallest $r \in \N$ such that $T \leq \unit{r}$.

A central open problem is the asymptotic complexity of matrix multiplication.
The goal is to determine the smallest (in the sense of infimum) real number $\omega$ such that $\tensorrank(\MM_{n}) = \OO(n^{\omega})$.
This number is called the \emph{matrix multiplication exponent} \cite{strassenRankOptimalComputation1983}.
The best known upper bound on $\omega$ is $2.3721339\ldots$ \cite{alman2024asymmetryyieldsfastermatrix}, and the best current lower bound is the trivial $\omega \geq 2$.

An important property of matrix multiplication is that it is self-similar.
If one splits both matrices up into four blocks, their product can be computed via block matrix multiplication.
This property is observed in the tensor by the fact that taking the Kronecker product of two matrix multiplication tensors results in another matrix multiplication tensor: $\MM_{m} \ot \MM_{\ell} = \MM_{m\ell}$.
As a consequence, knowing the asymptotics of $\tensorrank(\MM_2^{\ot n})$ in $n$ allows us to determine~$\omega$.
We define the \emph{asymptotic rank} as $\astack{\tensorrank}{\sim}(T) \coloneqq \lim_{n\to\infty} \tensorrank(T^{\ot n})^{1/n} = \inf_{n} \tensorrank(T^{\ot n})^{1/n}$, where the second equality holds by Fekete's lemma.
A simple computation shows $\astack{\tensorrank}{\sim}(\MM_2) = 2^\omega$.
In fact, in this characterization of $\omega$, $\MM_2$ can be replaced with any matrix multiplication tensor of fixed shape (with the right-hand side adjusted accordingly).

Properties of $T^{\ot n}$ are notoriously difficult to ascertain.
The moment polytope of $T$ describes some of these properties of representation-theoretic nature.
These come in the form of discrete data given by triples of integer vectors $(\lambda,\mu,\nu) \in \N^a\times\N^b\times\N^c$, each of which has non-negative non-increasing entries summing to $n$."
There exist natural projections of $T^{\ot n}$ to certain tensors $\big[T^{\ot n}\big]{}_{\lambda,\mu,\nu}$ (in the isotypic components of the Schur--Weyl decomposition of $(\C^a\ot\C^b\ot\C^c)^{\ot n}$) with strong representation-theoretic properties.
The moment polytope of $T$ describes the $(\lambda,\mu,\nu)$ such that this projection of $T^{\ot n}$ is non-zero, each normalized to a triple of probability distributions:
\begin{align}
    \label{eq:intro moment polytope rep theory}
    \Delta(T)
    =
    \overline{
    \bigg\{\ \Big(\frac{\lambda}{n}, \frac{\mu}{n}, \frac{\nu}{n}\Big)
    \quad \Big| \quad
    \text{for } n > 0 \text{ and } (\lambda,\mu,\nu)
    \ \text{ such that }\ \big[T^{\ot n}\big]{}_{\lambda,\mu,\nu} \neq 0\ \bigg\}
    },
\end{align}
where the closure is the Euclidean closure.

B\"urgisser and Ikenmeyer posed the problem of determining the moment polytopes of important computational problems, specifically those of $\MM_{n}$ and $\unit{r}$ \cite{burgisserGeometricComplexity2011}.
B\"urgisser, Christandl and Ikenmeyer proved that points which are uniform on two of the three subsystems are contained in the moment polytope of the unit tensor \cite[Theorem~1]{burgisserNonvanishingKroneckerCoefficients2011}.
We make progress on these fundamental questions by computing instances of these moment polytopes explicitly, which we discuss in a moment (\cref{{subsec:intro-333}}). These have led to new structural results for general dimensions (\cref{subsec:intro-outlook}).

\paragraph{Strassen's asymptotic spectrum and the quantum functionals.}

An important application of moment polytopes in algebraic complexity theory is for the construction of the \emph{quantum functionals} \cite{christandlUniversalPointsAsymptotic2021}, which combines the geometric and representation-theoretic perspectives.
The quantum functionals are a family of functions mapping $k$-tensors to $\R_{\geq 0}$, which map $\unit{r}$ to~$r$, are monotone under restriction, multiplicative under Kronecker products and additive under taking direct sums.
The collection of all functions with these properties form the \emph{asymptotic spectrum} of $k$-tensors \cite{strassen1988asymptotic}.
The landmark result by Strassen tells us that given a tensor $T$, its asymptotic rank is equal to the supremum of $f(T)$ for all $f$ in the asymptotic spectrum.
In fact, the asymptotic spectrum describes in its entirety the existence of asymptotic restrictions and degenerations between $k$-tensors, or in the context of quantum information, the existence of asymptotic SLOCC transformations. There is also an analogous theory for (asymptotic) LOCC~\cite{jensenAsymptoticSpectrumLOCC2020,vranaFamilyMultipartiteEntanglement2023}.

It has proven to be a challenge to describe the asymptotic spectrum explicitly.
A breakthrough was the discovery of the quantum functionals $F_\theta$ as elements in the asymptotic spectrum \cite{christandlUniversalPointsAsymptotic2021}.
They are defined for 3-tensors $T$ as $F_\theta(T) \coloneqq 2^{E_\theta(T)}$ with
$E_\theta(T) \coloneqq \max_{(p_1,p_2,p_3) \in \Delta(T)} \theta_1 H(p_1) + \theta_2 H(p_2) + \theta_3 H(p_3)$, where $(\theta_1,\theta_2,\theta_3)$ is any probability distribution and $H$ denotes the Shannon entropy.
Quantum functionals have been used to show barrier results for the techniques used to prove upper bounds on the matrix multiplication exponent \cite{christandl2021barriers}.
It is unknown whether the quantum functionals make up the entire asymptotic spectrum of 3-tensors. If the answer is ``yes'', this would in particular imply the matrix multiplication exponent equals $2$.

The quantum functionals are maximizations of concave functions on the moment polytope, and can hence be computed in polynomial time using standard convex optimization techniques given efficient access to the moment polytope.
Our computational results allow us to compute directly the quantum functionals (and other quantities that can be derived from the moment polytope such as $G$-stable rank \cite{derksen2022gStableRank}) for all tensors in $\C^3 \ot \C^3 \ot \C^3$.
This reveals obstructions for the existence of asymptotic restrictions between these tensors.

\paragraph{Scaling problems.}

In the context of scaling problems, moment polytopes are considered more generally.
The main goal is to determine whether some points are included in some moment polytope, a computational task that includes many applications \cite{burgisserTheoryNoncommutativeOptimization2019}.
Scaling algorithm achieve this task via gradient descent-like optimization over the group in question (in our case the group is $\G_a\times\G_b\times\G_c$).
Most notably, this framework led to efficient algorithms for operator scaling and from that a celebrated polynomial time algorithm for non-commutative polynomial identity testing \cite{allenzhu2018operatorScaling,cole2018operatorScaling,garg2019operatorScaling,burgisserTheoryNoncommutativeOptimization2019}.
For many other potential applications current upper bounds are not sufficient to show polynomial complexity, but in notable cases we believe efficient algorithms are possible.
The efficiency of these algorithms is controlled, among other parameters, by properties of moment polytopes and related quantities.
The main parameter in question is called the \emph{weight margin}, and there is a related quantity called the \emph{gap constant} \cite[Definition~1.16 \& Remark~3.20]{burgisserTheoryNoncommutativeOptimization2019}.
Better bounds on these parameters can lead to better upper bounds on the runtime of scaling algorithms. Our algorithms are a first step in determining these quantities computationally in special cases, which have the potential to lead to better bounds.

\subsection{New algorithms for computing moment polytopes}

We present an algorithm for computing the moment polytope of a tensor $T \in \C^a\ot\C^b\ot\C^c$ based on the description of moment polytopes by Franz~\cite{franz2002}.
Our algorithm advances the computational state-of-the art significantly.
The moment polytopes of tensors in $\C^2\ot\C^2\ot\C^2$ and $\C^2\ot\C^2\ot\C^2\ot\C^2$ were computed via \cref{eq:intro moment polytope rep theory} and a complete understanding of the underlying invariant theory~\cite{walterEntanglementPolytopes2013}; but such an understanding is not available in higher dimensions.
Scaling algorithms~\cite{burgisser2018tensorScaling,burgisserTheoryNoncommutativeOptimization2019} provide suitable membership oracles for moment polytopes which are effective in practice but are not known to run in polynomial time in all parameters in general (in particular it is unknown for tensors).
In any case, these optimization-based techniques do not yield a description of the moment polytopes in terms of vertices or inequalities.

We now discuss our algorithm at a high level (in the tensor setting).
Franz describes moment polytopes in terms of the support of the tensor after applying lower-triangular matrices to the three factors.
Denote by $\supp(T)$ the set of triples $(e_i,e_j,e_k) \in \R^a \times \R^b \times \R^c$ such that $T_{i,j,k} \neq 0$ and denote by $\weylchamber$ the set of triples of vectors with non-increasing entries (called \emph{dominant} vectors) in $\R^a \times \R^b \times \R^c$.
Write $\conv Q$ for the convex hull of a set $Q$.
Then we define the \emph{Borel polytope} of a tensor $S$ as
\begin{align}
\label{equation:franz-tensors intro}
    \DeltaB(S)
    \coloneqq
    \bigcap_{\substack{(A, B, C) \in \G \\\text{lower triangular}}} \conv \supp \big( (A \ot B \ot C) S \big)\ \cap\ \weylchamber
\end{align}
where $\G \coloneqq \GL_a \times \GL_b \times \GL_c$.
Borel polytopes have descriptions akin to \cref{eq:intro moment polytope rep theory,eq:intro moment polytope geometric} as well~\cite{burgisser2018tensorScaling,franks2022minimallengthorbitclosure}.
From the rep\-re\-sen\-ta\-tion-theoretic description it is possible to deduce
that for every~$S$ in a dense subset of the orbit~$\GL \cdot T$, we have~$\DeltaB(S) = \Delta(T)$.
In fact, this dense subset is exactly described by the non-vanishing of a certain set of polynomials, and hence equality holds for a nonempty Zariski-open subset of $\GL \cdot T$.

Franz's description leads to an algorithm for computing moment polytopes.
First generate a random element $S \in \GL \cdot T$, and then iterate over all possible supports, for each support checking whether it is attainable by a lower-triangular action on~$S$. This last step can be achieved by solving a polynomial system, and can be done using symbolic methods such as Gr\"obner basis computation.
The result will then equal $\Delta(T)$ with high probability.
The random element may also be described symbolically; in this way $\Delta(T)$ may be computed with certainty.

However, this approach quickly becomes infeasible due to the exponential number of possible supports, and cannot go much beyond previous methods.
The crucial insight is to instead focus on the inequalities defining $\Delta(T)$.
The inequalities defining $\Delta(T)$ (e.g.\ the inequalities that are tight on some face of $\Delta(T)$)
must all be defining for at least one of the finitely many terms occuring in the intersection in \cref{equation:franz-tensors intro}.
We can characterize such inequalities combinatorially.
The first step of our algorithm computes all of them and stores them into a finite set $\ineqs$.
We call an inequality \emph{attainable} for $S$ whenever
there exists lower triangular matrices $(A,B,C)$ such that all elements of the support $\supp( (A \ot B \ot C)S)$ satisfy the inequality.
For step two of our algorithm we iterate over $\ineqs$ and keep all $h \in \ineqs$ that are attainable. The resulting inequalities $\ineqs_S$ will define $\DeltaB(S)$, after the straightforward intersection with $\weylchamber$.
This describes the basic outline of the algorithm.

We note that both the enumeration of $\ineqs$ and the required Gr\"obner basis computations present a bottleneck for computation, and we provide optimizations to be able to compute the moment polytopes of tensors in $\C^4 \ot \C^4 \ot \C^4$.
A notable example is a heuristic which performs the Gr\"obner basis computations instead over finite fields of some random (large) prime characteristic.
To ensure correctness, we also provide a verification algorithm that determines (deterministically or probabilistically) whether $\Delta(T)$ equals a given polytope, which also integrates (tensor) scaling algorithms \cite{burgisser2018tensorScaling,burgisserTheoryNoncommutativeOptimization2019}.

\subsection{Explicit moment polytopes of \texorpdfstring{$3\times3\times3$ and $4\times4\times4$}{3x3x3 and 4x4x4} tensors}\label{subsec:intro-333}

We use our algorithms to compute all moment polytopes for tensors in $\C^3 \ot \C^3 \ot \C^3$.
We find that there are 28 moment polytopes (29 when we count the empty polytope), up to a cyclic symmetry. See \cref{table:333 vertex data} for the vertex descriptions.

For this computation we use a classification of all tensor orbits of this shape by Nurmiev \cite{nurmievOrbitsInvariantsCubic2000, ditraniClassificationRealComplex2023}.
This is not a finite classification, but indeed has families with continuous parameters.
To deal with the continuously parametrized families, we prove that certain families have the largest possible moment polytope.
Before our work, only this largest polytope had been computed \cite{franz2002}, and it was also not known which tensors had it as their moment polytope.
Our computations in particular reveal all inclusion relations between the moment polytopes in this format (see \cref{fig:c333-unstable-moment-polytope-inclusions-intro}). %

\begin{figure}
    \centering
\makebox[\textwidth][c]{
    \begin{tikzpicture}[
        node/.style = {circle, draw, fill=red!20, minimum size=7mm, inner sep=1pt},
        edgedotted/.style = {-Stealth,shorten >=10pt, shorten <=10pt,
                      },
        nodesq/.style = {draw, fill=blue!20, minimum size=7mm, inner sep=1pt},
        nodesqg/.style = {draw, fill=green!20, minimum size=7mm, inner sep=1pt},
        >=stealth, shorten >=1pt, auto,
        rotate=-90,xscale=0.9,yscale=1.2,every node/.style={scale=0.95}
      ]
      \node[node] (1) at (0, 0)   {$\Tnurmiev_1$};
      \node[node] (2) at (0, 1)   {$\Tnurmiev_2$};
      \node[node] (3) at (-1, 2)  {$\Tnurmiev_3$};
      \node[node] (4) at (1, 2)   {$\Tnurmiev_4$};
      \node[node] (5) at (-1, 3)  {$\Tnurmiev_5$};
      \node[node] (6) at (1, 3)   {$\Tnurmiev_6$};
      \node[node] (7) at (-2, 4)  {$\Tnurmiev_7$};
      \node[node] (8) at (0, 4)   {$\Tnurmiev_8$};
      \node[node] (9) at (2, 4)   {$\Tnurmiev_9$};
      \node[node] (10) at (-2, 5) {$\Tnurmiev_{10}$};
      \node[node] (11) at (0, 5)  {$\Tnurmiev_{11}$};
      \node[node] (12) at (2, 5)  {$\Tnurmiev_{12}$};
      \node[node] (13) at (-2, 6) {$\Tnurmiev_{13}$};
      \node[node] (14) at (0, 6)  {$\Tnurmiev_{14}$};
      \node[node] (15) at (2, 6)  {$\Tnurmiev_{15}$};
      \node[node] (16) at (0, 7)  {$\Tnurmiev_{16}$};
      \node[node] (17) at (-2, 7) {$\Tnurmiev_{17}$};
      \node[node] (18) at (2, 7)  {$\Tnurmiev_{18}$};
      \node[node] (19) at (-1, 8) {$\Tnurmiev_{19}$};
      \node[node] (20) at (0, 9)  {$\Tnurmiev_{20}$};
      \node[node] (21) at (0, 10) {$\Tnurmiev_{21}$};
      \node[node] (22) at (1, 8)  {$\Tnurmiev_{22}$};
      \node[node] (23) at (0, 11) {$\Tnurmiev_{23}$};
      \node[node] (24) at (0, 12) {$\Tnurmiev_{24}$};
      \node[node] (25) at (0, 13) {$\Tnurmiev_{25}$};
      \node[nodesq] (gen) at (-3.3, 0) {$\,\unit{3}\,$};
      \node[nodesq] (DW) at (-3.3, 2.5) {$\,\Tdet+\TW\,$};
      \node[nodesq] (De111) at (-3.3, 4.4) {$\,\Tdet+e_{111}\,$};
      \node[nodesq] (D) at (-3.3, 6) {$\Tdet$};
      \path[->] (1) edge (2);
      \path[->] (2) edge (3) edge (4);
      \path[->] (3) edge (5) edge (6);
      \path[->] (4) edge (5) edge (6);
      \path[->] (5) edge (7) edge (8) edge (9);
      \path[->] (6) edge (7);
      \path[->] (6) edge (8);
      \path[->] (7) edge (10) edge (11) edge (12);
      \path[->] (8) edge (10) edge (12);
      \path[->] (8) edge (11);
      \path[->] (9) edge (12);
      \path[->] (10) edge (13) edge (14);
      \path[->] (11) edge (14) edge (15);
      \path[->] (12) edge (13) edge (15);
      \path[->] (13) edge (16) edge (17) edge (18);
      \path[->] (14) edge (16) edge (17);
      \path[->] (15) edge (16);
      \path[->] (16) edge (19) edge (22);
      \path[->] (17) edge (19);
      \path[->] (18) edge (19);
      \path[->] (19) edge (20);
      \path[->] (20) edge (21);
      \path[->] (21) edge (23);
      \path[->, bend right=10] (22) edge (23);
      \path[->] (23) edge (24);
      \path[->] (24) edge (25);
      \path[->]  (DW) edge (De111);
      \path[->]  (De111) edge (D);
      \path[->]  (gen) edge (1);
      \path[->]  (gen) edge (DW);
      \path[->]  (DW) edge (7);
      \path[->]  (DW) edge (8);
      \path[->]  (De111) edge (10);
      \path[->]  (D) edge (17);
    \end{tikzpicture}}
    \caption{Overview of inclusions among the moment polytopes of tensors in $\C^3 \ot \C^3 \ot \C^3$, up to cyclic permutations of the factors.
        An arrow is drawn from polytope $P$ to polytope $Q$ if $P \supseteq Q'$ for a $Q'$ that can be obtained from $Q$ by a permutation of the three factors.
        Nodes are labeled by representative tensors with this moment polytope.
        The tensor $\Tnurmiev_i$ is tensor $i$ from \cref{table:c333 unstable tensors info}.
        The others are defined by the following tensors:
        $\unit{3} \coloneqq e_{111} + e_{222} + e_{333}$,
        $\Tdet \coloneqq e_1 \wedge e_2 \wedge e_3$, and $\TW \coloneqq e_{112}+e_{121}+e_{211}$ (where $e_{ijk} \coloneqq e_i \ot e_j \ot e_k$).
    }
    \label{fig:c333-unstable-moment-polytope-inclusions-intro}
\end{figure}
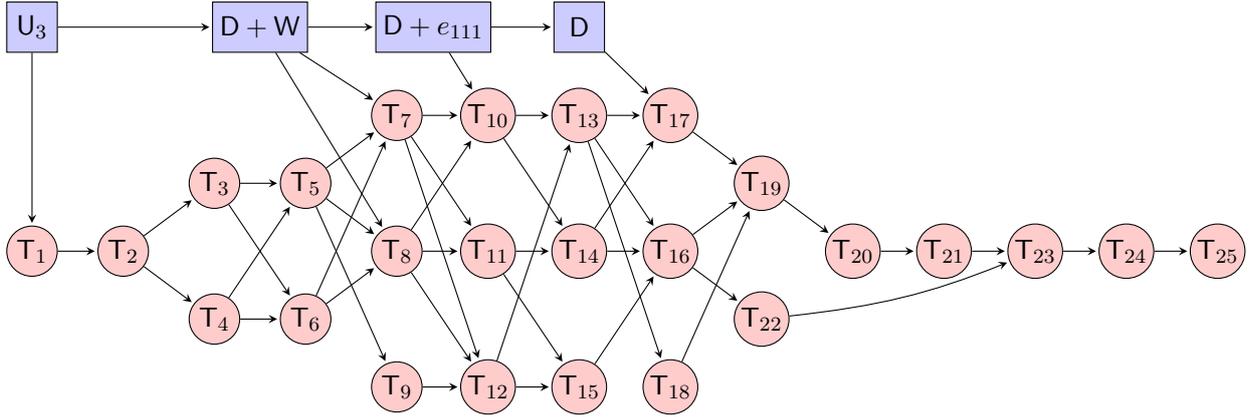

Secondly, we use our algorithms to probabilistically compute moment polytopes of tensors in $\C^4 \ot \C^4 \ot \C^4$, see \cref{table:444 vertex data}. Most notably, we compute moment polytope of the $2 \times 2$ matrix multiplication tensor and note that it is not maximal. (In subsequent work
we prove this non-maximality for $n \times n$ matrix multiplication tensors for every $n \geq 2$, which we discuss in \cref{subsec:intro-outlook}.)

The vertex descriptions of all moment polytopes that we computed are, besides in \cref{table:333 vertex data,table:444 vertex data}, also available in a machine-readable format at \cite{vandenBerg2025momentPolytopesGithub}.

\subsection{Algorithms for moment polytopes for general group actions}
\label{subsection:intro-general-moment-polytopes}

Our algorithm generalizes naturally to a broad range of groups and representations. %
We can replace the group $\GL_a \times \GL_b \times \GL_c$ with any reductive algebraic group,
such as any subgroup $G$ of $\GL_n$ defined by polynomial equations that is closed under taking complex conjugates.
This includes for instance all classical Lie groups and their products.
We can replace the representation $\C^a\ot\C^b\ot\C^c$ by any rational representation of $G$. %

Areas of application include symmetric tensors (polynomials) in the setting of algebraic complexity theory \cite{burgisserGeometricComplexity2011} and quiver representations \cite{chindris2024hivetypepolytopesquivermultiplicities,vergne2023momentconemembershipquivers}.
Another famous example is Horn's problem \cite{MR1654578,MR1671451,MR2177198,berline2016hornInequalitiesGeometric} which may be seen as a special case of the quantum marginal problem \cite{daftuar2005quantum,ChristandlSahinogluWalter2018}.
The single-particle $N$-representability problem also fits into this framework \cite{klyachko2006quantum,altunbulakPauliPrincipleRevisited2008}.
See also \cite{burgisserTheoryNoncommutativeOptimization2019}, which describes the generalization of tensor scaling \cite{burgisser2018tensorScaling} in much the same spirit.

Almost all optimizations we develop generalize.
For accessibility and clarity, we present the algorithm for tensors specifically. Throughout the text we will indicate  where necessary what the general picture looks like.
\subsection{Outlook}\label{subsec:intro-outlook}

We believe our algorithm for computing moment polytopes will be of interest for the discovery of relevant patterns in examples and towards generating new conjectures, and that this addition to the ``moment polytope toolbox'', alongside scaling algorithms, will be a useful tool for future work on moment polytopes.
In particular, our algorithm brings moment polytope computation ``up to speed'' with general methods for Kronecker polytope computation, which is currently known up to $\C^4 \ot \C^4 \ot \C^4$ \cite{vergneInequalitiesMomentCone2017} (however, several special cases for Kronecker polytopes of larger dimension are known as well, which we discuss in \cref{section:kronecker polytopes}).%
\footnote{\cite{buloisAlgorithmComputeKronecker2025a} further determined the (irredundant) inequalities for the Kronecker polytope for shape $5 \times 5 \times 5$ during preparation of this paper. }

We also note that moment polytope for tensors can have an exponential number of vertices and inequalities, as was observed by \cite{ressayrePersonalCommuncation,burgisser2018tensorScaling,garg2017operatorScalingBrascampLieb,burgisserTheoryNoncommutativeOptimization2019}.
Therefore, no algorithm can compute complete descriptions of moment polytopes of tensors efficiently.
Rather than asymptotic complexity, our algorithm improves over previous methods in terms of practicality, with an ``experimental mathematics'' goal in mind of computationally generating a large set of examples from which we can extract general results.
Indeed, we demonstrate this in our two complementary papers:

\begin{itemize}
\item In \cite{vandenBerg2025mmPolytope}, inspired by our computational results on the moment polytope of the $2\times2$ matrix multiplication tensor $\MM_2$, we prove that $\Delta(\MM_n)$ is not equal to the largest moment polytope for every $n \geq 2$.
In fact, we find explicit points which are contained in $\Delta(\unit{r})$ but not in $\Delta(\MM_n)$ (for suitable $n,r$), which imply obstructions for degenerations from $\MM_n$ to $\unit{r}$ (in other words, upper bounds on the subrank of $\MM_n$).
This makes progress on the question of B\"urgisser and Ikenmeyer \cite{burgisserGeometricComplexity2011} to determine $\Delta(\unit{r})$ and $\Delta(\MM_n)$, and in particular implements for the first time their intended objective: to use moment polytopes to obtain complexity-theoretic obstructions.

\item
In the context of quantum information, the results in \cite{vandenBerg2025mmPolytope} also show that for $k \geq 3$ and any $n \geq 2$, matrix product states (MPS) obey interesting constraints on top of those that are inherent from being the marginal distributions of a (pure) quantum state.
Our techniques and results may also lead to new insights in other connectivity scenarios for tensor networks~\cite{christandlTensorNetworkRepresentations2020a,christandl2023resourcetheorytensornetworks, Christandl2024Tensor}.
\item In \cite{vandenBerg2025nonFreeTensor}, we used insights from the computed moment polytopes in $\C^3\ot\C^3\ot\C^3$ to determine that two tensors in that format are not \emph{free}, which means that under no local basis changes the support of the tensor contains two elements that only differ in a single coordinate.
This proof subsequently led to the first construction of non-free tensors in $\C^n\ot\C^n\ot\C^n$ for every $n \geq3$.
Generic tensors are non-free (for any $n\geq 4$), but no explicit non-free tensors were known before our work.
The proof relies on a connection between freeness and moment polytopes discovered primarily as a result of our computations.
\end{itemize}

\subsection{Overview of the paper}

In \cref{section:tensor moment polytopes} we provide an introduction to the theory of moment polytopes specialized to the setting of tensors. We treat the three main descriptions of moment polytopes.
We describe in \cref{section:algorithms for computing tensor moment polytopes} our algorithm for computing moment polytopes.
In \cref{section:implementation and optimization} we make concrete several aspects of the algorithm and describe the optimizations used to make it practical.
In \cref{section:C333} we explain the classification of tensors in $\C^3\ot\C^3\ot\C^3$ and compute all their moment polytopes.
Lastly, in \cref{section:algorithms for 4x4x4} we heuristically extend our algorithms to make probabilistic computation of tensors in $\C^4 \ot \C^4 \ot \C^4$ possible. To enable rigorous guarantees on the probability of correctness of the heuristic, we also describe a verification procedure.

\section{Moment polytopes of tensors}
\label{section:tensor moment polytopes}

In this section we provide an introduction to the theory of moment polytopes.
The material is based on \cite{nessStratificationNullCone1984, brion1987momentMapImage, franz2002, burgisser2018tensorScaling, burgisserTheoryNoncommutativeOptimization2019}.
We treat the three main descriptions of the moment polytope of a tensor, which we call the \emph{geometric description}, the \emph{representation-theoretic description} and the \emph{support description}.
This section does not contain new results, however both the specialization of all descriptions to tensors and the unification of the three descriptions into a single source have not been available in the literature before.
We also put for the first time the description of a strongly related polytope, which we call the \emph{Borel polytope} as coined for the geometric description in \cite{burgisser2018tensorScaling}, into this unified framework.
We note in particular that it had also appeared in \cite{franz2002} via a description using supports.

After introducing the setting in \cref{subsection:setting}, we go over the main descriptions of the moment polytopes in \cref{subsection:definitions and overview}.
In \cref{subsection:tensor generic element} we do the same for the Borel polytopes.
\Cref{section:moment polytopes basic properties} contains important properties of moment polytopes in relation to restriction and degeneration, and \cref{section:kronecker polytopes} introduces Kronecker polytopes.

\begin{remark}
All theory presented here generalizes greatly.
We restrict our attention to the tensor setting for accessibility.
The generalization concerns polynomial representations of reductive algebraic groups (which includes all rational representations of products of general linear groups).
See also \cite{burgisserTheoryNoncommutativeOptimization2019} for an accessible introduction to this generalization for rational representations of $\GL_n$.
We refer to \cite{nessStratificationNullCone1984,brion1987momentMapImage,franz2002} for the general theory.
See e.g.\ \cite{procesi2007,onishchik2012,borel2012} for introductions to the theory of algebraic groups.
\end{remark}

\subsection{Setting and notation}
\label{subsection:setting}

For clarity of exposition, we limit our discussion to 3-tensors, that is, tensors with three indices or factors.
All concepts generalize to tensors with any fixed number of factors in a straightforward way.
Thus let $\tensor \in \C^{a} \ot \C^{b} \ot \C^{c}$ be a 3-tensor.
We call $(a,b,c)$ the \emph{shape} of $T$ and sometimes denote it as $a \times b \times c$.
The spaces $\C^{a}$, $\C^{b}$ and $\C^c$ are referred to as \emph{factors} or \emph{marginal systems}.
Denote by $\GL_a$ the group of invertible $a \times a$ matrices.
The group\footnote{The integers $a,b$ and $c$ are suppressed in the notation, but will always be clear from context.}
\begin{align}
    \GL \coloneqq \GL_a \times \GL_b \times \GL_c
\end{align}
acts on tensors by simultaneous basis change on each tensor factor, that is,
$ (A,B,C) \cdot T \coloneqq (A \ot B \ot C)T. $
Relevant subgroups are the special linear subgroup $\Sl \coloneqq \SL_a \times \SL_b \times \SL_c$ and the unitary subgroup $\K \coloneqq \U_a \times \U_b \times \U_c$, where $\SL_a$ and $\U_a$ denote the $a\times a$ matrices with determinant~$1$ and unitary $a\times a$ matrices respectively.

We will work with both the Euclidean and Zariski topologies on (finite dimensional) vector spaces. With $\norm{\cdot}$ we will always denote the Euclidean norm on $\C^n$. For a subset $X$ of the vector space, we denote with $\overline{X}$ the Euclidean closure.
With a \emph{variety} we mean a closed set in the Zariski topology.
A variety is called \emph{irreducible} if it is not the union of two closed sets that are proper subsets of the variety.
As is customary in algebraic geometry, we say a statement is ``true for generic $v \in X$'' if and only if the statement is true for all elements $v \in U \cap X$, where $U$ is some Zariski-open set and $X \cap U$ is a dense subset of $X$ (equivalently $X \cap U \neq \varnothing$ if $X$ is an irreducible variety).
We call such a property a \emph{generic property} on $X$.

For us, polytopes will always be convex and compact. That is, a polytope is a convex hull $\conv\{q_1,\ldots,q_\ell\}$ of a finite set of points $q_i \in \R^n$.
When $q_1,\ldots,q_\ell \in \Q^n$ we say the polytope is \emph{rational}.

Throughout this paper we will freely identify $\R^a \times \R^b \times \R^c$ with $\R^{a+b+c}$. For example, the vector $\big((1,0),(1,0),(1,0)\big) \in \R^2\times\R^2\times\R^2$ will be identified with %
$(1,0\sep1,0\sep1,0) \in \R^{6}$, where the vertical bars indicate the original separation of the systems.

\subsection{Moment polytopes}

\label{subsection:definitions and overview}

We now give the geometric and representation-theoretic definitions of the moment polytope of a tensor $T \in \C^a\ot\C^b\ot\C^c$, as well as Franz's description.
Equality of these three descriptions is stated in \cref{theorem:tensor moment polytopes nm} and \cref{theorem:tensor moment polytopes franz}.

\subsubsection{Geometric description}
The first description of the moment polytope, which we call the \emph{geometric description}, has its origins in symplectic geometry.
It is defined as follows.
By grouping the last two tensor factors, we can consider $T \in \C^a \ot (\C^b \ot \C^c)$ as a matrix $T_1$ mapping $\C^a$ to $\C^b \ot \C^c$.
Left-multiplying with its conjugate transpose, we obtain a Hermitian matrix $T_1^* T_1^{\phantom*} \in \C^{a \times a}$.
The two other ways to group the three factors lead to the matrices  $T_2^* T_2^{\phantom*} \in \C^{b \times b}$ and $T_3^* T_3^{\phantom*} \in \C^{c \times c}$.
Denote the set of Hermitian $a \times a$ matrices with $\Herm_a$.
The \emph{moment map} $\mu \colon \C^a\ot\C^b\ot\C^c \setminus\{0\} \to \Herm_a \times \Herm_b \times \Herm_c$ associates to a non-zero $T$ the above three Hermitian matrices, normalized using their traces:
\begin{align}
\label{equation:tensor moment map}
    \mu(T)
    \coloneqq
    \bigg( \frac{T_1^*T_1^{\phantom*}}{\Tr\!\big[T_1^*T_1^{\vphantom*}\big]}, \frac{T_2^*T_2^{\vphantom*}}{\Tr\!\big[T_2^*T_2^{\vphantom*}\big]}, \frac{T_3^*T_3^{\phantom*}}{\Tr\!\big[T_3^*T_3^{\vphantom*}\big]}
    \bigg).
\end{align}
These three matrices are called the \emph{marginals} of $T$.
Their eigenvalues are real and non-negative, and sum to 1 due to the normalization.
Note that $\Tr[T_i^*T_i^{\phantom*}] = \norm{T}^2$ does not depend on $i$. %
For any Hermitian matrix $M \in \C^{n \times n}$, let $\spec(M) = (\lambda_1, \lambda_2, \ldots, \lambda_n) \in \R^n$, where $\lambda_1 \geq \lambda_2 \geq \cdots \geq \lambda_n$ denote the eigenvalues of~$M$ ordered non-increasingly.
Define
\[
\spec(\mu(T)) \coloneqq
\big(\spec(\mu_1(T)) \ \big|\ \spec(\mu_2(T)) \ \big|\ \spec(\mu_3(T))\big) \in \R^{a+b+c}.
\]
Then the geometric expression for the moment polytope of $T$ is given by
\begin{align}
    \label{equation:geometric description}
    \Delta^\Dgeom(T)
    \coloneqq
    \Big\{\, \spec \bigl( \mu(T') \bigr)
    \ \ \Big|\ \
    T' \in \overline{\G \cdot T} \setminus\{0\}
    \,\Big\}
    \quad\subseteq \R^{a+b+c},
\end{align}
where $\overline{\G \cdot T}$ denotes the Euclidean closure of $\G \cdot T$.
One can show $\mu$ equals the gradient of the map $(A,B,C) \mapsto \log \norm{(A,B,C) \cdot T}$ at the identity element~$(I_a,I_b,I_c)$ \cite{burgisser2018tensorScaling,burgisserTheoryNoncommutativeOptimization2019}.
This map is known as the \emph{Kempf--Ness function} \cite{kempfLengthVectorsRepresentation1979}.

\subsubsection{Representation-theoretic description}
\label{subsection:representation-theoretic description}

To present the representation theoretic viewpoint, we use some basic concepts from representation theory (see one of \cite{fulton1991representation, etingof2011introductionrepresentationtheory, procesi2007} for an introduction).
Recall that a \emph{polynomial representation} of $\G$ is a vector space $V$ together with a group homomorphism $\rho \colon \G \to \GL(V)$ that can be written in coordinates as a polynomial map.
A \emph{rational representation} of $\G$, or simply \emph{$\G$-representation}, additionally allows multiplication with inverse powers of the determinants $\det(A)$, $\det(B)$ and $\det(C)$ for $(A,B,C) \in \G$.
Finite-dimensional irreducible rational representations of $\G$ are classified up to isomorphism.
Define the set of dominant vectors in $\R^m$ as $\weylchamber_m \coloneqq \big\{\lambda \in \R^{m} \mid \lambda_{1} \geq \lambda_{2} \geq \cdots \geq \lambda_{m}\big\}$. Define
\begin{align}
\label{equation:tensor weyl chamber}
    \weylchamber
    \coloneqq \weylchamber_a \times \weylchamber_b \times \weylchamber_c,
    \qquad\text{and} \qquad
    \weylchamber_\F
    \coloneqq \weylchamber \cap (\F^a \times \F^b \times \F^c)
    \quad\text{for any } \F \subseteq \R.
\end{align}
Then every irreducible representation of $\G$ corresponds to an unique element
$(\lambda, \mu,\nu) \in \weylchamber_\Z$ and each element of $\weylchamber_\Z$ has an associated irreducible representation $V_{\lambda,\mu,\nu}$.
Any $\G$-representation $W$ decomposes as a direct sum of irreducible representations
    $W \cong
    \bigoplus_{(\lambda,\mu,\nu) \in \weylchamber_\Z} (V_{\smash{\lambda,\mu,\nu}})^{\oplus m_{\lambda,\mu,\nu}}$, with $m_{\lambda,\mu,\nu} \in \N$.
The subspace of $W$ that gets mapped to $(V_{\smash{\lambda,\mu,\nu}})^{\oplus m_{\lambda,\mu,\nu}}$ under the isomorphism is called the \emph{isotypic component} of type $(\lambda,\mu,\nu)$ of $W$.

Our $\G$-representation of interest is the tensor representation $V = \C^{a} \ot \C^{b} \ot \C^{c}$, where the entries of $(A,B,C) \cdot T \coloneqq (A\ot B \ot C)T$ can be written as homogeneous degree 3 polynomials in the entries of $A,B$ and $C$.
Let $W \coloneqq V^{\ot n}$ with $n \in \N$, which is a polynomial representation of $\G$ via the map $(A,B,C) \mapsto (A \ot B \ot C) \ot \cdots \ot (A \ot B \ot C)$.%
\footnote{By grouping the tensor factors as $W \cong (\C^a)^{\ot n} \ot (\C^b)^{\ot n} \ot (\C^c)^{\ot n}$, one may regard the elements of $W$ as 3-tensors as well, and the $n$-th tensor power is then called is then called the \emph{$n$-th Kronecker power} of the tensor.}
The multiplicities $m_{\lambda,\mu,\nu}$ are non-zero only if
$(\lambda,\mu,\nu) \in \weylchamber_{\N}$ and $\sum_i \lambda_i = \sum_i \mu_i = \sum_i \nu_i = n$.
Let $\ISO_{\lambda,\mu,\nu}$ be the unique projection along the direct sum onto the isotypic subspace of type $(\lambda,\mu,\nu)$.%
\footnote{The projection $T \mapsto \ISO_{\lambda,\mu,\nu} T^{\ot n}$ can be interpreted as $m_{\lambda,\mu,\nu}$ equivariant degree $n$ polynomial maps to $V_{\lambda,\mu,\nu}$, which are known as \emph{covariants} of weight $(\lambda,\mu,\nu)$ \cite{walterEntanglementPolytopes2013}.}
Then the representation-theoretic description of the moment polytope equals
\begin{align}
\label{equation:representation theoretic description}
    \Delta^\Drepr(T)
    \coloneqq
    \overline{
    \bigg\{\ \Big(\frac{\lambda}{n}, \frac{\mu}{n}, \frac{\nu}{n}\Big)
    \quad \Big| \quad
    n > 0
    \ \text{ and }\
    \ISO_{\lambda,\mu,\nu} T^{\ot n} \neq 0\ \bigg\}
    }
    \quad\subseteq \R^{a+b+c},
\end{align}
where we take the closure with respect to the Euclidean topology.%
\footnote{The condition that $\ISO_{\lambda,\mu,\nu} T^{\ot n} \neq 0$ is equivalent to $V_{\lambda,\mu,\nu}^*$ having non-zero multiplicity in the coordinate ring of $\overline{\G \cdot T}$ \cite{walterEntanglementPolytopes2013}.
}
\subsubsection{Support description}
\label{subsubsection:description using supports}

Denote with $T_{i,j,k}$ the coefficients of $T$ in the standard basis, i.e.\ the coefficients satisfying $T = \sum_{i,j,k} T_{i,j,k}\, e_i\ot e_j \ot e_k$, where $e_\ell$ denotes the standard basis vector of the correct dimension.
The \emph{support} of $T$ is the selection of indices $(i,j,k)$ at which $T_{i,j,k}$ is non-zero.
We embed these indices as vectors $(e_i,e_j,e_k) \in \R^a \times \R^b \times \R^c$,
and write
\begin{align}
    \supp T \coloneqq \big\{(e_i,e_j,e_k) \mid T_{i,j,k} \neq 0\big\} \subseteq \R^a \times \R^b \times \R^c.
\end{align}
We call elements $(e_i,e_j,e_k) \in \R^a \times \R^b \times \R^c$ \emph{weights}, and we denote the set of all weights as
\begin{align}
\label{equation:tensor weights}
    \Omega \coloneqq \Omega_a \times \Omega_b \times \Omega_c,
    \qquad \text{ where }\
    \Omega_m \coloneqq \{e_1,\ldots,e_m\} \subseteq \R^m.
\end{align}

The description of the moment polytope due to Franz \cite{franz2002} makes use of a strongly related polytope, which we will call the \emph{Borel polytope} of $T$ and for which we introduce the support description now.
In \cref{subsection:tensor generic element} we provide geometric and representation-theoretic characterizations of this polytope.\footnote{The name ``Borel polytope'' was first used in \cite{burgisser2018tensorScaling} for the geometric characterization. It refers to the fact that the subgroup of lower triangular matrices is known as a \emph{Borel subgroup} in the general setting. (The standard choice of a Borel subgroup is the set of upper triangular matrices, in which case the set of lower triangular matrices is referred to as the \emph{opposite Borel subgroup}.)}
Denote with $\G_{\lowertriangular} \subseteq \G$ the subgroup of triples of lower triangular matrices. Then the support description of the Borel polytope is given by
\begin{align}
\label{equation:franz-tensors}
    \DeltaB^\Dsupp(T)
    \coloneqq
    \bigcap_{\substack{(A,B,C) \in \G_{\lowertriangular}}} \conv \supp\! \big( (A \ot B \ot C) T \big) \ \cap\ \weylchamber
    \quad\subseteq \R^{a+b+c}.
\end{align}
It is immediately clear that $\DeltaB^\Dsupp(T)$ is a rational polytope: there are only finitely many possibly supports, so $\DeltaB^\Dsupp(T)$ equals the intersection of a finite set of rational polytopes.

The Borel polytope of $T$  is contained in, but does always equal the moment polytope of $T$. However, equality holds for generic $\Tgeneric \in \GL \cdot T$, which we discuss below.

\subsubsection{Equality of the characterizations}

We arrive at the central theorems.
They state the equality of the objects defined above, provided we choose a generic element in the orbit of $T$ to evaluate the support description.
These theorems are originally stated more generally for irreducible varieties $\variety \subseteq \C^a\ot\C^b\ot\C^c$.\footnotemark{} Or indeed, for irreducible varieties in any finite dimensional representation of a reductive group.
\footnotetext{The geometric description of the moment polytope of $\variety$ is given by $\{\spec(\mu(T)) \mid T \in \variety\setminus\{0\}\}$,
and the representation-theoretic description is the same as \cref{equation:representation theoretic description} but adding an existential quantifier over $T \in \variety$. The support description is the same as \cref{equation:franz-tensors}, but with $T$ taken generically from $\variety$.}

We begin with the theorem due to Ness, Mumford and Brion.
\begin{theorem}[{\cite{nessStratificationNullCone1984,brion1987momentMapImage}}, special case]
\label{theorem:tensor moment polytopes nm}
    Let $T$ be a tensor. Then
    \begin{align}
        \Delta^\Dgeom(T) = \Delta^\Drepr(T).
    \end{align}
    We call it the \emph{moment polytope} of $T$ and denote it with $\Delta(T)$.
    It is a rational polytope.
\end{theorem}
Denote with $\G_{\uppertriangular} \subseteq \G$ the subgroup of upper triangular matrices.
The theorem due to Franz is then as follows.
\begin{theorem}[{\cite{franz2002}}, special case\footnote{This result also appeared in \cite[Corollary~2.14]{burgisser2018tensorScaling} in a slightly different form, see \cref{theorem:tensor randomization}.
}]
\label{theorem:tensor moment polytopes franz}
    Let $T$ be a tensor. For any $T' \in \GL \cdot T$ we have
    \begin{align}
        \DeltaB^\Dsupp(T') &\subseteq \Delta(T).
    \intertext{
    Taking $\Tgeneric$ generically from $\G \cdot T$, or generically from $\G_{\uppertriangular} \cdot T$, we have
    }
        \DeltaB^\Dsupp(\Tgeneric) &= \Delta(T).
    \end{align}
\end{theorem}
Combining the theorems and taking $\Tgeneric$ as in \cref{theorem:tensor moment polytopes franz}, we obtain three distinct descriptions of the moment polytope of a tensor $T$:
\begin{align}
    \Delta^\Dgeom(T) = \Delta^\Drepr(T) = \DeltaB^\Dsupp(\Tgeneric).
\end{align}

Observe also that because $\C^a \ot \C^b \ot \C^c$ is a polynomial representation of $\G$,
any polynomial $f$ on $\C^a \ot \C^b \ot \C^c$ induces a polynomial $(A,B,C) \mapsto f((A\ot B \ot C)T)$ on $\G$.
It follows that any generic property on $\G_{\uppertriangular} \cdot T$ induces a generic property on $\G_{\uppertriangular}$.
This implies the following corollary.

\begin{corollary}
\label{corollary:moment polytope for generic group elements}
For generic $(A,B,C) \in \G$
we have that $\Delta(T) = \DeltaB^\Dsupp\big((A \ot B \ot C) T\big)$.
The same is true for generic $(A,B,C) \in \G_{\uppertriangular}$.
\end{corollary}

\subsection{Borel polytopes}
\label{subsection:tensor generic element}

In this section we make precise the generic tensors $\Tgeneric \in \G \cdot T$
from \cref{theorem:tensor moment polytopes franz} and the generic group elements of \cref{corollary:moment polytope for generic group elements}.
First we provide the geometric and representation-theoretic descriptions of the Borel polytope.

\subsubsection{Geometric description}

Recall the moment map $\mu$ defined in \cref{equation:tensor moment map}.
Let $(U_1, U_2, U_3) \in \K \subseteq \G$ be unitaries which diagonalize (in a sorted manner) the three Hermitian marginal matrices in $\mu(T)$.
One can verify that
$
    \mu \big( (U_1 \ot U_2 \ot U_3)T \big)
    = \spec\bigl(\mu(T)\bigr).
$
We may embed $\weylchamber$ (\cref{equation:tensor weyl chamber}) into $\Herm_a \times \Herm_b \times \Herm_c$ by mapping to triples of diagonal matrices.
Leaving this embedding implicit, the above implies
\begin{align}
    \label{equation:polytope spec}
    \Delta^\Dgeom(T) &= \Big\{\ \mu\big(T'\big) \ \ \big|\ \ T' \in \overline{\G \cdot T} \setminus\{0\}\, \Big\} \ \cap\  \weylchamber.
\end{align}
By the QL-decomposition,\footnote{This is the QR decomposition but using lower triangular matrices instead. The existence is shown in the same way, using the Gram--Schmidt process on the columns of the matrix, but instead starting with the right column and moving left. Alternatively, letting $F$ be the matrix with $1$s on the anti-diagonal and~$0$s everywhere else, then a QR-decomposition $A F = QR$ to a QL-decomposition $A = (QF)(FRF)$ for~$A \in \G$.}
we have $\G = \K \cdot \G_{\lowertriangular}$.
The geometric Borel polytope is defined in the same way, but disallowing the unitaries that diagonalize the marginals.
It is given by
\begin{align}
    \label{equation:polytope spec borel}
    \DeltaB^\Dgeom(T)
    &\coloneqq
    \Big\{\ \mu\big(T'\big) \ \ \big|\ \ T' \in \overline{\G_{\lowertriangular} \cdot T} \setminus\{0\} \, \Big\}
    \ \cap\ \weylchamber.
\end{align}

\subsubsection{Representation-theoretic description}
\label{subsubsection:borel polytope representation theoretic}

Any $\G$-representation has a weight decomposition: a decomposition into one-dimensional irreducible representations of the subgroup of (triples of) diagonal matrices \cite{fulton1991representation, etingof2011introductionrepresentationtheory, procesi2007}.
An element of one of these subspaces is called a \emph{weight vector}, and its \emph{weight} is some element of $\Z^a\times\Z^b\times\Z^c$.
Each irreducible representation of type $(\lambda,\mu,\nu)\in\weylchamber_\Z$ contains an unique-up-to-scaling vector of weight $(\lambda,\mu,\nu)$, which is called a \emph{highest weight vector}.

For $V = \C^a\ot\C^b\ot\C^c$ the weight vectors are given by the standard basis elements $e_i \ot e_j \ot e_k$ (and the set of associated weights is given by $\Omega$ as defined in \cref{equation:tensor weights}).

Now consider the $\G$-representation $W \coloneqq V^{\ot n}$.
Denote with $\HWV_{\lambda,\mu,\nu}$ the projection onto the subspace of highest weight vectors of weight $(\lambda,\mu,\nu)$.
The representation-theoretic description of the Borel polytope is then given by
\begin{align}
    \DeltaB^\Drepr(T)
    \coloneqq
    \overline{
    \bigg\{\ \Big(\frac{\lambda}{n}, \frac{\mu}{n}, \frac{\nu}{n}\Big)
    \quad \Big| \quad
    n > 0 \ \text{ and }\
    \HWV_{\lambda,\mu,\nu} T^{\ot n} \neq 0\ \bigg\}
    }.
\end{align}

\subsubsection{Equality of the characterizations}

We have the following theorem analogous to \cref{theorem:tensor moment polytopes nm}.

\begin{proposition}[Borel polytopes {\cite{franz2002,guillemin2005convexityProperties,guillemin2006convexityBorel,burgisser2018tensorScaling}}]
\label{theorem:tensor borel polytopes}
    Let $T$ be a tensor. Then
    \begin{align}
        \DeltaB^\Dgeom(T)
        = \DeltaB^\Drepr(T)
        = \DeltaB^\Dsupp(T).
    \end{align}
    We call it the \emph{Borel polytope} of $T$ and denote it with $\DeltaB(T)$.
    It is a rational polytope.
\end{proposition}
The second equality was essentially first shown in \cite{franz2002}, and can be obtained almost directly from Franz' proof of his characterization.
The equality of $\DeltaB^\Dgeom(T)$ and $\DeltaB^\Dsupp(T)$ was proven independently and via different techniques in \cite{guillemin2005convexityProperties}\footnote{To be precise, they write in \cite[Theorem~2.3]{guillemin2005convexityProperties} instead of $\conv \supp T'$ the moment polytope corresponding to the action of diagonal matrices on $T'$, which was proven by Atiyah to be a convex polytope \cite{atiyah1982convexity}. It is not hard to show these two polytopes are equal \cite{burgisserTheoryNoncommutativeOptimization2019}.}, with the equality of all three polytopes being observed more explicitly in \cite{guillemin2006convexityBorel}.
They also observed the first equality follows essentially directly from the proof of \cref{theorem:tensor moment polytopes nm} as given by Brion in \cite{brion1987momentMapImage}.
We observed these equalities independently of \cite{guillemin2006convexityBorel} and moreover realized the connection with Franz' proof. In particular, no analog of \cref{theorem:tensor moment polytopes franz} is given in \cite{guillemin2006convexityBorel} (but we note this connection is likely well-known to experts).
An alternative proof of the first equality is given in \cite[Prop.~2.11 and Prop.~2.15]{burgisser2018tensorScaling}, where the term
\emph{Borel polytope} was coined as well, see also \cite{cole2018operatorScaling,franks2022minimallengthorbitclosure}.
See \cite{hiraiGradientDescentUnbounded2024} for an alternative treatment of $\DeltaB^\Dsupp(T)$ and $\Delta^\Dgeom(T)$ from a convex optimization perspective.

\subsubsection{Genericity conditions}

The requirements on $\Tgeneric \in \G \cdot T$ such that $\DeltaB(\Tgeneric) = \Delta(T)$ are obtained via the representation-theoretic description.
Denote with $\<\cdot,\cdot>$ the standard Euclidean inner product on $(\C^a\ot\C^b\ot\C^c)^{\ot n}$.

\begin{lemma}[{\cite{franz2002,burgisser2018tensorScaling}}]
\label{theorem:tensor randomization}
    Let $T \in \C^a\ot\C^b\ot\C^c$.
    For any $T' \in \GL \cdot T$ we have
    \begin{align}
        \DeltaB^\Drepr(T') \subseteq \Delta^\Drepr(T).
    \end{align}
    For generic $\Tgeneric \in \G \cdot T$ and for generic $\Tgeneric \in \G_{\uppertriangular} \cdot T$ we have
    \begin{align}
        \DeltaB^\Drepr(\Tgeneric) = \Delta^\Drepr(T).
    \end{align}
    More precisely, $\Delta^\Drepr(T)$ is a convex polytope and
    the generic properties are both defined by the non-vanishing of the degree-$n_i$ polynomials
    \begin{align}
        f_i \colon \C^a \ot \C^b \ot \C^c \to \C, \quad f_i\big(S\big) \coloneqq \big\langle w_i, S^{\ot n_i} \big\rangle,
    \end{align}
    where $w_i \in (\C^a\ot\C^b\ot\C^c)^{\ot n_i}$ is a highest weight vector of weight $(\lambda_i,\mu_i,\nu_i)$ and $q_i = \big(\frac{\lambda_i}{n_i}, \frac{\mu_i}{n_i},\frac{\nu_i}{n_i}\big)$ ranges over the vertices of $\Delta^\Drepr(T)$.
\end{lemma}

We do not know what these polynomials are, but in computing~$\Delta(T)$ we learn their normalized weights: these correspond to the vertices of the polytope.
We discuss this in more detail in \cref{subsection:tensor attainability}.

\subsection{Restriction and degeneration}
\label{section:moment polytopes basic properties}
We go over some basic facts about moment polytopes in relation to restriction and degeneration.
We say a tensor $T \in \C^a\ot\C^b\ot\C^c$ \emph{restricts} to a tensor $S \in \smash{\C^{\overline{a}} \ot \C^{{\overline{b}}} \ot \C^{{\overline{c}}}}$ whenever there exist triples of matrices $(A,B,C) \in \smash{\C^{\overline{a} \times a} \times \C^{{\overline{b}} \times b} \times \C^{{\overline{c}} \times c}}$ (not necessarily invertible) such that $(A \ot B \ot C)T = S$. We then write $T \geq S$.
When $T \geq S$ and $S \geq T$, we say $T$ and $S$ are equivalent and write $T \sim S$.
Whenever $S \in \C^a\ot\C^b\ot\C^c$, we have that $T \sim S$ if and only if $S$ lies in the $\G$-orbit of $T$.
We say $T$ \emph{degenerates} to $S$ if $\overline{\G \cdot T}$ contains a tensor equivalent to $S$ and write $T \degengeq S$.
Restriction implies degeneration, and both are transitive relations.
We say a tensor is \emph{concise} whenever it cannot be embedded into a smaller space.
That is, $T$ is concise when it is not equivalent to a tensor $S \in \smash{\C^{\overline{a}}\ot\C^{\overline{b}}\ot\C^{\overline{c}}}$ with $\overline{a} < a$ or $\overline{b} < b$ or $\overline{c} < c$.

Moment polytopes behave well under equivalence of tensors.
Denote with $0_{\smash{\hspace{0.10em}{\overline{a}},{\overline{b}},{\overline{c}}}} \in \C^{\overline{a}} \ot \C^{\overline{b}} \ot \C^{\overline{c}}$ the zero tensor.
Given $p = (p_1,p_2,p_3) \in \Delta(T)$, we denote the padding of $p$ with zeros by
\begin{equation}
    (p_1,p_2,p_3) \oplus (0_{\hspace{0.05em}\overline{a}} \sep 0_{\hspace{0.05em}\smash{\overline{b}}} \sep 0_{\hspace{0.05em}\overline{c}}) \coloneqq (p_1 \oplus 0_{\hspace{0.05em}\overline{a}} \sep p_2 \oplus 0_{\hspace{0.05em}\smash{\overline{b}}} \sep p_3 \oplus 0_{\hspace{0.05em}\overline{c}}) \in \R^{a+{\overline{a}}}\times\R^{b+{\overline{b}}}\times\R^{c+{\overline{c}}}.
\end{equation}

\begin{lemma}[Padding with zeros]
\label{lemma:padding with zeros}
$\Delta(T \oplus 0_{\smash{\hspace{0.10em}\overline{a},{\overline{b}},{\overline{c}}}}) = \Delta(T) \oplus (0_{\hspace{0.05em}\overline{a}} \sep 0_{\hspace{0.05em}\smash{\overline{b}}} \sep 0_{\hspace{0.05em}\overline{c}})$.
\end{lemma}

\begin{proposition}[Equivalence]
    \label{proposition:moment polytope embedding}
    Let $T$ and $S$ be equivalent tensors.
    Then $\Delta(T)$ and $\Delta(S)$ are equal up to padding with zeros.
\end{proposition}

\begin{remark}
\label{remark:tensor moment polytope embedding}
By \cref{proposition:moment polytope embedding}, we can write statements such as $\Delta(T) = \Delta(S)$, $\Delta(T) \supseteq \Delta(S)$ and $p \in \Delta(T)$ even when $T$ and $S$ are tensors of different dimensions and $p$ is a vector not in $\R^a \times \R^b \times \R^c$.
We understand the statements to concern the polytopes and points after appropriate padding or removal of zeros on each of the three systems, such that both live in the same space.

\end{remark}

Secondly, degeneration implies inclusion of moment polytopes.

\begin{proposition}[Degeneration]
\label{proposition:degeneration monotone}
If $T \degengeq S$ then $\Delta(T) \supseteq \Delta(S)$.
\end{proposition}

Moreover, almost all restrictions of a tensor $T$ to a tensor $S$ of an entry-wise smaller shape $({\overline{a}},{\overline{b}},{\overline{c}})$ will result in a moment polytope $\Delta(S)$ given by intersecting $\Delta(T)$ with $\R^{{\overline{a}}} \times \R^{{\overline{b}}} \times \R^{{\overline{c}}}$.

\begin{proposition}[{\cite[Corollary~3.7]{burgisser2018tensorScaling}}]
    \label{proposition:tensor moment polytope generic restriction}
    Let~$T \in \C^a \ot \C^b \ot \C^c$ and $({\overline{a}},{\overline{b}},{\overline{c}}) \leq (a,b,c)$ entry-wise.
    Denote with $P_{\hspace{0.10em}\overline{a}} \ot P_{\hspace{0.20em}\smash{\overline{b}}} \ot P_{\hspace{0.10em}\overline{c}}$ the restriction from $\C^a\ot\C^b\ot\C^c$ to $\C^{{\overline{a}}} \ot \C^{{\overline{b}}} \ot \C^{{\overline{c}}}$ that projects onto the first ${\overline{a}},{\overline{b}}$ and ${\overline{c}}$ coordinates in the respective system.
    Then for generic $(A,B,C) \in \G$,
    \[
        \Delta\big( (P_{\hspace{0.10em}\overline{a}} \ot P_{\hspace{0.20em}\smash{\overline{b}}} \ot P_{\hspace{0.10em}\overline{c}})(A \ot B \ot C)T \big) = \Delta(T) \cap \big(\R^{{\overline{a}}} \times \R^{{\overline{b}}} \times \R^{{\overline{c}}} \big).
    \]
    In particular, there exists $S \in \C^{\overline{a}} \ot \C^{\overline{b}} \ot \C^{\overline{c}}$ such that $T \geq S$ and $\Delta( S ) = \Delta(T) \cap (\R^{{\overline{a}}} \times \R^{{\overline{b}}} \times \R^{{\overline{c}}})$.
\end{proposition}

\subsection{Kronecker polytopes}
\label{section:kronecker polytopes}

We end this section by introducing a notion called the \emph{Kronecker polytope}.
It is the moment polytope corresponding to not just a single orbit but the entire space $\C^a\ot\C^b\ot\C^c$.
It can also be described in multiple ways, with the geometric and representation-theoretic ones given by
\begin{align}
    \label{equation:generic moment polytope}
    \kronpolgeom{a}{b}{c}
    &\coloneqq
    \Big\{\, \spec\bigl( \mu(T) \bigr)
    \ \ \Big|\ \
    T \in \C^a\ot\C^b\ot\C^c \setminus \{0\}
    \,\Big\}
\qquad\text{and}\\
    \kronpolrepr{a}{b}{c}
    &\coloneqq
    \overline{
    \bigg\{\ \Big(\frac{\lambda}{n}, \frac{\mu}{n}, \frac{\nu}{n}\Big)
    \quad \Big| \quad
    \ISO_{\lambda,\mu,\nu} T^{\ot n} \neq 0 \text{ for a } T \in \C^a\ot\C^b\ot\C^c \bigg\}
    }.
\end{align}

\begin{remark}
The name comes from the observation that, by Schur-Weyl duality, the subspace
$
\ISO_{\lambda,\mu,\nu} \{T^{\ot n} \mid  T \in \C^a\ot\C^b\ot\C^c\}$ is isomorphic to $V_\lambda \otimes V_\mu \otimes V_\nu \otimes ([\lambda] \otimes [\mu ]\otimes [\nu] ])^{S_n}$ as a $\GL \times S_n$ representation,
where $[ \lambda ]$ denote the Specht module of the symmetric group $S_n$ and $([\lambda] \otimes [\mu ]\otimes [\nu] ])^{S_n}$ is the $S_n$-invariant subspace under the diagonal action of $S_n$ on $[\lambda] \otimes [\mu ]\otimes [\nu]$. The dimension $c_{\lambda,\mu,\nu}$ of $([\lambda] \otimes [\mu ]\otimes [\nu] ])^{S_n}$ is known as a \emph{Kronecker coefficient}, and by the above, $\ISO_{\lambda,\mu,\nu} T^{\ot n} \neq 0$ for some $T$ if and only if $c_{\lambda,\mu,\nu} > 0$.
\end{remark}

\begin{proposition}[Kronecker polytope
\cite{christandlmitchison, klyachko2006quantum,christandl2007nonzeroKroneckerCoefficients}.]
    \label{proposition:generic moment polytope nm}
We have that
\begin{align}
\kronpolgeom{a}{b}{c}
=
\kronpolrepr{a}{b}{c}.
\end{align}
We call it the \emph{Kronecker polytope} of shape $(a,b,c)$ and denote it with $\kronpol{a}{b}{c}$. It is a rational polytope.
\end{proposition}

The proof can be obtained with techniques from symplectic geometry
\cite{nessStratificationNullCone1984} as in \cite{klyachkoQuantumMarginalProblem2004} (see also \cite{daftuar2005quantum})
or with probabilistic tools from quantum information theory \cite{KeylWerner2001} as in \cite{christandlmitchison, christandl2007nonzeroKroneckerCoefficients, Christandl2006Thesis}.

Note that $\Delta(T) \subseteq \kronpol{a}{b}{c}$ for all $T \in \C^a\ot\C^b\ot\C^c$.
Analogously to \cref{theorem:tensor randomization}, for generic tensors in $\C^a \ot \C^b \ot \C^c$, their moment polytope equals the Kronecker polytope as can be seen from~\cite{brion1987momentMapImage}.
\begin{proposition}[Kronecker polytope]
    \label{proposition:generic moment polytope}
For generic $T_{\textnormal{g}} \in \C^a\ot\C^b\ot\C^c$ we have
\begin{align}
    \kronpol{a}{b}{c}
    =
    \Delta(T_{\textnormal{g}}).
\end{align}
\end{proposition}

Therefore we also refer to $\kronpol{a}{b}{c}$ as the \emph{generic polytope} of shape $a \times b \times c$.
For this reason, the Kronecker polytope is also the maximal moment polytope.

There is much previous work on characterizations of the Kronecker polytope in various formats. We provide a brief overview. See also \cite[Chapter~3]{walter2014thesis} for a detailed account.
Various complete mathematical descriptions of the Kronecker polytope of arbitrary formats have been found using different techniques \cite{nessStratificationNullCone1984,berenstein2000coadjointOrbitsMomentPolytope,franz2002,ressayre2010generalizedEigenvalueProblem,ressayre2011generalizedEigenvalueProblemII,vergneInequalitiesMomentCone2017}.
They, among other techniques, have been used to determine explicit descriptions of the irredundant inequalities (i.e.\ as concrete lists of numbers) for the following formats: $3 \times 3 \times 3$ \cite{franz2002}, $2 \times 2 \times 4$ \cite{bravyi2004requirementsQuantumStates}, $2 \times 2 \times \cdots \times 2$ \cite{higuchi2003qubitReducedStates},
$2 \times 2 \times 2 \times 8$,$\ $
$2 \times 2 \times 2 \times 2 \times 16$,$\ $
$3 \times 3 \times 9$, $\ $
$2 \times n \times 2n$,$\ $
$2\times 2\times 3 \times 12$ \cite{klyachkoQuantumMarginalProblem2004} (based on \cite{berenstein2000coadjointOrbitsMomentPolytope} and the connection to Kronecker coefficients),
and $4 \times 4 \times 4$ \cite{vergneInequalitiesMomentCone2017} (using techniques related to those in \cite{ressayre2010generalizedEigenvalueProblem}).
\cite{buloisAlgorithmComputeKronecker2025a} determined those for $5 \times 5 \times 5$ (which appeared as a preprint after our paper \cite{vandenBerg2025mmPolytope} first appeared on arXiv).
For all results, the formats with system-wise lesser or equal dimensions can also be obtained, as can permutations of the formats.

\section{Algorithms for moment polytopes of tensors}
\label{section:algorithms for computing tensor moment polytopes}

In this section we introduce our algorithms for computing moment polytopes as well as Borel polytopes.
They are directly based on the description of moment polytopes due to Franz \cite{franz2002}, which describes the moment polytope in terms of supports, see \cref{subsection:definitions and overview}. For Borel polytopes we describe a deterministic algorithm (\cref{theorem:borel algorithm}). For moment polytopes we describe
\begin{enumerate}
\item a probabilistic algorithm, which outputs the correct answer with at least a bounded probability (\cref{theorem:probabilistic algorithm});
\item a deterministic algorithm, whose output is always correct (\cref{theorem:symbolic algorithm}); and
\item a randomized algorithm, whose output is always correct when successful but which may output \texttt{Failure} with at most a bounded desired probability (\cref{theorem:symbolic algorithm randomized}).
\end{enumerate}
In each case, we can boost the success probability as much as desired.

We begin with an outline of these algorithms in \cref{subsection:algorithm outline}, and introduce a condition central to our algorithms called \emph{attainability}.
In \cref{subsection:tensor attainability} we make this condition algorithmic.
The proofs of correctness of the algorithms are given in \cref{subsection:algorithm proofs}.
Some implementation details are left to \cref{section:implementation and optimization}.
In \cref{section:algorithms for 4x4x4} we build upon the algorithms presented here, in particular providing a practically-relevant heuristic and a verification algorithm.

\begin{remark}
As discussed at the beginning of \cref{section:tensor moment polytopes}, the theory vastly generalizes.
This also applies to our algorithm. We specialize to the tensor case for accessibility.
We can replace the group $\GL_a \times \GL_b \times \GL_c$ by any reductive algebraic group $G$,
and the representation $\C^a\ot\C^b\ot\C^c$ by any rational representation of $G$.
See also \cite{burgisserTheoryNoncommutativeOptimization2019}, which describes the generalization of the tensor scaling algorithm \cite{burgisser2018tensorScaling} in much the same spirit.
\end{remark}

\subsection{Outline of the algorithms}
\label{subsection:algorithm outline}

In this subsection we give a high-level description of our deterministic and randomized algorithms for computing moment polytopes.
The basis will be a deterministic algorithm for computing Borel polytopes, which we will describe first.
The proofs for correctness of the algorithms will be given in \cref{subsection:algorithm proofs}.
Let $T \in \C^a \ot \C^b \ot \C^c$.
Write $\Omega = \{(e_i,e_j,e_k)\}_{i,j,k} \subseteq \R^a \times \R^b \times \R^c$ for the set of weights.
Recall Franz' description of the Borel polytope (\cref{equation:franz-tensors}), which for convenience we repeat here:
\begin{align}
    \label{eq:franz borel polytope again}
    \DeltaB(T) = \bigcap_{(A,B,C) \in \G_{\lowertriangular}} \conv \supp((A \ot B \ot C) T) \cap \weylchamber.
\end{align}
Since there is only a finite set of possible supports $S \subseteq \Omega$, we know the intersection only has finitely many terms.
We can thus write
\begin{align}
    \label{eq:borel polytope finite terms}
    \DeltaB(T)
    =
    \conv S_1 \cap \cdots \cap \conv S_s\ \cap\ \weylchamber
\end{align}
for some $s \in \N$, where for each $i$, $S_i = \supp \bigl((A_i \ot B_i \ot C_i) T\bigr) \subseteq \Omega$ for some $(A_i, B_i, C_i) \in \G_{\lowertriangular}$. The following definition will be central in our algorithm.
\begin{definition}[Attainability]
    \label{definition:attainability}
    Let $T$ be a tensor and let $S \subseteq \Omega$.
    We call $S$ \emph{attainable for $T$} whenever there exist $(A, B, C) \in \G_{\lowertriangular}$
    such that $\supp((A \ot B \ot C) T) \subseteq S$.
\end{definition}
Given $S \subseteq \Omega$, one way to determine attainability is by symbolic methods, which we describe in \cref{subsection:tensor attainability}.
An immediate possible algorithm iterates over all $S \subseteq \Omega$ and determines the ones that are attainable, perhaps making use of the fact that attainability is preserved under taking supersets.
However, due to the sheer number of possible supports, this method becomes impractical quickly.

Instead, we will use the attainability condition to determine inequalities for $\Delta(T)$.
This is based on \cref{lemma:tensor attainability} below. First we introduce some terminology.
We identify inequalities $h$ with closed half-spaces $H_h \subseteq \R^a\times\R^b\times\R^c$. The set $H_h$ consists of the points that \emph{satisfy} the inequality $h$.
An inequality $h$ is \emph{valid} for a set $X$ if $X \subseteq H_h$.
A set of inequalities $\{h_1,\ldots,h_s\}$ \emph{defines} a polytope $P$ if $P = H_{h_1} \cap \cdots \cap H_{h_s}$. Any polytope has a defining set of inequalities.
Define the set of weights $\Omega_h \subseteq \Omega$ that satisfies an inequality $h$ as $\Omega_h \coloneqq H_h \cap \Omega$.

\begin{lemma}
\label{lemma:tensor attainability}
Let $T \in \C^a \ot \C^b \ot \C^c$ be a tensor and $\Omega$ the set of weights. %
Then there exists a finite set of inequalities $\ineqs$ such that for all $S \subseteq \Omega$, some subset $\ineqs_S \subseteq \ineqs$ defines $\conv S$.
Moreover, if $\ineqs$ is any such set and we let
$\ineqs_T \coloneqq \{h \in \ineqs \mid \Omega_h \textnormal{ is attainable for } T\}$, then
\begin{align}
\label{eq:weight matrix enumeration}
\DeltaB(T) = \bigcap_{h \in \ineqs_T} H_h\ \cap\ \weylchamber.
\end{align}
\end{lemma}
\begin{proof}
For any $S \subseteq \Omega$, choose a finite set of inequalities $\ineqs_S$ that defines $\conv S$.
Then take $\ineqs = \bigcup_{S \subseteq \Omega} \ineqs_S$.

By \cref{eq:borel polytope finite terms}, we have that $\DeltaB(T) = \bigcap_{h \in \ineqs_{S_1} \cup \cdots \cup \ineqs_{S_s}} H_h \cap \weylchamber$, where the~$S_i$ are the attainable supports for~$T$.
Take $h \in \ineqs_{S_i}$. The set $\Omega_h$ is attainable for $T$ because $S_i \subseteq \Omega_h$ and $S_i$ is attainable for $T$.
Hence $h \in \ineqs_T$.
This shows that $\DeltaB(T) \supseteq \bigcap_{h \in \ineqs_T} H_h \cap \weylchamber$.
Now take $h \in \ineqs_T$.
Then there exist $(A,B,C) \in \G_{\lowertriangular}$ such that $S \coloneqq \supp ((A \ot B \ot C) T) \subseteq \Omega_h$. This implies $\conv S$ is a term in \cref{eq:franz borel polytope again} and hence $\DeltaB(T) \subseteq \conv S$.
This shows that $\DeltaB(T) \subseteq \bigcap_{h \in \ineqs_T} H_h \cap \weylchamber$.
\end{proof}

A set of inequalities $\ineqs$ as in \cref{lemma:tensor attainability} could for instance be computed by iterating over all possible $S \subseteq \Omega$ and writing down defining inequalities.
In \cref{subsection:intro combinatorial description} we explain how to determine a minimal such set much faster, without having to iterate over all $2^{abc}$ subsets of $\Omega$.
In the rest of this section we assume we have access to a set $\ineqs$ as in \cref{lemma:tensor attainability}.

Our algorithm for computing Borel polytopes takes the set $\ineqs$ and determines the set $\ineqs_T$.
Determining attainability is done using \cref{algorithm:tensor checking-inequalities}, which we will discuss in \cref{subsection:tensor attainability}.

\begin{algorithmbreak}{Borel polytope computation.}
\label{algorithm:tensor algorithm borel}%
 \textbf{Input:} A tensor $T \in \C^a \ot \C^b \ot \C^c$. \\
 \textbf{Output:} A set of inequalities defining the Borel polytope $\DeltaB(T)$ of $T$.  \\
 \textbf{Algorithm:}
\begin{algorithmic}[1]
    \State Let $\ineqs$ be a set of inequalities as in \cref{lemma:tensor attainability}.
    \Statex
    \State Initialize $\ineqs_{T} \gets \varnothing$.
    \For{$h \in \ineqs$}
        \State Determine attainability of $\Omega_h$ for $T$ (\cref{algorithm:tensor checking-inequalities}).
        \If{attainable}
            \State Add $h$ to $\ineqs_{\Tgeneric}$.
        \EndIf
    \EndFor
    \Statex
    \State \Return The set $\ineqs_T$ along with a set of inequalities defining $\weylchamber$.
\end{algorithmic}
\end{algorithmbreak}

\begin{theorem}
    \label{theorem:borel algorithm}
    Let $T \in \C^a\ot\C^b\ot\C^c$ be a tensor.
    \Cref{algorithm:tensor algorithm borel} with input $T$ outputs a defining set of inequalities for $\DeltaB(T)$.
\end{theorem}

Now we turn our attention again to moment polytopes.
\Cref{theorem:tensor moment polytopes franz} states that for generic $\Tgeneric \in \G_{\uppertriangular} \cdot T$, we have that $\Delta(T) = \DeltaB(\Tgeneric)$.
Although the polynomials defining the generic property are unknown to us, we may, for instance, generate $\Tgeneric$ randomly.
An upper bound on the required randomness is given in the following theorem.

\begin{theorem}
    \label{theorem:probabilistic algorithm}
    Let $T \in \C^a\ot\C^b\ot\C^c$ be a tensor.
    Let $n \coloneqq a+b+c$, $m \coloneqq \max\{a,b,c\}$, $\ell \coloneqq n^{3n/2}3^{n^2-n}$ and $M \coloneqq 6 (\ell+1)^n ( \ell^n m )^{3 m^2}$.
    Generate $(A,B,C) \in \G_{\uppertriangular}$ by sampling the entries uniformly random from $\{1,\ldots,M\}$.
    Then \cref{algorithm:tensor algorithm borel} with input $(A \ot B \ot C)T$ outputs a defining set of inequalities for $\DeltaB((A \ot B \ot C)T) \subseteq \Delta(T)$.
    With probability at least $1/2$, equality holds.\footnote{For this theorem as well as the probabilistic statements below, the probability can be boosted from $1/2$ to any desired probability (less than 1). In this case, by running the algorithm multiple times and taking the output polytope which is the largest. Alternatively, the proofs we provide can be easily adjusted to obtain any desired probability as well.}
\end{theorem}

It is also possible to handle the generic property symbolically, which will lead to a deterministic algorithm.
Namely, in \cref{subsection:tensor attainability} we will describe the orbit $\G_{\uppertriangular} \cdot T$ symbolically.
This allows us to find an algorithm that, given $h \in \ineqs$, determines whether $\Omega_h$ is attainable for all $T_h \in U_h \cdot T$, where $U_h \subseteq \G_{\uppertriangular}$ is some non-empty Zariski-open subset depending on $h$.
Then a generic $\Tgeneric \in \G_{\uppertriangular} \cdot T$ will both satisfy $\Delta(T) = \DeltaB(\Tgeneric)$ and $\Tgeneric \in U_h$ for all $h \in \ineqs$, which follows from the fact that finitely many non-empty Zariski-open subsets always intersect in irreducible varieties.
We obtain the following algorithm for computing $\Delta(T)$.

\begin{algorithmbreak}{Deterministic moment polytope computation.}
\label{algorithm:tensor algorithm}%
 \textbf{Input:} A tensor $T \in \C^a \ot \C^b \ot \C^c$. \\
 \textbf{Output:} A set of inequalities defining the moment polytope $\Delta(T)$ of $T$. \\
 \textbf{Algorithm:}
\begin{algorithmic}[1]
    \State Let $\ineqs$ be a set of inequalities as in \cref{lemma:tensor attainability}.
    \Statex
    \State Initialize $\ineqs_{\Tgeneric} \gets \varnothing$.
    \For{$h \in \ineqs$}
        \State Determine attainability of $\Omega_h$ for all $T_h \in U_h \cdot T$ for non-empty Zariski-open $U_h \subseteq \G_{\uppertriangular}$ (\cref{algorithm:tensor checking-inequalities symbolic}).
        \If{attainable}
            \State Add $h$ to $\ineqs_{\Tgeneric}$.
        \EndIf
    \EndFor
    \Statex
    \State \Return The set $\ineqs_{\Tgeneric}$ along with a set of inequalities defining $\weylchamber$.
\end{algorithmic}
\end{algorithmbreak}

\begin{theorem}
    \label{theorem:symbolic algorithm}
    Let $T \in \C^a\ot\C^b\ot\C^c$ be a tensor.
    \Cref{algorithm:tensor algorithm} with input $T$ outputs a defining set of inequalities for $\Delta(T)$.
\end{theorem}

However, the symbolic approach of~\cref{algorithm:tensor algorithm} quickly becomes infeasible in practice.
Luckily, it is not necessary to run \cref{algorithm:tensor checking-inequalities symbolic} for all inequalities $h$ in the typically very large set $\ineqs$.
Namely, we can first generate $(A,B,C) \in \G_{\uppertriangular}$ and compute the Borel polytope $P \coloneqq \DeltaB((A \ot B \ot C)T)$ using \cref{algorithm:tensor algorithm borel}.
Then we know $P \subseteq \Delta(T)$ by \cref{theorem:tensor moment polytopes franz}.
If we can show that inequalities defining $P$ are also valid for $\Delta(T)$, we prove that $\Delta(T) \subseteq P$, and hence $\Delta(T) = P$.
Using this observation, we obtain a randomized algorithm that is always correct when successful and much faster than \cref{algorithm:tensor algorithm} in practice.

To reduce the number of inequalities to check, we make sure that the inequalities defining $P$ do not contain any redundancies; for a set of inequalities $\ineqs_P$ defining a polytope $P$, we call $h \in \ineqs_P$ \emph{redundant} if $P = \bigcap_{\ineqs_P \setminus \{h\}} H_h$. We say a subset of $\ineqs_P$ is \emph{irredundant} if it defines $P$ and contains no redundant inequalities.

\begin{algorithmbreak}{Randomized moment polytope computation.}
\label{algorithm:tensor algorithm randomized}%
\textbf{Input:} A tensor $T \in \C^a \ot \C^b \ot \C^c$ and randomization $(A,B,C) \in \G_{\uppertriangular}$. \\
 \textbf{Output:} A set of inequalities defining the moment polytope $\Delta(T)$ of $T$, or \textnormal{\texttt{Failure}}. \\
 \textbf{Algorithm:}
\begin{algorithmic}[1]
    \State Let $\ineqs$ be a set of inequalities as in \cref{lemma:tensor attainability}.
    \Statex
    \State Run \cref{algorithm:tensor algorithm borel} with input $T' \!=\! (A\ot B \ot C)T$ and let $\widehat{\ineqs}_{\Tgeneric}$ be the output, defining $\DeltaB(T')$.
    \State Replace $\widehat{\ineqs}_{\Tgeneric}$ by an irredundant subset. \label{line:randomized alg make irredundant}
  \State Remove from  $\widehat{\ineqs}_{\Tgeneric}$ the inequalities valid for $\weylchamber$.
    \Statex \label{line:determine attainability symbolically}
    \For{$h \in \widehat{\ineqs}_{\Tgeneric}$}
        \State Determine attainability of $\Omega_h$ for all $T_h \in U_h \cdot T$, with $U_h \subseteq \G_{\uppertriangular}$ non-empty Zariski-open (\cref{algorithm:tensor checking-inequalities symbolic}).
        \If{not attainable}
            \State \Return \textnormal{\texttt{Failure}}.
        \EndIf
    \EndFor
    \Statex
    \State \Return The set $\ineqs_{\Tgeneric}$ along with a set of inequalities defining $\weylchamber$.
\end{algorithmic}
\end{algorithmbreak}

\begin{theorem}
    \label{theorem:symbolic algorithm randomized}
    Let $T \in \C^a\ot\C^b\ot\C^c$ be a tensor and $(A,B,C) \in \G_{\uppertriangular}$.
    \Cref{algorithm:tensor algorithm randomized} with input $T$ and $(A,B,C)$ outputs a defining set of inequalities for $\Delta(T)$, or \textnormal{\texttt{Failure}}.
\end{theorem}
Note that when the output of \cref{algorithm:tensor algorithm randomized} is \texttt{Failure}, it may still be true that the Borel polytope of $(A \ot B \ot C)T$ equals $\Delta(T)$.
This can be the case because not every inequality in a defining set of inequalities needs to be attainable.
For instance, $\Delta((A \ot B \ot C)T)$ and $\Delta((A' \ot B' \ot C')T)$ may be equal while not having the same terms in \cref{eq:borel polytope finite terms}. Indeed, the choice of an irredundant set in \cref{line:randomized alg make irredundant} may be important.

The rest of this section is dedicated to explaining the algorithms for attainability and proving the above theorems.
Substantial effort is still required to translate these algorithms into a practical program.
We describe in \cref{section:implementation and optimization} the implementation details, as well as optimizations we apply to make running times tractable.

\begin{remark}
\label{remark:attainability general setting}
In the general setting of linear actions of reductive groups (as in \cref{subsection:intro-general-moment-polytopes}), $\Omega \subseteq \Z^n$ corresponds to the set of weights (written in integer coordinates) of the group with respect to some chosen fixed maximal torus.
Let $V$ be the representation of interest.
We write $\sum_{\omega \in \Omega} v_\omega$ for the corresponding weight decomposition of any vector $v \in V$, and $\supp v \coloneqq \{\omega \mid v_\omega \neq 0\}$.
We also fix a suitable Borel subgroup $B$ and write $B^-$ for the opposite Borel subgroup. Let $\weylchamber$ denote the positive Weyl chamber.
Then Franz' description of the Borel polytope reads
$\DeltaB(v) = \bigcap_{b \in B^-} \conv \supp(b \cdot v) \cap \weylchamber$ \cite{franz2002}.

Attainability of $S \subseteq \Omega$ for $v \in V$ is then defined as the existence of $b \in B^-$ such that $\supp(b \cdot v) \subseteq S$.
The above algorithms then generalize as written, replacing $T$ by $v$ and $\G_{\uppertriangular}$ by $B$. We discuss general attainability more in \cref{remark:general polynomial system}.
\end{remark}

\subsection{Determining attainability}
\label{subsection:tensor attainability}

We now fix an inequality $h$ and some tensor $T \in \C^a\ot\C^b\ot\C^c$.
We aim to determine attainability of $\Omega_h$ for $T$, as well as for generic $\Tgeneric \in \G \cdot T$ which we discuss later. That is, we want to verify
the existence of $(A,B,C) \in \G_{\lowertriangular}$ such that
\begin{align}
    \supp \big((A\ot B\ot C)T\big) \subseteq \Omega_{h} \coloneqq \{ \omega \in \Omega \mid \omega \in H_h\}.
\end{align}
Since the support is independent of scaling by diagonal matrices, we may assume the diagonal entries of $A,B$ and $C$ are all equal to one.
Write $T_\omega$ for the entry $T_{i,j,k}$ of $T$ corresponding to the weight $\omega = (e_i \sep e_j \sep e_k)$.
We can reformulate the above condition by requiring that $((A\ot B\ot C)T)_\omega = 0$ for all $\omega \in  \Omega \setminus \Omega_h$.

The entries $((A\ot B\ot C)T)_\omega$ are polynomials in the entries of $A, B$ and $C$ for fixed $T$.
Hence, attainability is equivalent to the existence of a common zero among the polynomials in the set
\begin{align}
\label{definition:tensor polysystem}
    \tensorpolysystem^T\!(h)
    \coloneqq \big\{ \big((\mathcal A\ot \mathcal B\ot \mathcal C)T\big){}_\omega \ \big|\  \omega \in \Omega \setminus \Omega_h \big\}
\end{align}
where $\mathcal A,\mathcal B$ and $\mathcal C$ denote symbolic lower triangular matrices with ones on the diagonal.

\begin{example}
Let $T = \sum_i e_i \ot e_i \ot e_i \in \C^3\ot\C^3\ot\C^3$ and $h = \scalebox{0.75}{$\left(
\begin{array}{rrr|rrr|rrr}
-11 &  -2 & 16 &  10 &  1 & -8 &  1 &  10 & -8
\end{array}
\right)$}$.
We write the symbolic lower triangular matrices $\mathcal A,\mathcal B$ and $\mathcal C$ using variables $x_i$, $y_i$ and $z_i$ respectively.
We depict resulting symbolic tensor $(\mathcal A \ot \mathcal B \ot \mathcal C)T$ below, denoted as a tuple of the slices along the first index.
Then the polynomial system $\tensorpolysystem^T\!(h)$ is given by the highlighted entries.
\begin{align}
\begin{gathered}
    \left(
    \left[\begin{smallmatrix}
        1 \\
        x_1 & 1 \\
        x_2 & x_3 & 1
    \end{smallmatrix}\right]
    \ot
    \left[\begin{smallmatrix}
        1 \\
        y_1 & 1 \\
        y_2 & y_3 & 1
    \end{smallmatrix}\right]
    \ot
    \left[\begin{smallmatrix}
        1 \\
        z_1 & 1 \\
        z_2 & z_3 & 1
    \end{smallmatrix}\right]
    \right)\, T
    =\\
    \left[
    \scalebox{0.8}{$
    \begin{array}{lll}
        1   & z_1    & \cellcolor{gray!20}z_2     \\
        \cellcolor{gray!20}y_1 & y_1z_2 & \cellcolor{gray!20}y_1z_2  \\
        \cellcolor{gray!20}y_2 & \cellcolor{gray!20}y_2z_1 & \cellcolor{gray!20}y_2z_2
    \end{array}$}%
    \,\,\right|\left.%
    \scalebox{0.8}{$
    \begin{array}{lll}
        x_1    & x_1z_1          & x_1z_2           \\
        x_1y_1 & x_1y_1z_2 + 1   & \cellcolor{gray!20}x_1y_1z_2 + z_3  \\
        \cellcolor{gray!20}x_1y_2 & x_1y_2z_1 + y_3 & \cellcolor{gray!20}x_1y_2z_2 + y_3z_3
    \end{array}$}%
    \,\,\right|\left.\!\!%
    \scalebox{0.8}{$
    \begin{array}{lll}
        x_2    & x_2z_1             & x_2z_2           \\
        x_2y_1 & x_2y_1z_2 + x_3    & x_2y_1z_2 + x_3z_3  \\
        x_2y_2 & x_2y_2z_1 + x_3y_3 & x_2y_2z_2 + x_3y_3z_3 + 1
    \end{array}$}
    \!\right]
    \end{gathered}
\end{align}
\end{example}

Finding a common zero in a system of polynomials is sometimes called the \emph{consistency problem} \cite[Chapter 1 \textsection 2]{CoxLittleOShea}.
This task can be achieved in multiple ways, most notably using Gr\"obner basis computations \cite{AdamsLoustaunauGroebner} and homotopy continuation
\cite{burgisser2013condition, bates2024numericalnonlinearalgebra}.
We will proceed with the former. We discuss why in more detail in \cref{remark:homotopy continuation}.

We note that the consistency problem is known to be a difficult problem in general. We comment more on this at the start of \cref{subsection:filtering inequalities}.

\begin{remark}
    \label{remark:general polynomial system}
    In the general setting (see \cref{remark:attainability general setting}), the lower triangular matrices with ones on the diagonal correspond to the (maximal) unipotent subgroup $N^-$ of the opposite Borel subgroup.
    Since the action of the maximal torus does not affect the support of $v$, we may also in the general case consider just the action of $N^-$.
    This action of $N^-$ on the representation is given by a polynomial map in the coordinates of the ambient space of the opposite Borel subgroup.
    Let $\mathcal B \in \C[x_1,\ldots,x_m]$ denote a symbolic element of this ambient space.
    Because $N^-$ is an algebraic group (i.e.\ a variety in the ambient space), there exist polynomials $\{f_i\}_i$ on $\C^m$ such that their common zero set defines $N^-$.

    For an inequality, we define
    $\polysystem$ as in \cref{definition:tensor polysystem}, but with the equations defining the unipotent subgroup added as well:
    \begin{align}
        \polysystem \coloneqq
        \big\{ (\mathcal B \cdot v)_\omega \mid \omega \in \Omega \setminus \Omega_h \big \} \cup \{f_i\}_i.
    \end{align}
    Then this is a polynomial system too, and the consistency problem corresponds correctly to the existence of an element in the unipotent subgroup that transforms $v$ to the desired support.
\end{remark}

\subsubsection{Gr\"obner bases}

Let $x_1,\ldots,x_m$ be the variables used in the symbolic lower triangular matrices $A,B$ and $C$ with ones on the diagonal.
Then the polynomial system (\cref{definition:tensor polysystem}) satisfies $\tensorpolysystem^T\!(h) \subseteq \C[x_1,\ldots,x_m]$.
To solve the consistency problem, we need to determine the existence of a $v \in \C^m$ such that $f(v) = 0$ for all $f \in \tensorpolysystem^T\!(h)$.
By Hilbert's weak Nullstellensatz (see \cite[Chapter~4]{CoxLittleOShea} for an introduction), this is equivalent to the ideal
\begin{align}
\label{eq:tensor polysystem ideal}
    \bigr\langle\tensorpolysystem^T\!(h)\bigr\rangle_{\C} \coloneqq \bigg\{\sum_{f\in\tensorpolysystem^T\!(h)} r_f f \quad \Big|\quad  r_f \in \C[x_1,\ldots,x_m]\ \bigg\}
\end{align}
in $\C[x_1,\ldots,x_m]$ not being equal to $\C[x_1,\ldots,x_m]$, or equivalently, $1 \notin \langle\tensorpolysystem^T\!(h)\rangle_\C$.

Containment of polynomials in an ideal $I$ can be solved by computing a Gr\"obner basis of $I$.
We follow the definition in \cite{AdamsLoustaunauGroebner}.

\begin{definition}[Gr\"obner basis]
\label{definition:groebner-basis}
    Let $\F$ be a field.
    Fix a total order $\leq$ on the subset of monomials in $\F[x_1,\ldots,x_m]$,
    such that for all monomials $q,p,r$ it holds that $q \leq p \Leftrightarrow qr \leq pr$ and
    $q \leq qr$.

    Let $I \subseteq \F[x_1,\ldots,x_m]$ be an ideal.
    A finite set of non-zero polynomials $G \subseteq I$ is called a \emph{Gr\"obner basis for $I$ over $\mathbb F$ with respect to $\leq$} if and only if $G$ generates $I$ and for all non-zero $f \in I$
    there exists $g \in G$ such that the leading term of $g$ (with respect to $\leq$) divides the leading term of $f$.

    A Gr\"obner basis is \emph{reduced} whenever each $g \in G$ has leading coefficient equal to 1, and no term in $g$ is divisible by the leading term of $g'$ for any $g' \in G\setminus\{g\}$.
\end{definition}

The first algorithm for computing (reduced) Gr\"obner bases was described by Buchberger \cite{BuchbergerThesis}.
Other algorithms include Faug\`ere's $\textnormal F_4$ and $\textnormal F_5$ algorithms \cite{FaugereF4,FaugereF5}.
We will not describe algorithms for computing Gr\"obner bases here, and instead refer to \cite{CoxLittleOShea,AdamsLoustaunauGroebner}.
Every ideal has a (reduced) Gr\"obner basis (by the existence of these algorithms).
Reduced Gr\"obner bases are unique for a fixed monomial order.

For our purposes, the following straightforward result is central.
\begin{proposition}
\label{proposition:reduced groebner basis trivial ideal}
Let $G$ be a reduced Gr\"obner basis for an ideal $I$. Then $1 \in I \Leftrightarrow G = \{1\}$.
\end{proposition}
\begin{proof}
    The direction from right to left follows from the fact that $G$ is a basis for $I$.
    For the other direction,
    simply note that the only polynomials that divide the leading term of $1$ are the constant polynomials $c \in \F \setminus \{0\}$, so $1 \in I$ implies $1 \in G$. Since $1$ divides all monomials, $G$ cannot contain any other polynomials, and thus $G = \{1\}$.
\end{proof}

We will soon require a more detailed understanding of Gr\"obner bases. In particular, we will use the characterization of Gr\"obner bases due to Buchberger \cite{BuchbergerThesis} in terms of so-called \emph{S-polynomials}, which we define now.
We follow the treatment in \cite{AdamsLoustaunauGroebner}.
Fix a monomial ordering.
Let $f,g \in \F[x_1,\ldots,x_r]$ be polynomials. Let their leading terms (with respect to the chosen monomial ordering) be $\alpha \prod_i x_i^{a_i}$ and $\beta \prod_i x_i^{b_i}$ respectively, where $\alpha,\beta \in \F$. Set $\gamma_i \coloneqq \max\{a_i,b_i\}$.
Then we define the \emph{$S$-polynomial} of $f$ and $g$ as
\begin{align}
    S(f,g) \coloneqq \frac{\prod_i x_i^{\gamma_i-a_i}}{\alpha} f \ +\ \frac{\prod_i x_i^{\gamma_i-b_i}}{\beta} g.
\end{align}

Next, we need the concept of a multivariate polynomial reduction.
Let $f,r,g \in \F[x_1,\ldots,x_m]$.
We say \emph{$f$ reduces to $r$ modulo $g$ in one step}, and write $f \xrightarrow{g} r$, whenever the leading term of $g$ divides a term of $f$ and
$r = f - X g$,
where $X$ equals this term of $f$ divided by the leading term of $g$.

Now let $G = \{g_1,\ldots,g_t\} \subseteq \F[x_1,\ldots,x_m]$.
We say \emph{$f$ reduces to $r$ modulo $G$}, and write $f \xrightarrow{G} r$, whenever there exists a sequence of indices $k_1,\ldots,k_q \in \{1,\ldots,t\}$ and polynomials $r_i \in \F[x_1,\ldots,x_m]$ such that
\begin{align}
    f = r_1
    \ \xrightarrow{g_{k_1}}\
    r_2
    \ \xrightarrow{g_{k_2}}\
    r_3
    \ \xrightarrow{g_{k_3}}\
    \cdots
    \ \xrightarrow{g_{k_{q-1}}}\
    r_{q}
    \ \xrightarrow{g_{k_q}}\
    r.
\end{align}

Then the characterization of Gr\"obner basis using S-polynomials is as follows.

\begin{proposition}[Buchberger's criterion {\cite{BuchbergerThesis}}, as presented in
{\cite[Thm.~1.7.4]{AdamsLoustaunauGroebner}}]
\label{proposition:buchbergers criterion}
    Let $G = \{g_1,\ldots,g_t\}\subseteq \F[x_1,\ldots,x_r]$.
    Then $G$ is a Gr\"obner basis for the ideal generated by $G$ if and only if for all $i\neq j$ we have
    $S(g_i,g_j) \xrightarrow{G} 0$.%
\end{proposition}

Suppose we are given a set of polynomials $F \subseteq \F[x_1,\ldots,x_m]$.
An important corollary of the above is that Gr\"obner bases of the ideal generated by $F$ are invariant under field extensions $\F' \supseteq \F$.
Denote the ideal generated by $F$ over $\F$ as
$
    \<F>_{\F} \coloneqq \big\{\sum_{f\in F} r_f f \ |\  r_f \in \F[x_1,\ldots,x_m]\ \big\},
$
and similarly define $\<F>_{\smash{\F'}}$.
Because the conditions of \cref{proposition:buchbergers criterion} remain true under field extensions, we obtain the following corollary.

\begin{corollary}
    \label{corollary:groebner field extension}
    Let $\F \subseteq \F'$ be fields and $F \subseteq \F[x_1,\ldots,x_m]$.
    If $G$ is a (reduced) Gr\"obner basis for $\<F>_\F$ over $\F$, then $G$ is also a (reduced) Gr\"obner basis for $\<F>_{\F'}$ over $\F'$.
\end{corollary}

Whenever our tensor $T$ has rational entries we have $\tensorpolysystem^T\!(h) \subseteq \Q[x_1,\ldots,x_m]$.
\Cref{corollary:groebner field extension} implies that the Gr\"obner basis of the ideal $\langle\tensorpolysystem^T\!(h)\rangle_{\Q}$ over $\Q$ equals the Gr\"obner basis of $\langle\tensorpolysystem^T\!(h)\rangle_{\C}$ over $\C$.
This is important because Gr\"obner basis algorithms require exact arithmetic, which is not available to computers for uncountable fields such as~$\C$.
The tensors we are interested in are all specified using integer entries (cf.\ \cref{section:C333,example:mm2 derandomization}), and the above holds.
Whenever this is not the case we can instead consider the extended field $\F \supseteq \Q$ which extends $\Q$ algebraically and/or transcendentally by all entries of $T$. Because $T$ has finitely many entries, this field will be representable symbolically in computer algebra systems. We denote it by $\F = \Q\big(\{T_\omega \mid \omega \in \Omega\}\big)$. Note that $\Q(\{T_\omega \mid \omega \in \Omega\}) \subseteq \C$.

Finally, we consider the case where the field is a function field $\F(z_1,\ldots,z_s)$. This field is defined as the set of formal rational functions in the variables $z_1,\ldots,z_s$, that is,
\begin{align}
    \F(z_1,\ldots,z_s) \coloneqq
    \Bigl\{
        \ \frac{p}{q} \ \Big|\ p,q \in \F[z_1,\ldots,z_s],q \neq 0\,
    \Bigr\}.
\end{align}
This enables us to compute Gr\"obner bases for polynomial systems $F$ where the coefficients of the polynomials are dependent on parameters $z_1,\ldots,z_s$, i.e.\ $F \subseteq \F[z_1,\ldots,z_s][x_1,\ldots,x_m] \subseteq \F(z_1,\ldots,z_s)[x_1,\ldots,x_m]$.
This is precisely how we will in \cref{subsection:attainability algorithms} describe the generic element $\Tgeneric \in \G \cdot T$ symbolically.

The situation is subtle however, as the resulting (reduced) Gr\"obner basis will in general not remain a (reduced) Gr\"obner basis when plugging in explicit values for $z_1,\ldots,z_s$ from $\F$.
However, this is true for generic values.

\begin{proposition}
\label{proposition:groebner function field}
    Let $\F' = \F(z_1,\ldots,z_s)$ denote the function field on variables $z_1,\ldots,z_s$.
    Let $F,G \subseteq \F'[x_1,\ldots,x_m]$ be a finite sets, with $G$ the reduced Gr\"obner basis for $\<F>_{\F'}$ over $\F'$.
    Then there exists a non-empty Zariski-open subset $U \subseteq \F^s$ such that for all $u \in U$,
    \begin{enumerate}
        \item
        \label{item:generic groebner can substitute}
        we can in any $f \in F$ and $g \in G$ substitute $z_i$ by $u_i$ for all $i$ to obtain $f^{(u)}, g^{(u)} \in \F[x_1,\ldots,x_m]$, and
        \item
        \label{item:generic groebner is basis}
        $G^{(u)} \coloneqq \{g^{(u)} \mid g \in G\}$ is a reduced Gr\"obner basis for $\<f^{(u)} \mid f \in F>$ over $\F$.
        \item
        \label{item:generic groebner non trivial}
        If $G \neq \{1\}$ then also $G^{(u)} \neq \{1\}$.
    \end{enumerate}
\end{proposition}
\begin{proof}[Proof sketch]
    The crucial observation is that there is only a finite number of polynomials (over $\F'$) appearing in the reductions $S(g_i,g_j)\ \smash{\overset{G}{\to}}\ 0$ that characterize Gr\"obner bases $G$ via \cref{proposition:buchbergers criterion}, as well as in those polynomials witnessing $f \in \<G>\!_{\F'}$ and $g \in \<F>\!_{\F'}$ for each $g \in G, f \in F$.

    This means we can assign generic values $u_1,\ldots,u_s$ to the variables $z_1,\ldots,z_s$ such that we never divide by zero nor change leading terms in any of the polynomials mentioned above.
    We observe that with this choice, all reductions $S(g_i,g_j)\ \smash{\overset{G}{\to}}\ 0$ give rise to reductions $S(g_i^{(u)},g_j^{(u)})\ \smash{\overset{\smash{G^{(u)}}}{\to}}\ 0$,
    which implies that $\smash{G^{(u)}}$ is a Gr\"obner basis
    by \cref{proposition:buchbergers criterion}.
    Moreover, $\smash{G^{(u)}}$ remains reduced.

    Next, note that the equations witnessing $f \in \<G>\!_{\F'}$ remain valid after substitution and so $f^{(u)} \in \langle G^{(u)}\rangle_{\F}$ for each $f \in F$. Similarly for $g \in \<F>\!_{\F'}$.
    It follows that $G^{(u)}$ is a reduced Gr\"obner basis for $\smash{\langle f^{(u)} \mid f \in F\rangle}$.
    Finally, \cref{item:generic groebner non trivial} follows because leading terms remain the same.
\end{proof}

\subsubsection{Algorithms for attainability}
\label{subsection:attainability algorithms}

Given a tensor $T$ and inequality $h$ we can now use any algorithm for computing Gr\"obner bases to determine attainability of $\Omega_h$ for $T$.
As mentioned before, the entries of $T$ can always be taken to sit inside the field extension $\F = \Q\big(\{T_\omega \mid \omega \in \Omega\}\big)$, satisfying $\Q \subseteq \F \subseteq \C$.

\begin{algorithmbreak}{Determine attainability for a tensor $T$.}
\label{algorithm:tensor checking-inequalities}%
\textbf{Input:} A tensor $T \in \C^a\ot\C^b\ot\C^c$ and an inequality $h \in \R^{a+b+c}$. \\
\textbf{Output:} \texttt{True} or \texttt{False} \\
\textbf{Algorithm:}
\begin{algorithmic}[1]
    \State Determine the field extension $\F = \Q\big(\{T_\omega \mid \omega \in \Omega\}\big) \subseteq \C$, such that $T \in \F^a\ot\F^b\ot\F^c$.
    \State Create symbolic lower triangular matrices $\mathcal A,\mathcal B,\mathcal C$ with ones on the diagonal.
    \State Setup the polynomial system $\tensorpolysystem^T\!(h) \coloneqq \big\{ \big((\mathcal A\ot \mathcal B\ot \mathcal C)T\big){}_\omega \mid \omega \in \Omega \setminus \Omega_h \big\}$.
    \State Compute the reduced Gr\"obner basis $G$ of $\smash{\tensorpolysystem^T\!(h)}$ over $\F$.
    \State \Return \texttt{True} \textbf{if} $G \neq \{1\}$ \textbf{else} \texttt{False}
\end{algorithmic}
\end{algorithmbreak}

\begin{lemma}
\label{lemma:attainability}
    Let $T \in \C^a\ot\C^b\ot\C^c$ and $h \in \R^{a+b+c}$.
    If \Cref{algorithm:tensor checking-inequalities} with input $(T,h)$ outputs \textnormal{\texttt{True}}, then $\Omega_h$ is attainable for $T$. Otherwise, $\Omega_h$ is not attainable for $T$.
\end{lemma}
\begin{proof}
    By \cref{proposition:reduced groebner basis trivial ideal},
    \cref{algorithm:tensor checking-inequalities} outputs \texttt{True} if and only if  $1 \notin \langle \F^T(h)\rangle_\F$.
    By \cref{corollary:groebner field extension}, $1 \notin \langle \tensorpolysystem^T\!(h)\rangle_\F$ if and only if $1 \notin \langle\F^T(h)\rangle_\C$.
    By Hilbert's weak Nullstellensatz \cite[Thm~4.1.1]{CoxLittleOShea},  $1 \notin \langle\tensorpolysystem^T\!(h)\rangle\!_\C$ is equivalent to the existence of $v \in \C^m$ such that $f(v) = 0$ for all $f \in \tensorpolysystem^T\!(h)$.
    This is equivalent to the existence of $(A,B,C) \in \G_{\lowertriangular}$ such that $\supp((A\ot B\ot C)T) \subseteq \Omega_h$.
\end{proof}

We may alter this algorithm to determine instead attainability of $\Omega_h$ for all $T_h \in U_h \cdot T$,
where $U_h \subseteq \G_{\uppertriangular}$ is some non-empty Zariski-open subset.
We describe the procedure.
Denote with $\F(\G_{\uppertriangular})$ the rational function field of the coordinate functions of $\G_{\uppertriangular}$.
That is, we introduce formal variables $z_1,\ldots,z_s$ for each possibly non-zero entry of the upper triangular matrices in $\G$, and let $\F(\G_{\uppertriangular}) \coloneqq \F(z_1,\ldots,z_s)$ be the function field.

Given the tensor $T$, let $T_{\textnormal{s}}$ be the symbolic tensor describing the polynomial action of $\G_{\uppertriangular}$ on $T$ using the introduced variables. The entries of $T_{\textnormal{s}}$ are degree $\leq 3$ polynomials in $\F(\G_{\uppertriangular})$ using the variables $z_1,\ldots,z_s$ .
Then act with the symbolic matrices $(\mathcal A, \mathcal B, \mathcal C)$
describing the polynomial action of $\G_{\lowertriangular}$ with ones on the diagonal on $T_{\textnormal{s}}$, resulting in the tensor
$(\mathcal A \ot \mathcal B \ot \mathcal C) T_{\textnormal{s}}$.
The entries of this tensor are elements of $\F(\G_{\uppertriangular})[x_1, \dotsc, x_m]$, where $x_1,\ldots,x_m$ are the variables of $(\mathcal A, \mathcal B, \mathcal C)$.

Let $h$ be an inequality and $\tensorpolysystem^{T_{\textnormal{s}}}(h) \subseteq \F(\G_{\uppertriangular})[x_1, \dotsc, x_m]$ the corresponding polynomial system as defined in \cref{definition:tensor polysystem}.
Then the algorithm is as follows.

\begin{algorithmbreak}{Determine attainability for all $T_h \in U_h \cdot T$, with $\varnothing \neq U_h \subseteq \G_{\protect\uppertriangular}$ Zariski-open.}
\label{algorithm:tensor checking-inequalities symbolic}%
\textbf{Input:} A tensor $T \in \C^a\ot\C^b\ot\C^c$ and an inequality $h \in \R^{a+b+c}$. \\
\textbf{Output:} \texttt{True} or \texttt{False} \\
\textbf{Algorithm:}
\begin{algorithmic}[1]
    \State Determine the field extension $\F = \Q\big(\{T_\omega \mid \omega \in \Omega\}\big) \subseteq \C$, such that $T \in \F^a\ot\F^b\ot\F^c$.
    \State Let $T_{\textnormal{s}}$ describe the polynomial action of $\G_{\uppertriangular}$ on $T$.
    \State Create symbolic lower triangular matrices $\mathcal A,\mathcal B,\mathcal C$ with ones on the diagonal.
    \State Setup the polynomial system $\tensorpolysystem^{T_{\textnormal{s}}}(h) \coloneqq \big\{ \big((\mathcal A\ot \mathcal B\ot \mathcal C)T_{\textnormal{s}}\big){}_\omega \mid \langle \omega \in \Omega \setminus \Omega_h \big\}$.
    \State Compute the reduced Gr\"obner basis $G$ of $\smash{\tensorpolysystem^{T_{\textnormal{s}}}(h)}$ over $\F(\G_{\uppertriangular})$.
    \State \Return \texttt{True} \textbf{if} $G \neq \{1\}$ \textbf{else} \texttt{False}
\end{algorithmic}
\end{algorithmbreak}

\begin{lemma}
    \label{lemma:attainability symbolic}
    Let $T \in \C^a\ot\C^b\ot\C^c$ and $h \in \R^{a+b+c}$.
    If \cref{algorithm:tensor checking-inequalities symbolic} with input $(T,h)$ outputs \textnormal{\texttt{True}}, then there exists a non-empty Zariski-open set $U_h \subseteq \G_{\uppertriangular}$ such that $\Omega_h$ is attainable for all $T_h \in U_h \cdot T$.
    Otherwise, there exists a non-empty Zariski-open set $U_h \subseteq \G_{\uppertriangular}$ such that $\Omega_h$ is not attainable for all $T_h \in U_h \cdot T$.
\end{lemma}
\begin{proof}
    Let $G$ be as computed by the algorithm. Suppose that the output is \texttt{True}, i.e.\ $G \neq \{1\}$.
    The polynomial system $\tensorpolysystem^{T_{\textnormal{s}}}(h)$ consists of polynomials with coefficients in $\F(\G_{\uppertriangular}) = \F(z_1,\ldots,z_s)$.

    By \cref{proposition:groebner function field}, there exists a non-empty Zariski-open subset $\widehat{U}_h \subseteq \F^s$ such that for all $u \in \widehat{U}_h$ we can obtain the reduced Gr\"obner basis $G^{(u)}$ for $\{f^{(u)} \mid f \in \tensorpolysystem^{T_{\textnormal{s}}}(h)\}$ by substituting $u$ into the coefficients of the polynomials in $G$. Moreover, $G^{(u)} \neq \{1\}$.

    We identified $\F^s$ with the set of triples of upper triangular matrices.
    Hence the set $\widehat{U}_h$ gives rise to a non-empty Zariski-open subset $U_h$ of $\G_{\uppertriangular}$ after intersecting the (Zariski-open) set where the determinant does not vanish.

    Write $u_{A,B,C} \in \F^s$ for the element corresponding to $(A,B,C) \in U_h$.
    It follows that for $(A,B,C) \in U_h$, $G^{(u_{A,B,C})} \neq \{1\}$ is a Gr\"obner basis for $\{f^{(u_{A,B,C})} \mid f \in \tensorpolysystem^{T_{\textnormal{s}}}(h)\} = \tensorpolysystem^{(A,B,C) \cdot T}(h)$.
    Then it follows that $\Omega_h$ is attainable for $(A,B,C) \cdot T$ by the same argument as in the proof of \cref{lemma:attainability}.
    We conclude $\Omega_h$ is attainable for all $T_h \in U_h \cdot T$.
    The argument for output \texttt{False} is analogous.
\end{proof}

\subsection{Proofs of correctness}
\label{subsection:algorithm proofs}

We are now ready to prove the correctness of our algorithms.

\begin{proof}[Proof of correctness of the Borel polytope algorithm (\cref{theorem:borel algorithm})]
By \cref{lemma:attainability}, the attainability algorithm (\cref{algorithm:tensor checking-inequalities}) correctly determines attainability of $\Omega_h$ for $T$.
By \cref{lemma:tensor attainability}, the resulting set $\ineqs_T$ together with the inequalities for $\weylchamber$ defines $\DeltaB(T)$.
\end{proof}

To prove \cref{theorem:probabilistic algorithm} we can use a Schwartz--Zippel type argument, provided we know bounds on the degrees of the polynomials defining the generic condition of $\Tgeneric$. These polynomials are given in \cref{theorem:tensor randomization}, with the degree being equal to the numbers $n_i$ which scale the vertices of $\Delta(T)$ to integer values.
Bounds on $n_i$ are known, most notably those obtained via the degree bounds from \cite{derskenDegreeBounds2000}.
In \cite{burgisser2018tensorScaling, burgisserTheoryNoncommutativeOptimization2019}
these bounds were used to obtain upper bounds on $n_i$ and with that the probability one of the polynomials $f_i$ would not vanish.%
\footnote{These bounds were used in \cite[Theorem~1.13]{burgisser2018tensorScaling} and \cite[Theorem 7.24]{burgisserTheoryNoncommutativeOptimization2019} to determine the required integer randomness for tensor scaling algorithm to scale to points in $\Delta(T)$ with probability at least $1/2$ (see also \cref{subsection:tensor scaling}).
This randomness is in fact precisely the same randomness we require \cite[Section~2.5]{burgisser2018tensorScaling}, although we consider the entire polytope at once instead of single rational points.}
We can apply an union bound to obtain a (somewhat crude) bound on the randomness we require.

\begin{proof}[Proof for the probabilistic algorithm (\cref{theorem:probabilistic algorithm})]
By \cref{theorem:borel algorithm}, we know the output equals $\DeltaB((A \ot B \ot C)T)$.
It remains to show $\DeltaB((A \ot B \ot C)T) = \Delta(T)$ with probability at least 1/2.

Define again $n \coloneqq a+b+c$ and $m \coloneqq \max\{a,b,c\}$.
Let $p = (\lambda,\mu,\nu)/\ell' \in \Delta(T)$ be a rational point, with $(\lambda,\mu,\nu) \in \weylchamber$ having integer entries and $\ell' \in \N$.
Define $M' = C \cdot 3 ( 3 \ell' m)^{3m^2}$ for $C \in \N_{> 0}$.
Generate $(A,B,C) \in \G_{\uppertriangular}$ by taking the entries uniformly at random from $\{1,\ldots,M'\}$.
Then \cite[Corollary~2.7]{burgisser2018tensorScaling} states\footnote{The statement is given for $C = 2$, but extends easily to arbitrary $C$. This follows from the fact that their result is obtained via an application of the Schwartz--Zippel lemma to a polynomial arising from a highest weight vector of weight $(\lambda,\mu.\nu)$, like the ones defined in \cref{theorem:tensor randomization}.} that $p \in \DeltaB((A \ot B \ot C)T)$ with probability at least $1-1/C$.

Now take $p \in \Delta(T)$ equal to any vertex. Then $p$ is rational.
By \cite[Proposition~7.23]{burgisserTheoryNoncommutativeOptimization2019}\footnote{Although our setting is slightly different, their proof applies to our setting without modification.}, the size $\ell'$ of the smallest integer such that $p\ell'$ has integer entries is upper bounded by $\ell \coloneqq n^{3n/2} 3^{n^2 - n}$.
This means that if we define $M \coloneqq C \cdot 3 ( 3 \ell m)^{3m^2}$ and generate $(A,B,C) \in \G$ by taking the entries uniformly at random from $\{1,\ldots,M\}$, we have that $p \in \DeltaB((A \ot B \ot C)T)$ with probability at least $1-1/C$.

We finish the proof by using a very rough upper bound on the number of vertices of $\Delta(T)$.
Namely, there are at most $(\ell+1)^n$ rational points in $\Q^n$ such that each entry is of the form $x/\ell$ for some $x \in \{0,\ldots,\ell\}$. By the above, this set includes all possible vertices of $\Delta(T)$. Taking $C = 2(\ell+1)^n$, the union bound gives that
\begin{align}
    \mathbb P \big(\DeltaB((A \ot B \ot C)T) \neq \Delta(T) \big)
    &= \mathbb P \big( p \notin \DeltaB((A \ot B \ot C)T) \text{ for a vertex } p \text{ of } \Delta(T) \big)
    \\&\leq \sum_{p \in \Delta(T) \colon p \text{ is a vertex}} \mathbb P \big( p \notin \DeltaB((A \ot B \ot C)T) \big)
    \\&\leq (\ell+1)^n\cdot 1/C
    = 1/2.
\end{align}
So with probability at least $1/2$, indeed $\DeltaB\bigl((A \ot B \ot C)T\bigr) = \Delta(T)$.
\end{proof}

\begin{proof}[Proof for the deterministic algorithm (\cref{theorem:symbolic algorithm})]
    Let $\ineqs_{\Tgeneric}$ be as computed by the algorithm.
    We will use that $\G_{\uppertriangular}$ is an irreducible variety, which follows from the fact that it is an open subset of the set of triples of upper triangular matrices, which is an affine space and hence irreducible.

    For any $h \in \ineqs$ let $U_h$ be the open subset obtained from \cref{lemma:attainability symbolic} (regardless of whether the output was \texttt{True} or \texttt{False}).
    Let $V \subseteq \G_{\uppertriangular}$ be the non-empty Zariski-open subset such that $\DeltaB(\Tgeneric) = \Delta(T)$ for all $\Tgeneric \in V$.
    Then $V \cap \bigcap_{h \in \ineqs} U_h \neq \varnothing$, because $\G_{\uppertriangular}$ is irreducible.

    Let $\Tgeneric \in V \cap \bigcap_{h \in \ineqs} U_h$.
    Then $\ineqs_{\Tgeneric} = \{h \in \ineqs \mid \Omega_h \text{ is attainable for } \Tgeneric\}$.
    Hence by \cref{lemma:tensor attainability}, the output of the algorithm equals $\DeltaB(\Tgeneric)$.
    Since $\Tgeneric \in V$, we know $\DeltaB(\Tgeneric) = \Delta(T)$, as desired.
\end{proof}

\begin{proof}[Proof for the randomized algorithm (\cref{theorem:symbolic algorithm randomized})]
    Let $V \subseteq \G_{\uppertriangular}$ again be the non-empty Zariski-open subset such that $\DeltaB(\Tgeneric) = \Delta(T)$ for all $\Tgeneric \in V$.
    Let $\widehat{\ineqs}_{\Tgeneric}$ be as computed by the algorithm.
    By \cref{theorem:borel algorithm}, $\DeltaB((A \ot B \ot C) T) = \bigcap_{h \in \widehat{\ineqs}_{\Tgeneric}} H_h \cap \weylchamber$.
    By \cref{theorem:tensor borel polytopes}, $\DeltaB((A \ot B \ot C) T) \subseteq \Delta(T)$.

    If the output of the algorithm is not \texttt{Failure}, the set $\widehat{\ineqs}_{\Tgeneric}$ contains inequalities $h$ such that $\Omega_h$ is attainable for all $\Tgeneric \in V \cap \bigcap_{h \in \widehat{\ineqs}_{\Tgeneric}} U_h$, where $U_h$ is as in \cref{lemma:attainability symbolic}.
    This set is non-empty because $\G_{\uppertriangular}$ is an irreducible variety.
    It follows that $\Delta(T) = \DeltaB(\Tgeneric) \subseteq \bigcap_{h \in \widehat{\ineqs}_{\Tgeneric}} H_h \cap \weylchamber = \Delta((A \ot B \ot C)T)$.
    We conclude $\Delta(T) = \Delta((A \ot B \ot C)T)$.
\end{proof}

\begin{remark}
    After generalizing the algorithms to the general setting, see \cref{remark:attainability general setting,remark:general polynomial system}, it is straightforward to generalize the above correctness proofs as well.
    Here we use that the Borel subgroup $B$ (that replaces $\G_{\uppertriangular}$ in the algorithm) is always connected and therefore irreducible as a variety \cite{procesi2007,borel2012}.
    Only the precise randomization requirements will be different, but they are still obtainable from degree bounds of invariants \cite{derskenDegreeBounds2000} combined with an application of the Schwartz--Zippel lemma, see also \cite{burgisserTheoryNoncommutativeOptimization2019}.
\end{remark}

\section{Implementation and optimization}
\label{section:implementation and optimization}

In this section we discuss implementation details of and optimizations for the algorithms presented in \cref{section:algorithms for computing tensor moment polytopes}.
Most importantly, in \cref{subsection:intro combinatorial description} we make concrete how to effectively enumerate a minimal set of inequalities $\ineqs$ which is ``complete'', in that it satisfies the condition in \cref{lemma:tensor attainability}.
Afterwards, we discuss in \cref{subsection:filtering inequalities} multiple ways to reduce the number of inequalities $h \in \ineqs$ for which attainability of $\Omega_h$ needs to be checked.

In \cref{subsection:tensor implementation} we discuss further implementation details, such as the software we use, and we present some statistics (such as number of inequalities, number of attainability conditions to check, etc.) obtained from running our algorithms.

\subsection{Enumerating a complete set of inequalities}
\label{subsection:intro combinatorial description}

Given integers $a,b$ and $c$ we want to compute a minimal set of inequalities $\ineqs$ containing sufficiently many elements to describe the moment polytope of any tensor in $\C^a\ot\C^b\ot\C^c$, as in \cref{lemma:tensor attainability}.
We begin with some terminology and notation.
To allow for computation, we will represent inequalities concretely as vectors.
A \emph{(linear) inequality} on $\R^{n}$ is given by a tuple $(h,z) \in \R^{n} \times \R$.
Whenever $z = 0$, we call $(h,z)$ a \emph{homogeneous inequality}, and write is simply as $h$.
Associated to an inequality $(h,z)$ is the half-space $H_{(h,z)} \coloneqq \{p \in \R^{n} \mid \<p,h> \geq z\}$, where $\<\cdot,\cdot>$ denotes the standard inner product.
Note that for any $\alpha >0$, $H_{(h,z)} = H_{(\alpha h,\alpha z)}$.
Any $(h,z) \in \R^n \times \R$ can also be interpreted as an \emph{equality}, with associated hyperplane $\{p \in \R^{n} \mid \<p,h> = z\} = H_{(h,z)} \cap H_{(-h,-z)}$.
We denote the \emph{affine span} of a (finite) subset of $\R^n$ by $\aff \{\omega_1,\ldots,\omega_\ell\} \coloneqq \big\{\sum_i \beta_i \omega_i \mid \beta_i \in \R, \sum_i\beta_i = 1 \big\}$.
The dimension of this affine span is equal to the dimension of the linear subspace $\textnormal{span} \{\omega_1-\omega_1,\ldots,\omega_\ell-\omega_1\}$.
The dimension of $\conv \{\omega_1,\ldots,\omega_\ell\}$ is the same as the dimension $\aff \{\omega_1,\ldots,\omega_\ell\}$.

We are interested in polytopes in $\aff \Omega \subseteq \R^{a+b+c}$, where $\Omega = \{(\,e_i\sep e_j \sep e_k\,)\}_{i,j,k}$ denotes the set of weights.
Indeed, $\conv S \subseteq \aff \Omega$ for all $S \subseteq \Omega$, so $\DeltaB(T) \subseteq \aff \Omega$.
This affine subspace $\aff \Omega$ is defined by the equalities
\begin{align}
\label{equation:affine span equalities}
    (h_1,z_1) \coloneqq
    \bigl(
    (\,\bm1\sep \bm0 \sep \bm0\,)
    , 1\bigr)
    ,\
    (h_2,z_2) \coloneqq
    \bigl(
    (\,\bm0\sep \bm1 \sep \bm0\,)
    ,1  \bigr)
    ,\
    (h_3,z_3) \coloneqq
    \bigl(
    (\,\bm0\sep \bm0\sep \bm1\,)
    , 1\bigr)
    ,
\end{align}
where $\bm0$ and $\bm1$ denote the all-zeros and all-ones vectors of the appropriate dimensions respectively.
It follows that $H_{(h,z) + (\alpha h_i,\alpha z_i)} \cap \aff \Omega = H_{(h,z)} \cap \aff \Omega$ for any $\alpha \in \R$ and $i$.
In particular,
we can take $\alpha = -z$ and consider the equivalent (in $\aff \Omega$) homogeneous inequality $(h',0) = (h-zh_1,0) = (h,z) + (\alpha h_1,\alpha z_1)$.
Going forward, all inequalities will be taken to be homogeneous, and written as $h \in \R^{a+b+c}$ (this also explains our notation of inequalities in \cref{section:algorithms for computing tensor moment polytopes}).

A set of vectors $S$ is \emph{affinely independent} whenever $v \notin \aff (S\setminus\{v\})$ for all $v \in S$. This is equivalent to linear independence whenever $\aff S$ does not contain zero.
\begin{lemma}[Affine and linear independence]
\label{lemma:affine and linear independence}
    Suppose $S = \{\omega_1,\ldots,\omega_r\} \subseteq \R^\gdim$ satisfies $0 \notin \aff S$.
    Then $\omega_1,\ldots,\omega_r$ are affinely independent if and only they are linearly independent.
\end{lemma}
\begin{proof}
    Linear independence clearly implies affine independence.
    Suppose $S$ is linearly dependent, so $\sum_{i} \beta_i \omega_i = 0$ for some scalars $\beta_i$.
    We know $\sum_i \beta_i = 0$, as otherwise one could consider the affine combination $\sum_{i} \frac{\beta_i}{\sum_j \beta_j} \omega_i = 0$, which contradicts $0 \notin \aff S$.
    Without loss of generality assume that $\beta_1 \neq 0$, and scale the $\beta_i$ such that $\beta_1 = -1$. Then $\sum_{i>1}\beta_i = 1$ and $\sum_{i>1} \beta_i \omega_i = \omega_1$, showing affine dependence.
\end{proof}
We know $0 \notin \aff \Omega$, so affine and linear independence are equivalent for sets of weights.

\subsubsection{Characterizing inequalities by affinely independent weights}

Let $d \coloneqq \dim \aff \Omega = a+b+c-3$.
Recall that our goal is to find a set $\ineqs$ such that for each $S \subseteq \Omega$ there exists a finite subset $\ineqs_S \subseteq \ineqs$ satisfying $\conv S = \bigcap_{h \in \ineqs_S} H_h \cap \weylchamber$ (as in \cref{lemma:tensor attainability}).
By the following lemma, we may assume that each $h \in \ineqs_S$ is tight on $d$ affinely independent weights. That is,
$\{\omega \in \Omega \mid \<h,\omega> = 0\} = \Omega_h \cap \Omega_{-h}$ %
contains $d$ affinely independent elements.
Note that this is not hard to see whenever $\dim \conv S = d$. %
We also note that the result is likely well-known to experts in polytope theory (cf.\ \cite[Theorem~2.15]{ziegler2012lectures}).

\begin{lemma}[Defining inequalities]
    \label{proposition:defining-inequalities}
    Let $S \subseteq \Omega \subseteq \R^n$ be finite subsets with $0 \notin \aff\Omega$. Set $\pdim \coloneqq \dim \aff \Omega$.
    Then $\conv S$ is a finite intersection of half-spaces in $\aff \Omega$ each with boundary of the form $\aff\{\omega_1,\ldots,\omega_\pdim\}$, $\omega_i \in \Omega$.

    In other words,
    there exists a finite set $\ineqs_S \subseteq \R^{n}$
    such that
    \begin{align}
        \conv S = \bigcap_{h \in \ineqs_S} \big\{ p \in \aff \Omega \mid \<p, h> \geq 0\big\},
    \end{align}
    and for each $h \in \ineqs_S$ there exist
    affinely independent $\omega_1, \ldots, \omega_\pdim \in \Omega_h \cap \Omega_{-h}$. %

    Moreover, if $h' \in \ineqs_S$ satisfies
    $\Omega_h \cap \Omega_{-h} = \Omega_{h'} \cap \Omega_{-h'}$,
    then $h = \pm h'$.
\end{lemma}
\begin{proof}
 Set $k \coloneqq \dim \aff S$.
Select affinely independent $W \coloneqq \{\omega_0,\ldots, \omega_\pdim\} \subseteq \Omega$ such that
$\aff S = \aff \{\omega_0, \ldots, \omega_k\}$ and $\aff \Omega = \aff \{\omega_0,\ldots, \omega_\pdim\}$.
Then $P_i \coloneqq \aff\, (W \setminus \{\omega_{k+i}\})$ is a hyperplane in $\aff \Omega$ for all $i \geq 1$. Furthermore, $\omega_{k+i} \notin P_i$ by affine independence of $W$.
As a result the dimension of $\bigcap_{i=1}^{d-k} P_i$ is $k$, as each term in the intersection reduces the dimension $d$ of $\aff \Omega$ by one.
Since $S \subseteq P_i$ for all $i$, it follows that $\aff S = \smash{\bigcap_{i=1}^{d-k}} P_i$.
Now extend $P_i$ to a hyperplane $P_i'$ in $\R^\gdim$ by adding $\gdim-d$ points to $W \setminus \{\omega_i\}$ such that the result consists of affinely independent points.
We can choose these points such that $\aff \Omega \cap P_i' = P_i$. Because $0 \notin \aff \Omega$, we may also choose one of these points to be 0.
Now let $h_i \in \R^\gdim$ be a normal vector of $P_i'$.
We have that $P_i = \aff \Omega \cap \{p \in \R^\gdim \mid \<p,h_i> = 0\}$, and thus
 $\aff S = \bigcap_{i} \{p \in \aff \Omega \mid \<p,h_i> = 0\}$.

Next let $\{ \conv \{\nu_1^i, \ldots, \nu_{\ell_i}^i\} \}_i$ be the facets of $\conv S$, where $\nu_j^i \in S$ and~$\ell_i \geq k$ for all $i$.
Assume without loss of generality that~$\smash{\nu_1^i, \ldots, \nu_k^i}$ are affinely independent.
Then the set $\{\nu_1^i, \ldots, \nu_k^i,\omega_{k+1},\ldots,\omega_\pdim\}$ consists of $d$ affinely independent vectors in $\aff \Omega$, and hence its affine span is a hyperplane in $\aff \Omega$.
As before, extend $F_i$ to a hyperplane in $\R^\gdim$
going though 0.
Let $f_i$ be a corresponding normal vector pointing in the direction of $S$.
When restricting to $\aff S$, $f_i$ is a normal vector of the facet, so
$
    \conv S = \bigcap_i \big\{ p \in \aff S \mid \<p, f_i> \geq 0 \big\}.
$
Taking $\ineqs_S \coloneqq \{\pm h_i\}_i \cup \{f_i\}_i$ then finishes the proof.
\end{proof}

The inequality enumeration algorithm enumerates all inequalities that are tight on $d$ affinely independent weights.
Because $d = \dim \aff \Omega$,
this requirement defines uniquely two half-spaces $H_{h} \cap \aff \Omega$ and $H_{-h} \cap \aff \Omega$ in $\aff \Omega$.
Namely, we can take the inequality $h$ to be an orthogonal vector to $\omega_1,\ldots,\omega_d$ as well as to the vectors
\begin{align}
\label{equation:tensor trivial inequalities}
    (\,\bm1 \sep -\bm1 \sep \bm0\,),\quad
    (\,\bm0 \sep \bm1 \sep -\bm1\,)
    \in \R^{a+b+c},
\end{align}
which are two homogeneous equalities that are valid for all elements in $\aff \Omega$.
This imposes $(a+b+c-3)+2 = a+b+c-1$ linear conditions on $h \in \R^{a+b+c}$, and thus uniquely defines $h$ up to scaling by some $\alpha \in \R$.
Hence, we obtain two inequalities $h$ and $-h$.

It will be useful to formulate the above more sharply in terms of matrices.
We call a matrix $M \in \N^{r \times (a+b+c)}$ with row $i$ equal to $\omega_i \in \Omega$ a \emph{weight matrix}.
The weights $\omega_1,\ldots,\omega_r$ are linearly independent precisely when the rank of this matrix equals $r$.
As shorthand notation, we denote weight matrices as
\begin{align}
    \label{eq:weight matrices}
    M = (\omega_1;\ldots;\omega_r)
    \coloneqq
    \left(\ \begin{matrix}
        \omega_{1,1} & \ldots & \omega_{1,a+b+c} \\\hline
        \omega_{2,1} & \ldots & \omega_{2,a+b+c} \\\hline
            & \vdots & \\\hline
        \omega_{r,1} & \ldots & \omega_{r,a+b+c}
    \end{matrix}\ \right).
\end{align}
Now suppose $r = d$ and $\rank M = d$. Let $M_{\textnormal{e}}$ be the matrix obtain from $M$ by appending two rows equal to the vectors in \cref{equation:tensor trivial inequalities}.
These vectors are orthogonal to $\aff \Omega$ and to each other, so $\rank M_{\textnormal{e}} = d+2 = a+b+c-1$.
The kernel of $M_{\textnormal{e}}$ is then one-dimensional, spanned by some vector $h$.
The weight matrix $M$ therefore defines the two unique inequalities $h$ and $-h$.

Furthermore, since the matrix $M_{\textnormal{e}}$ only has integer entries, we know its kernel is spanned by a vector with integer entries.
We choose $h \in \Z^{a+b+c}$ such that the absolute values of its entries have greatest common divisor~$1$.
This eliminates the ambiguity resulting from the fact that inequalities are defined up to positive scaling.

\begin{example}
Let $a = b = c = 3$. The matrix $M_{\textnormal{e}}$ below has weights in its first $d = 6$ rows, and has rank $d+2=8$.
The element $h$ spans its kernel. This weight matrix uniquely defines the two inequalities $h$ and $-h$.
\begin{align*}
M_{\textnormal{e}} = &\scalebox{0.8}{$\left(
\begin{array}{rrr|rrr|rrr}
\phantom{-}1 & \phantom{-}0 & \phantom{-}0  &  1 & 0 & 0  &  1 & 0 & 0 \\
1 & 0 & 0  &  0 & 1 & 0  &  0 & 1 & 0 \\
0 & 1 & 0  &  1 & 0 & 0  &  0 & 0 & 1 \\
0 & 1 & 0  &  0 & 1 & 0  &  1 & 0 & 0 \\
0 & 1 & 0  &  0 & 0 & 1  &  0 & 1 & 0 \\
0 & 0 & 1  &  0 & 0 & 1  &  0 & 0 & 1 \\\hline
1 & 1 & 1  & -1 &-1 &-1  &  0 & 0 & 0 \\
0 & 0 & 0  &  1 & 1 & 1  & -1 &-1 &-1 \\
\end{array}\right)$}\\
h =
&\scalebox{0.8}{$\hspace{0.40em}
\left(
\begin{array}{rrr|rrr|rrr}
\mathrlap{-11}\phantom{-1} &  \mathrlap{-2}\phantom{-1} & \mathrlap{16}\phantom{-1} &
\mathrlap{10}\phantom{-1} &  \mathrlap{\phantom{-}1}\phantom{-1} & \mathrlap{-8}\phantom{-1} &
\mathrlap{\phantom{-}1}\phantom{-1} &  \mathrlap{10}\phantom{-1} & \mathrlap{-8}\phantom{-1}
\end{array}
\right)$}
\end{align*}
\end{example}

\subsubsection{Inequality enumeration algorithm}

We arrive at the following algorithm for enumerating the set of possible inequalities $\ineqs$.
Order the set $\Omega$ by some total order $\preccurlyeq$ (later we will use the reverse-lexicographic order).
Kernels are invariant under row permutations, so we can order the rows of the weight matrices with respect to this ordering.
We write $\mathcal M$ for the set of ordered weight matrices with $d$ linearly independent rows.

\begin{algorithmbreak}{Inequality enumeration.}
\label{algorithm:tensor weight-matrix-enumeration}%
\textbf{Input:} A set of weights $\Omega \subseteq \R^{a+b+c}$. \\
\textbf{Output:} A set of inequalities $\ineqs \subseteq \Z^{a+b+c}$. \\
\textbf{Algorithm:}
\begin{algorithmic}[1]
    \State Initialize $\mathcal M \gets \{ (\omega) \mid \omega \in \Omega\}$. \Comment{\textit{Enumerate weight matrices}}
    \State Let $d \coloneqq \dim \aff \Omega = a+b+c-3$.
    \For{$r = 2,\ldots,d$}
        \State Update $\mathcal M \gets \big\{ (\omega_{1};\ldots;\omega_{r-1};\omega_{r}) \mid (\omega_{1};\ldots;\omega_{r-1}) \in \mathcal M, \omega_{r-1} \preccurlyeq \omega_r\big\}$.
        \State Filter $\mathcal M \gets \big\{ M \in \mathcal M \mid \rank\,  M = r \big\}$.
    \EndFor
    \Statex
    \State Initialize $\ineqs \gets \varnothing$. \Comment{\textit{Determine inequalities}}
    \For{$M \in \mathcal M$}
        \State Construct $M_{\textnormal{e}}$ from $M$ by appending to $M$ the rows $(\,\bm1 \sep -\bm1 \sep \bm0\,)$ and $(\,\bm0\sep\bm1\sep-\bm1\,)$.
        \State Compute a non-zero $h \in \Z^{a+b+c}$ with $M_{\textnormal{e}} h = 0 $ using row reduction.
        \State Scale $h$ such that the absolute values of its entries are co-prime.
        \State Add the output $+h$ and $-h$ to $\ineqs$.
    \EndFor
    \Statex
    \State \Return $\ineqs$
\end{algorithmic}
\end{algorithmbreak}

\begin{lemma}
\label{lemma:weight enumeration algorithm}
Let $\ineqs$ be the output of \cref{algorithm:tensor weight-matrix-enumeration} with as input a set of weights $\Omega$.
Then for each $S \subseteq \Omega$, there exists $\ineqs_S \subseteq \ineqs$ such that $\conv S = \bigcap_{h \in \ineqs_S} H_h$.
\end{lemma}
\begin{proof}
Let $S \subseteq \Omega$.
Observe that $0 \notin \aff \Omega$ and
let $\ineqs_{S}$ be the set of inequalities obtained from \cref{proposition:defining-inequalities} applied to $S$.
We have to show that we may without loss of generality take $\ineqs_{S} \subseteq \ineqs$.

Let $q_1 \coloneqq (\,\bm1 \sep -\bm1 \sep \bm0\,) \in \R^{a+b+c}$ and
$q_2 \coloneqq (\,\bm0 \sep \bm1 \sep -\bm1\,)\in \R^{a+b+c}$.
We may assume each $h \in \ineqs_{S}$ satisfies $\<q_1,h> = \<q_2,h> = 0$,
which follows from the fact that
$h' = h + \beta_1 h_1 + \beta_2 h_2$ satisfies $\<p,h'> = \<p,h>$ for all $p \in \aff \Omega$ and $\beta_1,\beta_2 \in \R$.

Now take $h \in \ineqs_{S}$ and let $d = a+b+c-3$.
By \cref{proposition:defining-inequalities}, we know that $\<\omega_i,h> = 0$ for affinely independent $\{\omega_1,\cdots,\omega_d\} \subseteq \Omega$. Index the set in such a way that $\omega_1 \preccurlyeq \ldots \preccurlyeq \omega_d$.
By \cref{lemma:affine and linear independence}, $\{\omega_1,\ldots,\omega_d\}$ is linearly independent and hence
the weight matrix $(\omega_1;\ldots;\omega_r)$ (see \cref{eq:weight matrices} for notation) has rank $r$ for all $r \leq d$.
It follows that $(\omega_1;\ldots;\omega_r) \in \mathcal M$ with $\mathcal M$ as computed in \cref{algorithm:tensor weight-matrix-enumeration}.
Moreover, $(\omega_1;\ldots;\omega_d;h_0;h_1)\, h = 0$.

Since the entries of $(\omega_1;\ldots;\omega_r;h_0;h_1)$ are integer, we may scale $h$ such that it lies in $\Z^{a+b+c}$ and the absolute values of its entries are co-prime.
We can perform this scaling for all elements of $\ineqs_{S}$. This means we may indeed choose $\ineqs_{S}$ such that $\ineqs_{S} \subseteq \ineqs$.
\end{proof}

\begin{remark}
\label{remark:general enumeration}
We briefly sketch how to extend \cref{algorithm:tensor weight-matrix-enumeration} to the general case (see \cref{remark:attainability general setting}).
We might not have $0 \notin \aff \Omega$, and we cannot apply \cref{lemma:affine and linear independence}. Hence inequalities must be taken to be affine, that is, equal to elements $(h,z) \in \R^{a+b+c} \times \R$.
In \cref{algorithm:tensor weight-matrix-enumeration}, affine independence of $\{\omega_1,\ldots,\omega_r\}$ can be determined by adding to the weight matrix $M = (\omega_1;\cdots;\omega_r)$ a column with all entries equal to $-1$, and checking if the rank of the resulting matrix equals $r$.
The kernel entry corresponding to the added column then corresponds to the scalar $z$ of the inequality.

For uniqueness of $(h,z) \in \ineqs$, instead of adding the homogeneous equations given in \cref{equation:tensor trivial inequalities} as rows, one can append to $M$ as rows the $\codim \aff \Omega$ many affine equalities defining $\aff \Omega$ (e.g.\ the equalities from \cref{equation:affine span equalities}).
The rest of the algorithm remains the same.
\end{remark}

\subsubsection{Permutation symmetry reduction}

In \Cref{algorithm:tensor weight-matrix-enumeration} we compute integer elements in the kernel of a large number of matrices.%
\footnote{A rough upper bound is $\binom{|\Omega|}{d} = \binom{abc}{a+b+c-3}$.}
In practice, this is a computationally expensive task.
We can greatly reduce the number of such computations by exploiting the symmetries in the set of weights $\Omega$.
Recall that $\Omega = \Omega_a \times \Omega_b \times \Omega_c$ where $\Omega_a = \{e_1,\ldots,e_a\} \subseteq \R^a$.
Let $S_n$ denote the symmetric group on $n$ symbols.
Then $\Omega$ is closed under the action of any $\pi \in S_a \times S_b \times S_c$, where $\pi \cdot (e_i \sep e_j \sep e_k) = (e_{\pi_1(i)} \sep e_{\pi_2(j)} \sep e_{\pi_3(k)})$.
Define $\weylgroup \coloneqq  S_a \times S_b \times S_c$.\footnote{The notation comes from the general setting described in \cref{remark:general weyl group symmetries} below.}
This same group then acts on the set of weight matrices by permuting the columns.

Of course, permuting the columns will simply permute the entries of the kernel elements in the same way.
By the above, we would only need to compute the kernel element for a single weight matrix $M$ in its $\weylgroup$-orbit. Afterwards, we can find all kernel elements by permuting the kernel element according to $\weylgroup$.

Additionally, kernel elements are invariant under row permutations, as we exploited in \cref{algorithm:tensor weight-matrix-enumeration}.
Hence, we will define in \cref{definition:tensor weight-matrix-normal-form} a subset of \emph{standard} weight matrices, with the property that any weight matrix can be permuted---using any row permutations, and system column permutations from $\weylgroup$---to a standard weight matrix. Then we can reduce the search for kernel elements to just those of the standard weight matrices. The definition was chosen for efficiency of checking the condition.

\begin{definition}[Standard weight matrix]
    \label{definition:tensor weight-matrix-normal-form}
    Define the \emph{cost function}
    $\cost \colon \Omega \to \N^3$ by setting $\cost((e_i \sep e_j \sep e_k)) \coloneqq (i,j,k)$.
    Let $\preceq$ be the lexicographic order on
    $\N^{a}\times \N^b \times \N^c \supset \Omega$.
    We set a maximum over the empty set to equal 0.
    We call a weight matrix $(\omega_1;\ldots;\omega_r)$ \emph{standard} when
    \begin{enumerate}[label=(\arabic*)]
        \item $\omega_1 \succeq \omega_2 \succeq \cdots \succeq \omega_r$,
        \item \label{item:standard weight matrix definition cost}
        $\displaystyle \cost(\omega_\ell)_1 \leq 1 + \max_{1\leq \ell' < \ell} \cost(\omega_{\ell'})_1$ for all $\ell \in [r]$, and similarly for the other two systems.
    \end{enumerate}
\end{definition}
In plain language, a weight matrix is standard when its rows are ordered with respect to the reverse lexicographic ordering and the cost of every system at a row only increases by at most 1 beyond the maximum of the rows above.
The former fixes the ordering of the rows, and the latter removes the permutation freedom in every system.

\begin{example}
The following weight matrix for $\C^4 \ot \C^3 \ot \C^3$ is standard. The weight matrices of $\C^2 \ot \C^3 \ot \C^2$ to its right are not: the top weight matrix does not have lexicographically ordered rows, and the bottom one has a system cost increase of $2$ in the second subsystem.
\begin{equation*}
\scalebox{0.8}{$
\left(\begin{array}{cccc|ccc|ccc}
1 & 0 & 0 & 0 & 1 & 0 & 0 & 1 & 0 & 0 \\
0 & 1 & 0 & 0 & 0 & 1 & 0 & 0 & 1 & 0 \\
0 & 1 & 0 & 0 & 0 & 1 & 0 & 0 & 0 & 1 \\
0 & 0 & 1 & 0 & 0 & 0 & 1 & 1 & 0 & 0 \\
0 & 0 & 0 & 1 & 1 & 0 & 0 & 0 & 0 & 1 \\
0 & 0 & 0 & 1 & 0 & 1 & 0 & 1 & 0 & 0 \\
\end{array}\right)$}
\qquad\qquad
\scalebox{0.8}{$
\begin{array}{c}
\left(\begin{array}{cc|ccc|cc}
1 & 0 & 1 & 0 & 0 & 1 & 0 \\
0 & 1 & 0 & 1 & 0 & 0 & 1 \\
0 & 1 & 1 & 0 & 0 & 0 & 1 \\
\end{array}\right)
\\ \\
\left(\begin{array}{cc|ccc|cc}
1 & 0 & 1 & 0 & 0 & 1 & 0 \\
0 & 1 & 0 & 0 & 1 & 1 & 0 \\
\end{array}\right)
\end{array}$}
\end{equation*}
\end{example}

\begin{lemma}[Standard weight matrix]
\label{lemma:tensor standard-weight-matrix}
Every weight matrix $(\omega_1;\ldots;\omega_r)$ can be permuted into a standard weight matrix using row permutations, and column permutations from $\weylgroup = S_a \times S_b \times S_c$.
\end{lemma}
\begin{proof}%
We give an algorithm to permute $M = (\omega_1;\ldots;\omega_r)$ into a standard weight matrix.
First, permute the rows such that it satisfies the first condition.
Then iterate over the rows from top to bottom. We consider only the first system for now.
Suppose we come across a row $\omega_\ell = (e_i \sep e_j \sep e_k)$ that does not satisfy the second condition for the first system.
We then know the columns $m_\ell \coloneqq 1 + \max_{1\leq \ell' < \ell} \cost(\omega_{\ell'})_1$ and $i$ of $M$ contain only zeros in the first $\ell-1$ rows.
We swap these columns, after which $M$ satisfies the second condition for the first system up to row $\ell$.
The first $\ell-1$ rows stay the same, and this new row $\omega_\ell'$ satisfies $\omega_{\ell-1} \succeq \omega_{\ell}' \succeq \omega_\ell$. Hence the first condition is satisfied until row $\ell$.
Permute the rows below to satisfy the first condition, and continue iterating.

After going from top to bottom, we satisfy the first condition, and the second condition for the first system.
We iterate again from top to bottom, this time considering the second system.
We do exactly as before. This leaves the first system invariant. Namely, the column permutations do not affect the first system, and the row transpositions only occur when the rows are not anymore sorted.
Because of the lexicographic ordering, this can only happen when the rows are equal in the first subsystem.
We conclude that after reaching the bottom again, we satisfy the first condition and the second condition for the first two system. We do it once more for the third system, and we are done.
\end{proof}

Using this definition, we can update our inequality enumeration algorithm.
Order the weights $\{\omega^1, \ldots, \omega^\ell\} = \Omega$ according to $\preceq$.
Now let $\mathcal M$ be all standard weight matrices with $d$ affinely independent rows.
We can run the kernel computation procedure on all elements of $\mathcal M$.
Write $\ineqs^{\setminus \textnormal{perm}}$ for the set of outcomes.
Then we determine $\ineqs$ by letting the $\weylgroup$ permute all elements of $\ineqs^{\setminus \textnormal{perm}}$.

Since $\ineqs$ gets large and $\ineqs^{\setminus \textnormal{perm}}$ can contain many permutation symmetries, the following procedure is used to construct $\ineqs$. It is based on the fact that $S_a = (1,a)S_{a-1} \cup (2,a)S_{a-1} \cup \cdots \cup (a,a)S_{a-1}$.
It will be useful to indicate the system to which a permutation applies with a subscript. For example, $(i,j)_\ell$ with applies the transposition $(i,j)$ to the $i$-th and $j$-th columns among the columns corresponding to the $\ell$-th system.

\begin{algorithmbreak}{Inequality enumeration, using permutation symmetries.}
\label{algorithm:tensor weight-matrix-enumeration-no-weyl}%
\textbf{Input:} A set of weights $\Omega = \{\omega^1,\ldots,\omega^\ell\} \subseteq \Z^{a+b+c}$ ordered reversed-lexicographically. \\
\textbf{Output:} A set of inequalities $\ineqs \subseteq \Z^{a+b+c}$. \\
\textbf{Algorithm:}
\begin{algorithmic}[1]
    \State Initialize $\mathcal M \gets \{ (\omega^1),\ldots, (\omega^\ell)\}$. \Comment{\textit{Enumerate weight matrices}}
    \State Compute $d \coloneqq \dim \aff \Omega = a+b+c-3$.
    \For{$r = 2,\ldots,d$}
        \State Update $\mathcal M \gets \big\{ (\omega_{i_1};\ldots;\omega_{i_{r-1}};\omega_{i_r}) \mid (\omega_{i_1};\ldots;\omega_{i_{r-1}}) \in \mathcal M, i_{r-1} < i_r\big\}$.
        \State Filter $\mathcal M \gets \big\{ M \in \mathcal M \mid M \textnormal{ is standard}\big\}$.
        \State Filter $\mathcal M \gets \big\{ M \in \mathcal M \mid \rank\,  M = r \big\}$.
    \EndFor \label{line:weight matrix loop}
    \Statex
    \State Initialize $\ineqs^{\setminus \textnormal{perm}} \gets \varnothing$. \Comment{\textit{Determine inequalities}}
    \For{$M \in \mathcal M$}
        \State Construct $M_{\textnormal{e}}$ from $M$ by appending to $M$ the rows $(\,\bm1 \sep -\bm1 \sep \bm0\,)$ and $(\,\bm0\sep\bm1\sep-\bm1\,)$.
        \State Compute a non-zero $h \in \Z^{a+b+c}$ with $M_{\textnormal{e}} h = 0 $ using row reduction.
        \State Scale $h$ such that the absolute values of its entries are co-prime.
        \State Add the output $\pm h$ to $\ineqs^{\setminus \textnormal{perm}}$.
    \EndFor
    \Statex
    \State Initialize $\ineqs \gets \ineqs^{\setminus \textnormal{perm}}$ \Comment{\textit{Restore $\weylgroup$ symmetries}}
    \For{$m \in \{a,a-1,\ldots,1\}$}
        \State Update $\ineqs \gets (1,m)_1 \ineqs \cup (2,m)_1\ineqs \cup \ldots \cup (m,m)_1\ineqs$
    \EndFor
    \State Do the above for the second and third system as well.\label{line:symmetry restoration loop}
    \Statex
    \State \Return $\ineqs$
\end{algorithmic}
\end{algorithmbreak}

\begin{lemma}
\label{lemma:inequality enumeration weyl permutations}
Given input $\Omega = \{(\,e_i\sep e_j \sep e_k\,)\}_{i,j,k} \subseteq \Z^{a+b+c}$, the outputs $\ineqs$ of algorithm \cref{algorithm:tensor weight-matrix-enumeration} and $\ineqs'$ of \cref{algorithm:tensor weight-matrix-enumeration-no-weyl} are equal.
\end{lemma}
\begin{proof}
Take $h \in \ineqs$. Then $M_{\textnormal{e}} h = 0$ for some weight matrix $M$ with $d$ rows and of rank $d$.
By \cref{lemma:tensor standard-weight-matrix}, there exists permutations $\sigma \in S_d$ and $\pi \in S_a \times S_b \times S_c$ such that $M_{\sigma,\pi}' h_{\pi} = 0$, where $M_{\sigma,\pi}$ is equal to $M$ with the rows and columns permuted according to $\sigma$ and $\pi$ and where $h_\pi$ is equal to $h$ with the entries permuted according to $\pi$, such that $M_{\sigma,\pi}$ is standard.
Restricting $M_{\sigma,\pi}$ to its first $r$ rows results in a rank $r$ standard weight matrix.
Then it is not hard to see that $h_{\pi} \in \ineqs'$.

Next, we show that $\ineqs'$ is closed under permutation of entries according to $S_a \times S_b \times S_c$, which immediately implies $h \in \ineqs'$.
Namely, by the fact that $S_a = (1,a)S_{a-1} \cup (2,a)S_{a-1} \cup \cdots \cup (a,a)S_{a-1}$ have that $\ineqs^{\setminus\textnormal{perm}}$ (as computed by the algorithm before the last loop) satisfies $\ineqs' = (S_a \times S_b \times S_c)\ineqs^{\setminus\textnormal{perm}}$.
It follows that indeed, $(S_a \times S_b \times S_c) \ineqs' = \ineqs'$.

Finally, we note that $(S_a \times S_b \times S_c) \ineqs = \ineqs$, as permuting the entries of any $h \in \ineqs$ by $\pi \in S_a \times S_b \times S_c$ corresponds to permuting the columns of a corresponding weight matrix $M$ satisfying $M_{\textnormal{e}}h = 0$ by $\pi$ as well. This results in another weight matrix of the same rank, which we may order in reverse-lexicographic order. We conclude $h_\pi \in \ineqs$.

Because $\ineqs^{\setminus \textnormal{perm}} \subseteq \ineqs$, we conclude $\ineqs = \ineqs'$.
\end{proof}

\begin{remark}
\label{remark:general weyl group symmetries}
    In the general setting, the group $\weylgroup$ is equal to the Weyl group of the reductive group $G$ in question.
    Whenever the Weyl group is well-behaved, it is in principle possible to extend the optimization to any representation of $G$.
    Using a slightly generalized definition of standard weight matrices and a more involved proof of the analog of \cref{lemma:tensor standard-weight-matrix}, the above optimization can be applied for all the classical Lie groups and their products.

    For example, for products of $\GL_n$, say the familiar $G = \GL_a \times \GL_b \times \GL_c$, we may generalize \cref{definition:tensor weight-matrix-normal-form} by redefining the cost function $\cost \colon \Omega \to \N^3$ as follows.
    Let $\omega = (\omega^1,\omega^2,\omega^3) \in \Z^a \times\Z^b\times\Z^c$ be a weight. Then define $\cost(\omega)_i$ by
    \begin{align*}
        \cost(\omega)_i \coloneqq \min \underset{k}{\textnormal{argmax}}\ \omega^i_{k}
    \quad\text{ where }\ \underset{k}{\textnormal{argmax}}\ a_k \coloneqq \{k \mid a_k \geq a_{k'} \text{ for all } k'\}.
    \end{align*}
\end{remark}

\subsubsection{Removing system symmetries}

The next optimization also concerns inequality enumeration (\cref{algorithm:tensor weight-matrix-enumeration}).
Whenever $a = b = c$, $\ineqs$ is symmetric under permutation of the systems.
These are additional symmetries not captured by $\weylgroup$ as defined in the previous optimization.
We can update our definition of standard weight matrices with the following additional property.

\begin{definition}[Outer standard weight matrix]
Suppose $\Omega = \Omega_a \times \Omega_a \times \Omega_a$.
We call a standard weight matrix $M$ \emph{outer standard} when the sums of the first column of each system satisfy $M e_1 \geq M e_{1+a} \geq M e_{1+2a}$.
\end{definition}

\begin{lemma}[Outer standard weight matrix]
    \label{lemma:tensor outer-standard-weight-matrix}
    Suppose $\Omega = \Omega_a \times \Omega_a \times \Omega_a$.
    Every weight matrix can be permuted to a outer standard one via row permutations, column permutations from $\weylgroup$, and column permutations permuting the systems.
\end{lemma}
\begin{proof}
    Let the weight matrix be given by $M = (\omega_1; \ldots; \omega_d)$.
    Use column permutations from the Weyl group to ensure $\omega_1 = (e_1 \sep e_1 \sep e_1)$ (note that this is always true for standard weight matrices, as $c_1(\omega_1),c_2(\omega_1),c_3(\omega_1) \leq 1$ by \cref{definition:tensor weight-matrix-normal-form}\ref{item:standard weight matrix definition cost}%
    ).
    Then use the outer permutations to sort the systems such that the sums of the first columns in each system are decreasing: $M e_1 \geq M e_{1+a} \geq M e_{1+2a}$.
    These sums are invariant under row permutations and column permutations from $\weylgroup$ that fix the first columns.
    Now proceed with the procedure in the proof of \Cref{lemma:tensor standard-weight-matrix}.
    It is not hard to see we will never need to permute the first column of each subsystem. As a result, the sums above are left invariant, and we are left with a outer standard weight matrix.
\end{proof}

We update the weight matrix enumeration and symmetry restoration components in \cref{algorithm:tensor weight-matrix-enumeration-no-weyl}.
Denote with $\bm{(}\ell,\ell'\bm{)}$ the transposition of systems $\ell$ and~$\ell'$.
\begin{algorithmbreak}{Inequality enumeration, using permutation and system symmetries.}
\label{algorithm:tensor weight-matrix-enumeration-no-symmetries}%
\textbf{Input:} A set of weights $\Omega = \{\omega^1,\ldots,\omega^\ell\} \subseteq \Z^{a+b+c}$ ordered reversed-lexicographically. \\
\textbf{Output:} A set of inequalities $\ineqs \subseteq \Z^{a+b+c}$. \\
\textbf{Algorithm:}
\begin{algorithmic}[1]
    \Statex Run \cref{algorithm:tensor weight-matrix-enumeration-no-weyl} and output the results, but if $a = b = c$, add the following after \cref{line:weight matrix loop}:
    \Statex $\qquad$ {\footnotesize \phantom{1}\ref*{line:weight matrix loop}$'\phantom{'}$:} Filter $\mathcal M \gets \big\{ M \in \mathcal M \mid M \textnormal{ is outer standard}\big\}$.
    \Statex And the following after \cref{line:symmetry restoration loop}:
    \Statex $\qquad$ {\footnotesize \ref*{line:symmetry restoration loop}$'\phantom{'}$:} Update $\ineqs \gets \bm{(}1,3\bm{)}\ineqs \cup \bm{(}2,3\bm{)}\ineqs \cup \ineqs$
    \Statex $\qquad$ {\footnotesize \ref*{line:symmetry restoration loop}$''$:} Update $\ineqs \gets \bm{(}1,2\bm{)}\ineqs \cup \ineqs$
\end{algorithmic}
\end{algorithmbreak}

\begin{lemma}
\label{lemma:inequality enumeration weyl and system permutations}
Given input $\Omega = \{(\,e_i\sep e_j \sep e_k\,)\}_{i,j,k} \subseteq \Z^{a+b+c}$, the outputs of algorithm \cref{algorithm:tensor weight-matrix-enumeration} and \cref{algorithm:tensor weight-matrix-enumeration-no-symmetries} are equal.
\end{lemma}
The proof is essentially the same as the proof of \cref{lemma:inequality enumeration weyl permutations}, now combined with \cref{lemma:tensor outer-standard-weight-matrix}.

\subsection{Filtering inequalities}
\label{subsection:filtering inequalities}

For determining attainability (\cref{subsection:attainability algorithms}) we need to solve the consistency problem, which is a hard problem in general.
We briefly summarize what is known on the complexity of this problem.
We saw in \Cref{subsection:tensor attainability} that it is a special case of the ideal membership problem, asking whether $\sum_{f} f g_f = 1$, where the sum is over the elements of the polynomial system $\tensorpolysystem^T\!(h)$ (\cref{definition:tensor polysystem}) and $g_f$ are some polynomials.
It is known that the degrees of $g_f$ are at worst doubly exponential in the number of variables, and the decision problem variant (deciding whether such $g_f$ exist) is an exponential-space complete problem \cite{mayr1982complexityWordProblem, bayer1988complexitySyzygies}.
Better upper bounds for the consistency problem are known only via assuming the generalized Riemann hypothesis, which places it in the second level of the polynomial hierarchy; see \cite{koiran1996hilbertNullstellensatzPolynomialHierachy} for the result over~$\Q$, and~\cite{manssourParametricVersionHilbert2025} for a recent extension to rational function fields, which is related to the approach in~\cref{lemma:attainability symbolic}.

Since Gr\"obner bases solve the ideal membership problem, it comes at no surprise that algorithms computing them have doubly exponential complexity (in the number of variables, or the dimension of the solution variety) in the worst case as well~\cite{mayr1982complexityWordProblem,dubeStructurePolynomialIdeals1990,mayrDimensiondependentBoundsGrobner2013}.
In this section, we present several essential optimizations to circumvent as many Gr\"obner basis computations as possible when running the algorithms from \cref{subsection:algorithm outline}.

\paragraph{Highest weight inclusion}
Let $T$ be a non-zero tensor. Then $(e_1\sep e_1 \sep e_1) \in \Delta(T)$.
This is most easily seen using the representation theoretic description (\cref{equation:representation theoretic description}), as $\ISO_{e_1,e_1,e_1}$ equals the identity on the irreducible representation $V = \C^a\ot\C^b\ot\C^c$.
Thus, we can update $\ineqs \gets \{h \in \ineqs \mid \<(e_1\sep e_1 \sep e_1), h> \geq 0\}$.

\paragraph{Generic polytope inequalities}

Since the Kronecker polytope $\kronpol{a}{b}{c}$ (\cref{equation:generic moment polytope}) contains $\Delta(T)$ for any $T \in \C^a\ot\C^b\ot\C^c$, we know all inequalities that are valid for $\kronpol{a}{b}{c}$ are also valid for $\Delta(T)$.%
\footnote{In fact, for tensors of shape $4 \times 4 \times 4$ attainability verification of these specific inequalities (\cref{algorithm:tensor checking-inequalities}) takes too much time to be practical.}
So if the Kronecker polytope is known, we can restrict our attention to $\ineqs^{\setminus \textnormal{g}} \coloneqq \{ h \in \ineqs \mid \kronpol{a}{b}{c} \not\subseteq H_h\}$ instead.
The inequalities defining $\kronpol{a}{b}{c}$ can be added after filtering for attainability.
We use the generic inequalities describing $\kronpol{3}{3}{3}$ and $\kronpol{4}{4}{4}$ found in \cite{franz2002, vergneInequalitiesMomentCone2017}.

\paragraph{Intermediate polytope construction}
As we are executing \cref{algorithm:tensor algorithm borel}, we gradually gain knowledge about the Borel polytope $\DeltaB(T)$, in the form of valid inequalities. Let $\ineqs_t \subseteq \ineqs$ be the set of inequalities that are valid at step $t$.
Then $\bigcap_{h\in\ineqs_t} H_h \supseteq \Delta(T)$.
In particular, any $h \in \ineqs \setminus \ineqs_t$ that is valid for $\bigcap_{h\in\ineqs_t} H_h$ will also be valid for $\Delta(T)$, and it will be redundant to run the attainability algorithm for it.
In practice, there will be many inequalities in $\ineqs$ that are valid but redundant.  Their potentially lengthy Gr\"obner basis computations can be avoided using intermediate polytope constructions.

In principle, one can compute whether an inequality~$h$ is already implied by the inequalities in~$\ineqs_t$ by determining if~$h$ is in the cone generated by the~$h' \in \ineqs_t$, which is a linear feasibility problem.
However, this is computationally expensive in practice, given the number of~$h$ for which we would like to do so.
Instead, we enumerate the vertices of $\bigcap_{h'\in\ineqs_t} H_{h'}$, and determine redundancy of~$h$ by evaluating it on the vertices, which amounts to a single matrix-matrix multiplication.
We note that vertex enumeration for polyhedra is an NP-hard problem in general~\cite{MR2383757} (more precisely, it is NP-hard in the unbounded case, while for our setting of (bounded) polytopes the complexity is still open), and in particular implementations can be slow in scenarios when there are many redundant inequalities, as is the case for us.

Hence, this optimization needs to be fine-tuned to balance the hardness of the Gr\"obner basis computations with the hardness of vertex enumeration.

We take a heuristic approach.
First, we sort the inequalities based on their maximal absolute value, from small to large.
This is motivated by the fact that all polytopes we have computed are defined by inequalities with small integer entries, so we check these first.
Secondly, we enumerate the vertices of the outer approximation using $\ineqs_t$ only when we have found $C \in \N$ new inequalities, where $C$ is some parameter which can be adjusted.

\begin{remark}[Vertex enumeration with irredundant inequalities]
We use the same idea to speed up (intermediate) polytope construction and finding irredundant set of inequalities for polytopes (such as is done as part of \cref{algorithm:tensor algorithm randomized}).
In general $\ineqs_T$ will contain many redundant inequalities, many with relatively large entries.
This can drastically increase the running times of most vertex enumeration algorithms.
We once again sort $\ineqs_T$ based on the maximal absolute entry.
Then instead of running the vertex enumeration algorithm on all valid inequalities $\ineqs_T$ at once,
we instead construct the polytope in chunks, removing the redundant inequalities from $\ineqs_T$ in each step.
\end{remark}

\paragraph{Point inclusion: maxranks}

Just as we could remove inequalities that exclude the highest weight, we can remove all inequalities from $\ineqs$ that exclude some explicitly known points $p \in \Delta(T)$.
In our setting an easy to compute set of points in $\Delta(T)$ can be derived from the \emph{maxranks} of $T$.
The first maxrank of $T \in \C^{a} \ot \C^{b} \ot \C^{c}$ is defined as follows.
Given a vector $\beta \in \C^a$, we can take the weighted sum of the slices of $T$, which is a matrix of size $b \times c$, and consider its matrix rank.
The first maxrank is equal to the largest rank we can find. We have
\begin{equation}
    \label{definition:maxrank}
    \maxrank_1(T) \coloneqq \max \big\{  \rank\! \big((\beta \ot I_b \ot I_c) T\big)
    \ \big|\ \beta \in \C^{1 \times a} \big\}
\end{equation}
We can do the same for the other two systems to define $\maxrank_2$ and $\maxrank_3$.
Next, denote with $u_r \in \C^m$ the uniform vector of length $r$, defined as $u_r = \sum_{i=1}^r \tfrac1r e_i$.
Then define the points $p_i(r)$ by
\begin{align}
    p_1(r) \coloneqq \big(\, e_1 \sep u_r \sep u_r \,\big), \quad
    p_2(r) \coloneqq \big(\, u_r \sep e_1 \sep u_r \,\big), \quad
    p_3(r) \coloneqq \big(\, u_r \sep u_r \sep e_1 \,\big).
\end{align}
We have that $p_i(r) \in \Delta(T)$ whenever $r \leq \maxrank_i(T)$.
Namely, let $\beta \in \C^{1 \times a}$ be the vector achieving the maximum $r$ in \cref{definition:maxrank}.
Then using Gaussian elimination, we can transform the resulting matrix into $I_r$.
This defines a restriction $T \geq e_1 \ot \big(\sum_{i=1}^r e_i \ot e_i\big)$. It is not hard to its image under $\mu$ equals $(e_1 \sep u_r \sep u_r)$, and the argument is done.
\begin{lemma}%
\label{lemma:maxranks}
    Let $T$ be a $3$-tensor. If $r \leq \maxrank_i(T)$, then $p_i(r) \in \Delta(T)$.\footnote{The converse of \cref{lemma:maxranks} is true as well, but we will not need it.}
\end{lemma}
We conclude that when running the algorithms, we can replace $\ineqs$ by the set
\begin{align}
    \big\{h \in \ineqs \mid \< p_i(r_i), h> \geq 0 \text{ for all } i \in [3] \text{ and } r_i \leq \maxrank_i(T)\big\}.
\end{align}

\subsubsection{Minimizing the starting support}

In \cref{subsection:tensor attainability} we generated generic elements $\Tgeneric \in \G \cdot T$ either via randomization (say $T'$) or symbolic means (as $T_{\textnormal{s}}$).
In both cases we obtain these tensors by application of upper triangular matrices to $T$.
This tensor then defines the polynomial system we need to solve to verify attainability.
If we can minimize the support of $T'$ and $T_{\textnormal{s}}$ the resulting polynomial system will be simpler, with fewer terms per polynomial.
This makes the Gr\"obner basis computations faster.
We can influence the this support by replacing $T$ with a tensor $T'$ equivalent to it, for instance by minimizing the total distance of its non-zero entries to the corner at index $(1,1,1)$.

\begin{example}
We can replace the tensor in $\C^3\ot\C^3\ot\C^3$ given by (notated by putting its slices side-by-side)
\begin{align}
\unit{3} \coloneqq
\left[
\begin{array}{ccc|ccc|ccc}
1 & \grayzero & \grayzero   & \grayzero & \grayzero & \grayzero   & \grayzero & \grayzero & \grayzero \\
\grayzero & \grayzero & \grayzero   & \grayzero & 1 & \grayzero   & \grayzero & \grayzero & \grayzero \\
\grayzero & \grayzero & \grayzero   & \grayzero & \grayzero & \grayzero   & \grayzero & \grayzero & 1 \\
\end{array}
\right]
\quad\text{with}\quad
\left[
\begin{array}{ccc|ccc|ccc}
\grayzero & \grayzero & \grayzero   & \grayzero & \grayzero & \grayzero   & 1 & \grayzero & \grayzero \\
\grayzero & \grayzero & \grayzero   & \grayzero & 1 & \grayzero   & \grayzero & \grayzero & \grayzero \\
\grayzero & \grayzero & 1   & \grayzero & \grayzero & \grayzero   & \grayzero & \grayzero & \grayzero \\
\end{array}
\right] \sim \unit{3}.
\end{align}
These tensors are clearly equivalent, so their moment polytopes are equal.
These two expressions would lead to the following two supports after applying upper triangular matrices respectively:
\begin{align}
\left[
\begin{array}{ccc|ccc|ccc}
* & * & *   & * & * & *   & * & * & * \\
* & * & *   & * & * & *   & * & * & * \\
* & * & *   & * & * & *   & * & * & * \\
\end{array}
\right]
\quad\text{and}\quad
\left[
\begin{array}{ccc|ccc|ccc}
* & * & *   & * & * & \grayzero   & * & \grayzero & \grayzero \\
* & * & *   & * & * & \grayzero   & \grayzero & \grayzero & \grayzero \\
* & * & *   & \grayzero & \grayzero & \grayzero   & \grayzero & \grayzero & \grayzero \\
\end{array}
\right].
\end{align}
\end{example}
\section{The moment polytopes of all \texorpdfstring{$3\times3\times3$}{3 x 3 x 3} tensors}
\label{section:C333}

In this section, we discuss how, using the algorithm of \cref{section:algorithms for computing tensor moment polytopes}, we determine the moment polytope $\Delta(T)$ of every tensor $T\in V = \C^3 \ot \C^3 \ot \C^3$.

Our approach relies on the orbit classification of $V$ by Nurmiev \cite{nurmievOrbitsInvariantsCubic2000} (building on \cite{vinbergWeylGroupGraded1976}) with a correction by \cite{ditraniClassificationRealComplex2023}. This classification divides $V$ into five parametrized families of orbits. We discuss these in
\cref{subsubsection:333-classification}.
In \cref{subsec:fam123} we prove that for every tensor in families 1, 2 and 3 the moment polytope equals the maximal moment polytope $\kronpol{3}{3}{3}$. In \cref{subsec:fam4}
we discuss how to use the algorithms of \cref{section:algorithms for computing tensor moment polytopes} to compute the moment polytopes for the tensors in families 4 and~5.
The complete description of the vertices of all these polytopes can be found in \cref{table:333 vertex data} and is available online at \cite{vandenBerg2025momentPolytopesGithub}.

\subsection{SL-semistability}

We introduce some well-known machinery which we will require in this section.
Recall that $\Sl \subseteq \G$ is the subset of triples of matrices each with determinant one.
We call a tensor $T \in \C^a\ot\C^b\ot\C^c$
\emph{$\Sl$-semistable} whenever $0 \not\in \overline{\Sl \cdot T}$, and \emph{$\Sl$-unstable} otherwise.
We call $T$ \emph{$\Sl$-polystable} when the $\Sl$-orbit $\Sl \cdot T$ is closed.
For us, the key fact about $\Sl$-semistability is that it is equivalent to inclusion of the uniform point in the moment polytope. Recall that we write $u_m \coloneqq (1/m,\ldots,1/m) \in \C^m$ for the uniform vector.
Then $(u_a \sep u_b \sep u_c) \in \Delta(T)$ if and only if $\mu(S) = (I_a/a, I_b/b, I_c/c)$ for some $S \in \overline{\G \cdot T}$.
We will use the following three results.\footnote{These results are also known in the general setting. We reference to literature in the tensor setting where possible.}

\begin{theorem}[{\cite[Theorem~3.2]{burgisser2018alternatingMinimization}, \cite[Lemma~4.34]{christandlUniversalPointsAsymptotic2021}}%
]
    \label{corollary:uniform marginals}
    Let $T \in \C^a\ot\C^b\ot\C^c \setminus \{0\}$.
    The tensor $T$ is $\Sl$-semistable if and only if $\mu(S) = (I_a/a,I_b/b,I_c/c)$ for some $S \in \overline{\G \cdot T}$.
\end{theorem}

\begin{theorem}[{\cite[Theorem~3.26]{wallachGeometricInvariantTheory2017}}%
\footnote{This is the well-known Kempf--Ness theorem for the $\Sl$-action on $\C^a \ot \C^b \ot \C^c$.
Here we use that $\mu(S) = (I_a/a, I_b/b, I_c/c)$ is equivalent to being ``critical'' as defined in \cite{wallachGeometricInvariantTheory2017}, by a straightforward computation.}%
]
    \label{theorem:tensor kempf-ness SL}
    Let $T \in \C^a\ot\C^b\ot\C^c \setminus \{0\}$.
    The tensor $T$ is $\Sl$-polystable if and only if $\mu(S) = (I_a/a,I_b/b,I_c/c)$ for some $S \in \G \cdot T$.
\end{theorem}

\begin{proposition}[{\cite{burgisserFundamentalInvariants2017}}%
\footnote{In \cite[Proposition~3.10]{burgisserFundamentalInvariants2017} it is proved that a specific non-constant $\Sl$-invariant vanishes on the boundary, but their proof in fact applies to any non-constant $\Sl$-invariant.
Since the vanishing of all non-constant $\Sl$-invariants implies $\Sl$-instability, the result follows.
In \cite[Lem.~3.4]{vandenBerg2025mmPolytope} we provide a complete proof.}%
]
    \label{lemma:orbit border}
    Let $T \in \C^a \ot \C^b \ot \C^c$. If $T$ is $\Sl$-polystable,
    then every element in the boundary $\overline{\G \cdot T} \setminus (\G \cdot T)$ is $\Sl$-unstable.
\end{proposition}

\subsection{Orbit classification of \texorpdfstring{$3\times 3 \times 3$}{3 x 3 x 3} tensors}
\label{subsubsection:333-classification}

To determine the moment polytopes for all tensors in $V = \C^3 \ot \C^3 \ot \C^3$ we will use the following classification \cite{nurmievOrbitsInvariantsCubic2000, ditraniClassificationRealComplex2023}.
The classification obtains every $\Sl$-orbit in $\C^3\ot\C^3\ot\C^3$ as the orbit of a tensor
$T_{\textnormal{ss}} + T_{\textnormal{u}}$, where $(T_{\textnormal{ss}}, T_{\textnormal{u}}) \in \bigcup_{i=1}^5 \mathcal S_i \times \mathcal U_i$ with $\mathcal S_i, \mathcal U_i \subseteq \C^3\ot\C^3\ot\C^3$ as defined below.
Each set $\mathcal S_i$ contains only $\Sl$-semistable tensors, and each set $\mathcal U_i$ is finite and contains only $\Sl$-unstable tensors.
We will write $e_{i,j,k} \coloneqq e_i \ot e_j \ot e_k \subseteq \C^3\ot\C^3\ot\C^3$, and use $e_{ijk} = e_{i,j,k}$ as a shorthand. %

\begin{definition}
Define the tensors
\begin{align}
\label{equation:333 vs}
  v_1 \coloneqq e_{111} + e_{222} + e_{333}, \qquad
  v_2 \coloneqq e_{123} + e_{231} + e_{312}, \qquad
  v_3 \coloneqq e_{132} + e_{213} + e_{321}.
\end{align}
Define the sets
\begin{align}
\begin{split}
\label{definition:333 semistable sets}
\mathcal S_1
&\coloneqq \bigl\{  \familyparam_1 v_1 + \familyparam_2 v_2 + \familyparam_3 v_3  \mid \familyparam_1,\familyparam_2,\familyparam_3 \in \C \setminus\{0\},\ (\familyparam_1^3 + \familyparam_2^3 + \familyparam_3^3)^3 - (3 \familyparam_1 \familyparam_2 \familyparam_3)^3 \neq 0 \bigr\}
\\
\mathcal S_2
&\coloneqq \bigl\{  \familyparam_1 v_1 + \familyparam_2 v_2  \mid \familyparam_1,\familyparam_2 \in \C \setminus \{0\},\ (\familyparam_1^3 + \familyparam_2^3) \neq 0 \bigr\}
\\
\mathcal S_3
&\coloneqq \bigl\{  \familyparam v_1  \mid \familyparam \in \C \setminus \{0\} \bigr\}
\\
\mathcal S_4
&\coloneqq \bigl\{  \familyparam (v_2 - v_3)  \mid \familyparam \in \C \setminus \{0\} \bigr\}
= \bigl\{  \familyparam e_1 \wedge e_2 \wedge e_3  \mid \familyparam \in \C \setminus \{0\} \bigr\}
\\
\mathcal S_5  &\coloneqq \{0\}.
\end{split}
\end{align}
Define the finite sets $\mathcal U_1, \mathcal U_2, \mathcal U_3$ and $\mathcal U_4$
as in \cref{table:unstable tensors families 1 to 4}.%
\footnote{Nurmiev's~\cite{nurmievOrbitsInvariantsCubic2000} original definition of $\mathcal U_2$ was inaccurate, and was corrected in~\cite{ditraniClassificationRealComplex2023}.}
Define the finite set $\mathcal U_5$ to be the tensors in \cref{table:c333 unstable tensors info} together with their cyclic permutations.%
\footnote{The cyclic permutations of a tensor $\sum_{\ell=1}^r e_{i_\ell,j_\ell,k_\ell}$ are
$\sum_{\ell=1}^r e_{k_\ell,i_\ell,j_\ell}$
and $\sum_{\ell=1}^r e_{j_\ell,k_\ell,i_\ell}$.}
Lastly, for $i \in [5]$ define the set
\begin{align}
\mathcal F_i \coloneqq \bigcup_{(T_{\textnormal{ss}}, T_{\textnormal{u}})\in \mathcal S_i \times \mathcal U_i} \Sl \cdot \big( T_{\textnormal{ss}} + T_{\textnormal{u}}\big),
\end{align}
which we refer to as \emph{family $i$}.
\end{definition}
\begin{table}[H]
\def\arraystretch{1.05}
\small
\setlength\tabcolsep{0.4em}
\makebox[\textwidth][c]{
\begin{tabular}{*{4}{cl}}\toprule
\multicolumn{2}{c}{$\mathcal U_1$} &
\multicolumn{2}{c}{$\mathcal U_2$} &
\multicolumn{2}{c}{$\mathcal U_3$} &
\multicolumn{2}{c}{$\mathcal U_4$} \\\cmidrule(lr){1-2}\cmidrule(lr){3-4}\cmidrule(lr){5-6}\cmidrule(lr){7-8}
No.\ & \multicolumn{1}{c}{Tensor} &
No.\ & \multicolumn{1}{c}{Tensor} &
No.\ & \multicolumn{1}{c}{Tensor} &
No.\ & \multicolumn{1}{c}{Tensor}
\\
\cmidrule(lr){1-8}
1 & 0 &  1 &$e_{132}+e_{213}$   &   1 &$e_{123}+e_{132}+e_{213}+e_{231}$& 1 &$e_{113}+e_{122}+e_{131}+e_{212}+e_{221}+e_{311}$ \\
  &   &  2 &$e_{132}$  &   2 &$e_{123}+e_{132}+e_{213}$&         2 &$e_{113}+e_{131}+e_{311}+e_{222}$ \\
  &   &  3 & 0  &   3 &$e_{123}+e_{132}$&                 3 &$e_{112}+e_{121}+e_{211}$ \\
  &   &     &   &   4 &$e_{123}+e_{231}$&    4 &$e_{111}+e_{222}$ \\
  &   &     &   &   5 &$e_{123}$&    5 &$e_{111}$ \\
  &   &     &   &   6 & 0 &   6 & 0 \\  \bottomrule
\end{tabular}}
\caption{Definition of the finite sets of tensors $\mathcal U_1$, $\mathcal U_2$, $\mathcal U_3$ and $\mathcal U_4$ in $\C^3\ot\C^3\ot\C^3$. %
The set~$\mathcal U_5$ is described in~\cref{table:c333 unstable tensors info}.}
\label{table:unstable tensors families 1 to 4}
\end{table}
\begin{proposition}[{\cite{nurmievOrbitsInvariantsCubic2000, ditraniClassificationRealComplex2023}}]
    Every tensor in $\bigcup_i \mathcal U_i$ is $\Sl$-unstable.
    Moreover,
    for any $(T_{\textnormal{ss}}, T_{\textnormal{u}}) \in \mathcal S_i \times \mathcal U_i$,
    the tensor $T_{\textnormal{ss}} + T_{\textnormal{u}}$ is $\Sl$-unstable if and only if $i = 5$.
    \label{lemma:classification stability}
\end{proposition}

\begin{theorem}[{\cite{nurmievOrbitsInvariantsCubic2000, ditraniClassificationRealComplex2023}}]
\label{thm:c333 classification merged}
    It holds that
    \begin{align}
    \C^3 \ot \C^3 \ot \C^3 = \bigsqcup_{i=1}^5 \mathcal F_i.
    \end{align}
    Here $\bigsqcup$ denotes the disjoint union.
\end{theorem}
\begin{remark}
It is also true that each family $\mathcal F_i$ is partitioned by the $\Sl$-invariant subsets $\bigcup_{T_{\textnormal{ss}} \in \mathcal S_i} \Sl \cdot \big( T_{\textnormal{ss}} + T_{\textnormal{u}}\big)$ with $T_{\textnormal{u}}$ ranging over $\mathcal U_i$ \cite{nurmievOrbitsInvariantsCubic2000}, using which we can associate to each tensor in $\mathcal F_i$ a unique unstable tensor from $\mathcal U_i$. We will not require this here.
\end{remark}

\begin{remark}
It is known that family 1 is Zariski-dense in $\C^3 \ot \C^3\ot\C^3$ \cite{nurmievOrbitsInvariantsCubic2000}.
It can also be shown that the complement of family 1 in $\C^3 \ot \C^3\ot\C^3$ is the set of tensors with vanishing hyperdeterminant.%
\footnote{This refers to the hyperdeterminant of $3 \times 3 \times 3$ tensors as defined in~\cite{gelfandHyperdeterminants1992}, which is an~$\SL$-invariant polynomial of degree~$36$. We refer to~\cite{bremnerFundamentalInvariantsAction2013,bremnerHyperdeterminantPolynomialFundamental2014} for explicit strategies for evaluating it.}
Family 3 consists of the tensors with $\Sl$-semistable part equivalent (under~$\G$) to the unit tensor $\unit{3}$, and
family 4 consists of the tensors that are equivalent to the skew-symmetric tensor $e_1 \wedge e_2 \wedge e_3$.\footnote{This tensor is sometimes called the Aharonov state or singlet state in the physics literature.}
And finally, family 5 is equal to the set of $\Sl$-unstable tensors.
\end{remark}

\begin{remark}
The proof of the classification of \cref{lemma:classification stability,thm:c333 classification merged} uses the fact that there exists a linear embedding~$\C^3 \ot \C^3 \ot \C^3 \hookrightarrow \mathfrak{e}_6$, where~$\mathfrak{e}_6$ is the unique complex Lie algebra of type~$E_6$.
Then $T_{\textnormal{ss}}$ and $T_{\textnormal{u}}$ make up the Jordan decomposition of $T_{\textnormal{ss}} + T_{\textnormal{u}}$ in this Lie algebra.
This construction is an example in Vinberg's theory of~$\theta$-groups~\cite{vinbergWeylGroupGraded1976}.
\end{remark}

We will also make use of the following structural property of the above classification.
\begin{lemma}[{\cite[Sec.~3]{elashviliClassificationTrivectors9dimensional1978}}]
  \label{lem:c333 classification structural properties}
  For any $(T_{\textnormal{ss}}, T_{\textnormal{u}}) \in \bigcup_{i=1}^5 \mathcal S_i \times \mathcal U_i$,
  the stabilizer of $T_{\textnormal{ss}}$ in $\Sl$ acts transitively on $\Sl \cdot T_{\textnormal{u}}$.
\end{lemma}
\begin{corollary}
  \label{cor:semistable part has smaller polytope}
  Let $T\in \C^3 \ot \C^3\ot\C^3$.
  If $T \in \mathcal F_i$,
  then
    $\Delta(T_{\textnormal{ss}}) \subseteq \Delta(T)$
    for some $T_{\textnormal{ss}} \in \mathcal S_i$.
\end{corollary}
\begin{proof}
    By \cref{thm:c333 classification merged}, $T_{\textnormal{ss}} + T_{\textnormal{u}} \in \Sl \cdot T$ for some $(T_{\textnormal{ss}}, T_{\textnormal{u}}) \in \mathcal S_i \times \mathcal U_i$.
    Let $H \subseteq \Sl$ be the stabilizer of $T_{\textnormal{ss}}$.
    Then by \cref{lem:c333 classification structural properties},
    $H \cdot (T_{\textnormal{ss}} + T_{\textnormal{u}}) = T_{\textnormal{ss}} + \Sl \cdot T_{\textnormal{u}}$.
    The closure of this set contains $T_{\textnormal{ss}}$, because $0 \in \overline{\SL \cdot T_{\textnormal{u}}}$ by the fact that $T_{\textnormal{u}}$ is $\Sl$-unstable (\cref{lemma:classification stability}).
    We conclude that $T_{\textnormal{ss}} \in \overline{\Sl \cdot T} \subseteq \overline{\G \cdot T}$, which implies the result by \cref{proposition:degeneration monotone}.
\end{proof}

\subsection{Proof that families 1, 2 and 3 have maximal moment polytope}\label{subsec:fam123}

We will prove that for any tensor $T$ in family 1, 2 or 3 (in the classification as in~\cref{subsubsection:333-classification}), their moment polytope equals the maximal moment polytope (Kronecker polytope).

\begin{theorem}\label{th:fam123gen}
  Let~$T \in \C^3 \ot \C^3 \ot \C^3$ be in family $1$, $2$ or $3$ .
  Then %
  $\Delta(T) = \kronpol{3}{3}{3}$.
\end{theorem}

Our strategy is based on the following observation.

\begin{proposition}
  \label{prop:c333 generic is origin c233 perms}
  $\kronpol{3}{3}{3}$ is the convex hull of $\kronpol{2}{3}{3}$, $\kronpol{3}{2}{3}$, $\kronpol{3}{3}{2}$ and the point $(\sfrac13, \sfrac13, \sfrac13\sep \sfrac13, \sfrac13, \sfrac13\sep \sfrac13, \sfrac13, \sfrac13)$.
\end{proposition}
\begin{proof}
We use the explicit set of vertices of $\kronpol{3}{3}{3}$ as determined in \cite[Prop.~5.1]{franz2002}. Up to permutations of the three subsystems, the vertices are given by
\begin{center}
  \renewcommand{\arraystretch}{1.3}
  \begin{tabular}{*{2}{c@{}c@{}c@{}c@{\ }c@{\ }c@{}c@{}c@{\ }c@{\ }c@{}c@{}c@{}c@{\hspace{3em}}}}
    $($ & $\sfrac13,$ & $\sfrac13,$ & $\sfrac13$ & $\sep$ & $\sfrac13,$ & $\sfrac13,$ & $\sfrac13$ & $\sep$ & $\sfrac13,$ & $\sfrac13,$ & $\sfrac13$ & $),$  & $($ & $\sfrac12,$ & $\sfrac14,$ & $\sfrac14$ & $\sep$ & $\sfrac12,$ & $\sfrac12,$ & $0$ & $\sep$ & $\sfrac34,$ & $\sfrac14,$ & $0$ & $),$ \\
    $($ & $\sfrac13,$ & $\sfrac13,$ & $\sfrac13$ & $\sep$ & $\sfrac13,$ & $\sfrac13,$ & $\sfrac13$ & $\sep$ & $\sfrac12,$ & $\sfrac12,$ & $0$ & $),$        & $($ & $\sfrac12,$ & $\sfrac12,$ & $0$ & $\sep$ & $\sfrac12,$ & $\sfrac12,$ & $0$ & $\sep$ & $\sfrac12,$ & $\sfrac12,$ & $0$ & $),$ \\
    $($ & $\sfrac13,$ & $\sfrac13,$ & $\sfrac13$ & $\sep$ & $\sfrac13,$ & $\sfrac13,$ & $\sfrac13$ & $\sep$ & $1,$ & $0,$ & $0$ & $),$                    & $($ & $\sfrac12,$ & $\sfrac12,$ & $0$ & $\sep$ & $\sfrac12,$ & $\sfrac12,$ & $0$ & $\sep$ & $1,$ & $0,$ & $0$ & $),$ \\
    $($ & $\sfrac13,$ & $\sfrac13,$ & $\sfrac13$ & $\sep$ & $\sfrac12,$ & $\sfrac12,$ & $0$ & $\sep$ & $\sfrac12,$ & $\sfrac12,$ & $0$ & $),$              & $($ & $\sfrac23,$ & $\sfrac16,$ & $\sfrac16$ & $\sep$ & $\sfrac23,$ & $\sfrac16,$ & $\sfrac16$ & $\sep$ & $\sfrac12,$ & $\sfrac12,$ & $0$ & $),$ \\
    $($ & $\sfrac13,$ & $\sfrac13,$ & $\sfrac13$ & $\sep$ & $\sfrac23,$ & $\sfrac16,$ & $\sfrac16$ & $\sep$ & $\sfrac12,$ & $\sfrac12,$ & $0$ & $),$        & $($ & $1,$ & $0,$ & $0$ & $\sep$ & $1,$ & $0,$ & $0$ & $\sep$ & $1,$ & $0,$ & $0$ & $),$ \\
    $($ & $\sfrac13,$ & $\sfrac13,$ & $\sfrac13$ & $\sep$ & $\sfrac23,$ & $\sfrac13,$ & $0$ & $\sep$ & $\sfrac23,$ & $\sfrac13,$ & $0$ & $).$ &&&&&&&&&&&&$\phantom{\sfrac13}$\\
  \end{tabular}
\end{center}
By \cref{proposition:generic moment polytope}, there exists a tensor $T \in \C^3\ot\C^3\ot\C^3$ such that $\Delta(T) = \kronpol{3}{3}{3}$. Then by \cref{proposition:tensor moment polytope generic restriction}, all vertices listed above except $(\sfrac13,\sfrac13,\sfrac13 \sep \sfrac13,\sfrac13,\sfrac13 \sep \sfrac13,\sfrac13,\sfrac13)$ are included in the moment polytope $\Delta(S)$ of some tensor $S \in \C^3\ot\C^3\ot\C^2$ with $T \geq S$.
It follows that $\kronpol{3}{3}{2}$ includes these vertices as well. Since we obtain  $\kronpol{3}{2}{3}$ and $\kronpol{2}{3}{3}$ by permuting the systems of $\kronpol{3}{3}{2}$, the claim follows.
\end{proof}

Now let $T \in \mathcal F_1 \cup \mathcal F_2 \cup \mathcal F_3$.
The inclusion of the uniform point $(\sfrac13,\sfrac13,\sfrac13 \sep \sfrac13,\sfrac13,\sfrac13 \sep \sfrac13,\sfrac13,\sfrac13)$ in $\Delta(T)$ is given by \cref{corollary:uniform marginals}, using the fact that all tensors of family 1, 2 and 3 are $\Sl$-semistable (\cref{lemma:classification stability}).
We will now prove that for establishing inclusion of $\kronpol{2}{3}{3}$, it is sufficient to show that $T$ restricts to an $\Sl$-semistable tensor in $\C^2 \ot \C^3 \ot \C^3$.
A classification of the~$\G$-orbits in $\C^2 \ot \C^3 \ot \C^3$ follows from the classification of $\Sl$-unstable tensors in~$\C^3 \ot \C^3 \ot \C^3$, i.e., family~5 from~\cref{thm:c333 classification merged} (alternatively, one can consult~\cite{chenRangeCriterionClassification2006a}).
By inspection of~\cref{table:c333 unstable tensors polytope data}, tensor~$9$ from~\cref{table:c333 unstable tensors info} is the only $\Sl$-semistable tensor in this format up to equivalence.
We formalize this below by showing that it has dense~$\G$-orbit, although we use a slightly different description of the tensor in the proof.

\begin{lemma}
  \label{proposition:c233 dense orbit}
  In $\C^2 \ot \C^3 \ot \C^3$, any two $\Sl$-semistable tensors are $\G$-equivalent.
  The $\G$-orbit of any $\Sl$-semistable tensor is a dense subset of $\C^2 \ot \C^3 \ot \C^3$.
\end{lemma}
\begin{proof}
  Let~$T = e_1 \ot I_3 + e_2 \ot M$ where $I_3 = e_1 \ot e_1 + e_2 \ot e_2 + e_3 \ot e_3$ and $M = e_1 \ot e_1 + \zeta_3 e_2 \ot e_2 + \zeta_3^2 e_3 \ot e_3$, where $\zeta_3$ is a primitive cube root of unity.
  Then the moment map evaluates to $\mu(T) = (I_3/3 , I_3/3 , I_3/3)$, which implies that $T$ has closed $\Sl$-orbit by the Kempf--Ness theorem (\cref{theorem:tensor kempf-ness SL}).
  We may apply \cref{lemma:orbit border} to find that all tensors in the boundary of $\G \cdot T$ are $\Sl$-unstable.
  We will show that the dimension of $\G \cdot T$ is equal to $\dim(\C^2\ot\C^3\ot\C^3) = 18$.
  This implies its Zariski closure has dimension 18, which means $\overline{\G \cdot T} = \C^2\ot\C^3\ot\C^3$.
  In particular, any tensor in $\C^2\ot\C^3\ot\C^3$ is either in $\G \cdot T$ (in which case it is $\Sl$-semistable) or in the boundary of $\overline{\G \cdot T}$ (in which case it it $\Sl$-unstable). This proves the first claim.

  We now compute the dimension of $\G \cdot T$.
  It is equal to the dimension of $\G$ minus the dimension of the stabilizer of $T$ in $\G$~\cite[Thm.~AG.10.1]{borel2012}.
  We can compute these dimensions by computing the dimension of their tangent vector spaces at the identity element (that is, the dimension of their Lie algebras)~\cite[Cor.~I.3.6]{borel2012}.
  The Lie algebra of $\G$ consists of triples of matrices and has dimension $2^2 + 3^2 + 3^2 = 22$.
  The Lie algebra of the stabilizer of $T$ in $\G$ (by e.g.~\cite[Prop.~I.3.8 and I.3.22]{borel2012}) consists of those triples $(A_1, A_2, A_3)$ of matrices such that
  \begin{equation*}
    (A_1 \ot I_3 \ot I_3 \,+\, I_2 \ot A_2 \ot I_3 \,+\, I_2 \ot I_3 \ot A_3) T = 0.
  \end{equation*}
  It is easy to verify that this is spanned by the tuples
  \[
      (0,\ \diag(1,0,-1),\ \diag(-1,0,1))
    \qquad\text{and}\qquad
      (0,\ \diag(1,-1,0),\ \diag(-1,1,0)),
  \]
  as well as $(I_2,-I_3,0)$ and $(I_2,0,-I_3)$ (which stabilize any tensor).
  Therefore the stabilizer of $T$ is $4$-dimensional.
  We conclude that the $\G$-orbit of T has dimension $22-4=18$.
\end{proof}

\begin{corollary}
  \label{cor:c233 semistable implies generic polytope}
  Let $S \in \C^2 \ot \C^3 \ot \C^3$ be $\Sl$-semistable. Then $\Delta(S) = \kronpol{2}{3}{3}$.
  If a tensor $T$ satisfies $S \degenleq T$, then $\kronpol{2}{3}{3} \subseteq \Delta(T)$.
\end{corollary}
\begin{proof}
    By \cref{proposition:c233 dense orbit} and the geometric description of the moment polytope (\cref{equation:polytope spec,equation:generic moment polytope}), we find that $\Delta(S) = \mu(\C^2 \ot \C^3 \ot \C^3 \setminus \{0\}) \cap \weylchamber = \kronpol{2}{3}{3}$.
    The second claim then follows directly from monotonicity of moment polytopes under degeneration (\cref{proposition:degeneration monotone}).
\end{proof}

It remains to show that any tensor $T$ in family 1, 2 or 3 restricts to an $\Sl$-semistable tensor $S_T \in \C^2 \ot \C^3 \ot \C^3$.
We achieve this by showing that there exists a non-constant $\Sl$-invariant polynomial that does not vanish on $S_T$, which directly implies $0 \notin \overline{\Sl \cdot S_T}$ as desired.
We now construct this polynomial.
It is defined via a reduction to the homogeneous cubic discriminant, which is defined as follows.
Let $q(x,y) = a x^3 + 3 b x^2 y + 3 c x y^2 + d y^3 \in \C[x,y]_3$ be a bivariate homogeneous cubic polynomial.
Then the homogeneous discriminant $\Disc(q)$ of $q$ is defined by
\begin{equation}
    \label{def:discriminant-of-cubic}
    \Disc(q) \coloneqq 3 b^2 c^2  + 6 a b c d - 4 b^3 d - 4c^3 a - a^2 d^2.
\end{equation}

\begin{lemma}[see e.g.\ {\cite{kimuraClassificationIrreduciblePrehomogeneous1977}}]
  \label{lem:discriminant-is-invariant}
  The homogeneous discriminant is invariant under the $\SL_2$-action on~$\C[x,y]_3$. That is, %
  for every $A \in \SL_2$ and $q \in \C[x,y]_3$, we have $\Disc(q \circ A^{-1}) = \Disc(q)$.
\end{lemma}

We use the homogeneous discriminant to construct an $\Sl$-invariant on~$\C^2 \ot \C^3 \ot \C^3$.
Let $T = e_1 \ot M_1 + e_2 \ot M_2$ for $M_1, M_2 \in \C^3 \ot \C^3$ be an arbitrary tensor in $\C^2 \ot \C^3 \ot \C^3$.
Interpret $M_1$ and $M_2$ as $3 \times 3$ matrices. Then $\det(x M_1 + y M_2)$ is a homogeneous cubic polynomial in $x$ and $y$. Define
\begin{equation}
I_{233}(T) \coloneqq \Disc(\det(x M_1 + y M_2)) \in \C[\C^2\ot\C^3\ot\C^3].
\end{equation}
\begin{proposition}
  \label{prop:c233 only invariant}
  $I_{233}$ is a non-zero homogeneous $\Sl$-invariant on $\C^2 \ot \C^3 \ot \C^3$.
\end{proposition}
\begin{proof}
  Write $T = e_1 \ot M_1 + e_2 \ot M_2 \in \C^2 \ot \C^3 \ot \C^3$.
  Let $A \in \SL_2$ and $B, C \in \SL_3$.
  Then $I_{233} \big( (A,I_3,I_3) \cdot T \big) = \Disc\!\big( \det(x M_1 + y M_2) \circ A^{-1} \big) = I_{233}(T)$ by the $\SL_2$-invariance of the discriminant (\cref{lem:discriminant-is-invariant}).
  We also find that $(I \ot B \ot C) T = e_1 \ot (B M_1 C^T) + e_2 \ot (B M_2 C^T)$ and $\det(B (x M_1 + y M_2) C^T) = \det(x M_1 + y M_2)$.
  We conclude that $I_{233} \big( (A,B,C) \cdot T \big) = I_{233}(T)$, as desired.
\end{proof}

\begin{remark}
    It can be shown that the ring of invariants $\C[\C^2 \ot \C^3 \ot \C^3]^{\SL_2 \times \SL_3 \times \SL_3}$ is in fact equal to $\C[I_{233}]$. That is, $I_{233}$ is essentially the only $\Sl$-invariant on $\C^2 \ot \C^3 \ot \C^3$.
    See e.g.\ \cite[pg.~145]{kimuraClassificationIrreduciblePrehomogeneous1977} or \cite[Table~II]{kacRemarksNilpotentOrbits1980}.
\end{remark}

\begin{proof}[Proof of \cref{th:fam123gen}]
    By \cref{cor:semistable part has smaller polytope}, there exists some $T_{\textnormal{ss}} \in \bigcup_{i=1}^3 \mathcal S_i$ with $\Delta(T_{\textnormal{ss}}) \subseteq \Delta(T)$.
    We may assume without loss of generality that $T = T_{\textnormal{ss}}$.
    We will prove that $(\sfrac13, \sfrac13, \sfrac13\sep \sfrac13, \sfrac13, \sfrac13\sep \sfrac13, \sfrac13, \sfrac13) \in \Delta(T)$ and $\kronpol{2}{3}{3}, \kronpol{3}{2}{3}, \kronpol{3}{3}{2} \subseteq \Delta(T)$. Then \cref{prop:c333 generic is origin c233 perms} together with convexity of $\Delta(T)$ implies the result.
    The first inclusion follows from the fact that $T$ is $\Sl$-semistable (\cref{lemma:classification stability}) combined with \cref{corollary:uniform marginals}.

    We prove the other inclusions.
    By definition of $\mathcal S_i$ (\cref{definition:333 semistable sets}), we can write $T = \familyparam_1 v_1 + \familyparam_2 v_2 + \familyparam_3 v_3$ for some $\familyparam_1,\familyparam_2,\familyparam_3 \in \C$.
    Observe that~$T$ is invariant under cyclic permutations of the subsystems.
    So to prove the remaining inclusions, it suffices to prove $\kronpol{2}{3}{3} \subseteq \Delta(T)$.
    To this end, we exhibit a $2 \times 3$ matrix $A$ such that $I_{233}((A \ot I \ot I) T) \neq 0$.
    Then it follows by \cref{prop:c233 only invariant} that $(A \ot I \ot I) T$ is $\Sl$-semistable, after which \cref{cor:c233 semistable implies generic polytope} completes our argument.

    It only remains to construct a $2\times 3$ matrix $A$ such that $I_{233}((A \ot I \ot I) T) \neq 0$. Consider first $T \in \mathcal S_1$.
    Then all $\familyparam_1, \familyparam_2, \familyparam_3$ are non-zero.
    It may be verified that
    \begin{equation*}
      \left(\begin{bmatrix}
        1 & 0 & 0 \\
        0 & 1 & 0
      \end{bmatrix}
      \ot I \ot I\right) T =
      e_1 \ot
      \begin{bmatrix}
        \familyparam_1 & 0 & 0 \\
        0 & 0 & \familyparam_2 \\
        0 & \familyparam_3 & 0
      \end{bmatrix}
      + e_2 \ot
      \begin{bmatrix}
        0 & 0 & \familyparam_3 \\
        0 & \familyparam_1 & 0 \\
        \familyparam_2 & 0 & 0
      \end{bmatrix},
    \end{equation*}
    and evaluating $I_{233}$ on this tensor yields $- \familyparam_1^4 \familyparam_2^4 \familyparam_3^4 \neq 0$, as desired.
    For $T \in \mathcal S_2 \cup \mathcal S_3$ we use a slightly different restriction. In this case, $\familyparam_3 = 0$, and we find for all $u,v \in \C$ that
    \begin{equation*}
      I_{233}\left(\left(\begin{bmatrix}
          1 & 0 & u \\
          0 & 1 & v
        \end{bmatrix}
    \ot I \ot I\right) T\right) = \frac{1}{27} (\familyparam_1^3 + \familyparam_2^3)^4 u^2 v^2.
    \end{equation*}
    For $T \in \mathcal S_2 \cup \mathcal S_3$ we have $\familyparam_1^3 + \familyparam_2^3 \neq 0$, so setting $u = v = 1$ yields a non-zero value, as desired.
\end{proof}

\begin{corollary}
  The moment polytope of the unit tensor~$\unit{3}$ equals the Kronecker polytope. That is, $\Delta(\unit{3}) = \kronpol{3}{3}{3}$.
\end{corollary}
\begin{proof}
The unit tensor~$\unit{3}$ is in family 3, so this follows from \cref{th:fam123gen}.
\end{proof}

\subsection{Algorithmic computation of the moment polytopes for families 4 and 5}\label{subsec:fam4}

We can determine the moment polytopes of all tensors in family 4 and 5 using \cref{algorithm:tensor algorithm randomized} from \cref{section:algorithms for computing tensor moment polytopes}.
This is our randomized algorithm which is always correct when successful.
However, the symbolic attainability verification as done in \cref{line:determine attainability symbolically} presents a bottleneck, and we need to do some manual work to ensure our algorithm terminates.
We discuss our strategy now.
We use the notation from \cref{subsection:attainability algorithms}.
The strategy is based on two observations. The first is that if we add any polynomial $f$ to $\tensorpolysystem^{T_{\textnormal{s}}}(h)$, and the corresponding reduced Gr\"obner basis $G_{\Q(\G_{\uppertriangular})}(\tensorpolysystem^{T_{\textnormal{s}}}(h) \cup \{f\})$ is not equal to $\{1\}$, then the same is true for the original Gr\"obner basis $G_{\Q(\G_{\uppertriangular})}(\tensorpolysystem^{T_{\textnormal{s}}}(h))$.
Secondly, we saw that the Gr\"obner basis over $\Q$ using $T' = (A \ot B \ot C) T$ for generic $(A,B,C) \in \G_{\uppertriangular}$ arises from the symbolic Gr\"obner basis by filling $(A,B,C)$ into the symbolic coefficients (see the proof of \cref{lemma:attainability symbolic}).
The idea is to use the structure of the Gr\"obner basis over $\Q$ to help construction of the symbolic basis.
Adding polynomials can greatly reduce the runtime of the Gr\"obner basis computation.
This is best explained via the following example.

\begin{example}
\label{example:symbolic groebner derandomization}
    Consider $T =  e_{113} + e_{122} + e_{212} + e_{221} + e_{331}$ where $e_{ijk} = e_i \ot e_j \ot e_k \in \C^3\ot\C^3\ot\C^3$ (this is tensor 4 from \cref{table:c333 unstable tensors info})
    and the inequality $h = (\,0, -1, 1 \sep 1, 0, -1 \sep 1, 0, 0\,)$. %
    For some randomization the Gr\"obner basis over $\Q$ equals
    \begin{align*}
        \Big\{\ x_1,\ x_4 + \frac{223735}{102354},\ x_5,\ x_6 + \frac{2146}{4629},\ x_7 + \frac{12325}{29518},\  x_8 + \frac{12325}{29518}x_9\ \Big\}.
    \end{align*}
    We were not able to compute the Gr\"obner basis of $\tensorpolysystem^{T_{\textnormal{s}}}(h)$ directly.
    However, the Gr\"obner basis of $\tensorpolysystem^{T_{\textnormal{s}}}(h) \cup \{x_1\}$ could be computed within 0.001 seconds.
    It equals
    \begin{align*}
        \Big\{
        \ &x_1,\  x_4 + \frac{z_1z_4z_5z_{10} + z_1z_5^2z_{11}}{z_1z_4z_5z_8 + z_1z_5^2z_9 - z_2z_4z_5z_7 + z_3z_4^2z_7},\  x_5,\
        x_6 + \frac{z_5z_{12}}{z_4z_{10} + z_5z_{11}},\
        \\&x_7 + \frac{z_1z_5z_{16}}{z_1z_5z_{14} + z_2z_5z_{13} - z_3z_4z_{13}},\
        x_8 + \frac{z_1z_5z_{16}}{z_1z_5z_{14} + z_2z_5z_{13} - z_3z_4z_{13}}x_9
        \ \ \Big\}.
    \end{align*}%
    Because this is not equal to $\{1\}$, we conclude $h$ is a valid inequality for $\Delta(T)$.
\end{example}

The strategy of this example turns out to be sufficient to carry out \cref{algorithm:tensor algorithm randomized} for each of the moment polytopes of family 4 and 5.

In family 4 we find that all tensors have generic polytope, except for the $\GL$-orbits of three tensors: $\Tdet := e_1 \wedge e_2 \wedge e_3,\ \Tdet + e_{111}$ and $\Tdet + \TW$ with $\TW \coloneqq e_{112} + e_{121} + e_{211}$.\footnote{The notation $\Tdet$ is used because the tensor is known as a \emph{determinant tensor} (namely, $(A,A,A)\cdot \Tdet = \det(A) \Tdet$ for any matrix $A$). The notation $\TW$ is used because the tensor is known in quantum information as the $W$-state.}
In family 5, consisting of $25$ $\Sl$-unstable tensors (up to permutation of the subsystems), every tensor has a distinct moment polytope.
The complete list of vertices for all these polytopes is given in
\cref{table:333 vertex data}, as well as online in \cite{vandenBerg2025momentPolytopesGithub}.

\begin{theorem}
\label{thm:family4 moment polytope computation}
    For any tensor $\familyparam \Tdet + T_{\textnormal{u}}$ in family 4, we have that $\Delta(\familyparam \Tdet + T_{\textnormal{u}}) = \Delta(\Tdet + T_{\textnormal{u}})$.
    Moreover,
    \begin{align}
        \Delta(\Tdet) \ \subsetneq\  \Delta(\Tdet + e_{111}) \ \subsetneq\  \Delta(\Tdet + \TW) \ \subsetneq\  \kronpol{3}{3}{3},
    \end{align}
    and $\Delta(\Tdet + T_{\textnormal{u}}) = \kronpol{3}{3}{3}$ for all $T_{\textnormal{u}} \in \mathcal U_4 \setminus \{0,e_{111},\TW\}$.
\end{theorem}
\begin{proof}
For the first claim, we prove that $\Tdet + T_{\textnormal{u}} \in \G \cdot (\familyparam\Tdet + T_{\textnormal{u}})$ for any $\familyparam \neq 0$.
By \cref{lem:c333 classification structural properties}, we know $\Tdet + \familyparam^{-1}  \Sl \cdot T_{\textnormal{u}} \subseteq \Sl \cdot (\Tdet + \familyparam^{-1} T_{\textnormal{u}})$.
Observe also that $\Sl \cdot (\Tdet + \familyparam^{-1} T_{\textnormal{u}}) \subseteq \G \cdot (\familyparam \Tdet + T_{\textnormal{u}})$.
Hence it is sufficient to find for each $T_{\textnormal{u}} \in \mathcal U_4$ a group element $(A,B,C) \in \Sl$ such that $(A \ot B \ot C) T_{\textnormal{u}} = \familyparam T_{\textnormal{u}}$. For $T_{\textnormal{u}} = 0$ any element suffices trivially. The group elements for the other tensors are given in the following table:
\begin{center}
\begin{tabular}{ll@{}r@{,}r@{,}r@{}l}\toprule
\multicolumn{1}{c}{Tensors $T_{\textnormal{u}} \in \mathcal U_4$} & \multicolumn{5}{c}{Scaling $\Sl$ element} \\\cmidrule(lr){1-6}%
    $e_{113} + e_{131} + e_{311} + e_{222}$ &
    $\bigg($& $\left[\begin{smallmatrix} 1 & 0 & 0 \\ 0 & \familyparam^3 & 0 \\ 0 & 0 & \familyparam^{-3} \end{smallmatrix}\right]$&
    $\left[\begin{smallmatrix} 1 & 0 & 0 \\ 0 & \familyparam^3 & 0 \\ 0 & 0 & \familyparam^{-3} \end{smallmatrix}\right]$&
    $\left[\begin{smallmatrix} \familyparam^4 & 0 & 0 \\ 0 & \familyparam^{-5} & 0 \\ 0 & 0 & \familyparam \end{smallmatrix}\right]$& $\bigg)$ \\
    $e_{113} + e_{122} + e_{131} + e_{212} + e_{221} + e_{311}$ &
    $\bigg($& $\left[\begin{smallmatrix} \familyparam & 0 & 0 \\ 0 & 1 & 0 \\ 0 & 0 & \familyparam^{-1} \end{smallmatrix}\right]$&
    $\left[\begin{smallmatrix} \familyparam & 0 & 0 \\ 0 & 1 & 0 \\ 0 & 0 & \familyparam^{-1} \end{smallmatrix}\right]$&
    $\left[\begin{smallmatrix} \familyparam & 0 & 0 \\ 0 & 1 & 0 \\ 0 & 0 & \familyparam^{-1} \end{smallmatrix}\right]$& $\bigg)^{\phantom{l}}$
    \\
    $\TW$, $e_{111}$  &
    $\bigg($& $\left[\begin{smallmatrix} \familyparam & 0 & 0 \\ 0 & \familyparam & 0 \\ 0 & 0 & \familyparam^{-2} \end{smallmatrix}\right]$&
    $I_3$& $\qquad I_3
    $& $\bigg)^{\phantom{l}} $\\
\bottomrule
\end{tabular}
\end{center}
For the second claim, we need to compute the moment polytopes of only a finite number of tensors. This we do using \cref{algorithm:tensor algorithm randomized} as in \cref{theorem:symbolic algorithm randomized} (combined with the strategy of \cref{example:symbolic groebner derandomization} to obtain fast execution).
\end{proof}

\begin{remark}
The tensor $\Tdet + e_{111}$ was studied in~\cite{christandlTensorNetworkRepresentations2020a} (where it is denoted by $\lambda$). There, it is shown that $\MM_2$ degenerates, but does not restrict, to $\Tdet + e_{111}$.
\end{remark}

Additional information about the computed moment polytopes are given in the remarks below and their corresponding tables/figures.

\begin{remark}[Free, tight and oblique tensors]
\label{remark:free tensors}
In~\cref{table:c333 unstable tensors info} we list whether each tensor it is \emph{free} or \emph{tight}.
A tensor $T \in \C^a \ot \C^b \ot \C^c$ said to have free support if for $\Gamma = \{ (i,j,k) \mid  T_{i,j,k} \neq 0 \}$, any choice of two distinct elements $(i,j,k), (i',j',k') \in \Gamma$ differ in at least two coordinates.
It is said to have tight support if there exist injective functions $f\colon [a] \to \Z$, $g\colon [b] \to \Z$, $h\colon [c] \to \Z$ such that $f(i) + g(j) + h(k) = 0$ for all $(i,j,k) \in \Gamma$.
It is said to have \emph{oblique} support if~$\Gamma$ is antichain, that is, no two elements are comparable under the partial ordering on~$[a] \times [b] \times [c]$ induced by the total orders on~$[a], [b], [c]$.
A tensor $T$ is called free (respectively, oblique, tight) if there exists some $g \in \G$ such that $g \cdot T$ has free (respectively, oblique, tight) support.
All tight tensors are oblique, and all oblique tensors are free.
Equivalently, a tensor is tight if its stabilizer contains a regular semisimple element, see e.g.~\cite[Sec.~2.1]{connerGeometricApproachStrassen2021}, but no such geometric characterization is known for oblique or free tensors.
For the specific case of~$3 \times 3 \times 3$-tensors, it is also known that oblique supports are tight~\cite[Sec.~2.5]{connerGeometricApproachStrassen2021}.

In~\cite{connerGeometricApproachStrassen2021} the dimensions of the (Zariski closure of) the sets of free, tight and oblique tensors are computed.
In~$\C^3 \ot \C^3 \ot \C^3$, the set of free tensors is dense, but this fails in larger balanced formats.
In~\cite{vandenBerg2025nonFreeTensor} we prove that tensors $2$ and $5$ from family $5$ (\cref{table:c333 unstable tensors info}) are not free. We also observe that tensors~$3$ and~$4$ are not tight, through an explicit computation of their stabilizers.
\end{remark}

\begin{remark}[Special $\Sl$-unstable tensors]
\label{remark:special tensors}
We have identified several tensors in~\cref{table:c333 unstable tensors info} which arise in other contexts.
Tensors $20$, $22$, $23$, $24$ are easily identified as the rank-$2$ unit tensor $\unit{2}$ and the matrix multiplication tensors $\MM_{1,1,3}$, $\MM_{1,1,2}$, $\MM_{1,1,1}$, respectively
(here $\MM_{a,b,c} \in \C^{ab} \ot \C^{bc} \ot \C^{ac}$ is the tensor describing the multiplication of an $a\times b$ matrix with a $b \times c$ matrix).
Tensor $21$ (which we denote by $\TW$) is equivalent to every concise $\Sl$-unstable tensor in $\C^2 \ot \C^2 \ot \C^2$.

Some of the tensors can be described as algebra-tensors.
If $A$ is a finite-dimensional algebra over $\C$, then its multiplication $A \times A \to A$ is a bilinear map and hence can be interpreted as a tensor in $A^* \otimes A^* \otimes A$.
For instance, the unit tensor $\unit{n}$ describes the algebra $\C^n$ with element-wise multiplication (after identifying $(\C^n)^*$ with $\C^n$).
The $3$-dimensional associative unital algebras, as well as the algebra degenerations between them (which is equivalent to degeneration for the multiplication tensor~\cite{DBLP:conf/mfcs/BlaserL16}), were classified by Gabriel~\cite{gabrielFiniteRepresentationType1975}.
We observe that tensors $8$, $10$, $11$ and $14$ are of this form, with the algebras being given by $\C[x]/(x^3)$, the upper-triangular $2 \times 2$ matrices, and $\C[x,y]/(x^2,y^2,xy)$ respectively.
The only other $3$-dimensional associative unital algebra is $\C \oplus \C \oplus \C$, whose multiplication tensor is the unit tensor $\unit{3}$.
The support functionals on these tensors were also computed by Strassen~\cite{strassen1991}.
Tensor $17$ was also studied by Strassen~\cite{strassen1991}, and is referred to as the \emph{truncated null-algebra} or \emph{Strassen's tensor} \cite{blaser2013fast}.

Tensor $10$ is also known as the \emph{cap-set tensor} (over $\F_3$), and the computation of its asymptotic slice rank played an important role for the cap-set problem~\cite{taoCapsetBlogPost2016,taoSawinCapsetBlogPost2016,costaGapSliceRank2021}.
\end{remark}

\begin{remark}[Inclusion relations]
\label{remark:inclusion relations}
The inclusion relations of all the moment polytopes of tensors in $\C^3\ot\C^3\ot\C^3$ are given in \cref{fig:c333-unstable-moment-polytope-inclusions}.
Many of the inclusions between the moment polytopes of $\SL$-unstable tensors follow from the existence of degenerations between these tensors (via an application of \cref{proposition:degeneration monotone}), which were determined by Nurmiev \cite{nurmievClosuresNilpotentOrbits2000}.
Among these, we discover two new inclusions which do not have a corresponding degeneration.

Three other moment polytopes come from tensors in family~$4$.
We find the following ``generating'' relations to the polytopes of these tensors:
(1) $\Delta(\Tdet)$ contains~$\Delta(\Tnurmiev_{17})$,
(2) $\Delta(\Tdet + e_{111})$ contains~$\Delta(\Tnurmiev_{10})$, and (3) $\Delta(\Tdet + \TW)$ contains $\Delta(\Tnurmiev_{7})$ and~$\Delta(\Tnurmiev_{8})$.
We know that the inclusion in (1) can be realized through a degeneration~$\Tdet \degengeq \Tnurmiev_{17}$.
Based on numerical experiments we expect that~$\Tdet + e_{111} \degengeq \Tnurmiev_{10}$ and $\Tdet + \TW \degengeq \Tnurmiev_{7}$, but~$
\Tdet + \TW \not\degengeq \Tnurmiev_{8}$.

Finally, one notable tensor that has the Kronecker polytope as its moment polytope is $\unit{3}$.
We have that $\unit{3} \not\degengeq \Tdet+W$: namely, we know $\Tdet + W \notin \G \cdot \unit{3}$ (e.g.\ because the moment polytopes are different), after which this follows from \cref{lemma:orbit border} (noting that $\unit{3}$ is $\Sl$-polystable by \cref{theorem:tensor kempf-ness SL} and $\Tdet + W$ is $\Sl$-semistable).
It is also known that $\unit{3} \not\degengeq \Tnurmiev_1$, but $\unit{3} \degengeq \Tnurmiev_8$ and $\unit{3} \degengeq \Tnurmiev_9$ \cite{MR3239293}.
\end{remark}

\begin{remark}[Quantum functionals]
\label{remark:quantum functionals}
In \cref{table:c333 unstable tensors polytope data} we provide (for the $\Sl$-unstable tensors) data on the minimum-norm points in the moment polytope as well as several values of the \emph{quantum functionals} \cite{christandlUniversalPointsAsymptotic2021}.
The quantum functionals are defined as $F_\theta(T) \coloneqq 2^{E_\theta(T)}$, where
$E_\theta(T) \coloneqq \max_{p \in \Delta(T)} \theta_1 H(p_1) + \theta_2 H(p_2) + \theta_3 H(p_3)$, with $\theta = (\theta_1,\theta_2,\theta_3)$ any probability distribution and $H(p)$ the (base-$2$) entropy of a probability distribution $p$.
The values of $\theta$ are chosen for illustrative purposes: we may compute $F^{\theta}$ for any $\theta$ via any algorithm for concave optimization over our computed polytopes.
Note that for all $\Sl$-semistable tensors, the uniform point is included in the moment polytope, and $F^\theta$ evaluates to $3$ for any $\theta$.

The last column of the table gives the minimum of $F_\theta$ over probability distributions $\theta$, which is known to equal the base-$2$ logarithm of the asymptotic slice rank of the tensor~\cite{christandlUniversalPointsAsymptotic2021}. When the tensor is tight, asymptotic slice rank agrees with the asymptotic subrank~\cite{christandlUniversalPointsAsymptotic2021}.

The quantum functionals also form obstructions for the existence of asymptotic restrictions between tensors: if $F^{\theta}(T) \ngeq F^{\theta}(S)$ for some $\theta$, then $T$ does not asymptotically restrict to $S$ \cite{christandlUniversalPointsAsymptotic2021}.
Hence our polytope data can be used to rule out existence of asymptotic restrictions.
For instance, from~\cref{table:c333 unstable tensors polytope data} we see that tensors $10$ and $11$ are asymptotically incomparable (even after possibly permuting their subsystems).
\end{remark}

\begin{table}
\def\arraystretch{1.1}
\setlength\tabcolsep{0.4em}
\footnotesize
\makebox[\textwidth][c]{
\begin{tabular}{rlccll}\toprule
No.\ & \multicolumn{1}{c}{Tensor} & Shape     &   Free/tight       & \multicolumn{1}{c}{Algebra}  & \multicolumn{1}{c}{Name}      \\
\cmidrule(lr){1-6}
1   & $e_{123} + e_{132} + e_{213} + e_{222} + e_{231} + e_{311}$ & $(3,3,3)$ & tight             &                                                  \\
2   & $e_{123} + e_{132} + e_{213} + e_{221} + e_{222} + e_{311}$ & $(3,3,3)$ & not free &                                                  \\
3   & $e_{113} + e_{122} + e_{131} + e_{212} + e_{223} + e_{311}$ & $(3,3,3)$ & free   &                                                  \\
4   & $e_{113} + e_{122} + e_{212} + e_{221} + e_{331}$         & $(3,3,3)$ & free   &                                                  \\
5   & $e_{113} + e_{131} + e_{132} + e_{221} + e_{312}$         & $(3,3,3)$ & not free &                                                  \\
6   & $e_{113} + e_{122} + e_{212} + e_{231} + e_{321}$         & $(3,3,3)$ & tight             &                                                  \\
7   & $e_{113} + e_{122} + e_{131} + e_{212} + e_{321}$         & $(3,3,3)$ & tight             &                                                  \\
8   & $e_{113} + e_{131} + e_{222} + e_{311}$                 & $(3,3,3)$ & tight             &  $\C \oplus \C[x]/(x^2)$ \\ %
9   & $e_{111} + e_{122} + e_{222} + e_{233}$                 & $(2,3,3)$ & tight             & \\ %
10  & $e_{113} + e_{122} + e_{131} + e_{212} + e_{221} + e_{311}$ & $(3,3,3)$ & tight             &
$\C[x]/(x^3)$ & Cap-set tensor \\ %
11  & $e_{113} + e_{131} + e_{212} + e_{321}$                 & $(3,3,3)$ & tight             &
$\{\left[\begin{smallmatrix} \alpha & \beta \\0 & \gamma \end{smallmatrix}\right] \mid \alpha,\beta,\gamma \in \C\}$ \\
12  & $e_{113} + e_{131} + e_{211} + e_{222}$                 & $(2,3,3)$ & tight             &                                                  \\
13  & $e_{113} + e_{122} + e_{131} + e_{212} + e_{221}$         & $(2,3,3)$ & tight             &                                                  \\ %
14  & $e_{113} + e_{121} + e_{132} + e_{211} + e_{312}$         & $(3,3,3)$ & tight             & $\C[x,y]/(x^2,y^2,xy)$  & Null-algebra   \\
15  & $e_{122} + e_{133} + e_{211}$                         & $(2,3,3)$ & tight             & & $\unit{1} \oplus \MM_{2,1,1}$             \\ %
16  & $e_{113} + e_{122} + e_{131} + e_{211}$                 & $(2,3,3)$ & tight             &                                                  \\ %
17  & $e_{112} + e_{121} + e_{213} + e_{231}$                 & $(2,3,3)$ & tight             & & Truncated null-algebra                               \\ %
18  & $e_{111} + e_{122} + e_{212} + e_{223}$                 & $(2,2,3)$ & tight             & & \\ %
19  & $e_{113} + e_{121} + e_{212}$                         & $(2,2,3)$ & tight             &                                                  \\
20  & $e_{111} + e_{222}$                                 & $(2,2,2)$ & tight             & $\C^2$ & $\unit{2}$                                       \\
21  & $e_{112} + e_{121} + e_{211}$                         & $(2,2,2)$ & tight             & $\C[x]/(x^2)$ & $\TW$                   \\
22  & $e_{111} + e_{122} + e_{133}$                         & $(1,3,3)$ & tight             &  & $\MM_{1,1,3}$                          \\
23  & $e_{111} + e_{122}$                                 & $(1,2,2)$ & tight             &  & $\MM_{1,1,2}$                          \\
24  & $e_{111}$                                         & $(1,1,1)$ & tight             & $\C$ & $\unit{1}$\\
25 & 0 & $(0,0,0)$ & tight \\ \bottomrule
\end{tabular}}
\caption{
Representatives of the $\G$-orbits of the $\Sl$-unstable tensors in~$\C^3 \ot \C^3 \ot \C^3$ (up to permutation of the systems) and additional information.
The finite set $\mathcal U_5$ from \cref{subsubsection:333-classification} is equal to this set of tensors and their cyclic permutations.
The fourth column indicates whether the tensor is free or tight, where the label ``free'' indicates it is free but not tight (see \cref{remark:free tensors} for definitions).
The last two columns contains additional information about some tensors, as explained in \cref{remark:special tensors}.
}
\label{table:c333 unstable tensors info}
\end{table}

\begin{figure}
    \vspace{-2em}
    \centering
\makebox[\textwidth][c]{
    \begin{tikzpicture}[
        node/.style = {circle, draw, fill=red!20, minimum size=7mm, inner sep=1pt},
        edgedotted/.style = {-Stealth,shorten >=10pt, shorten <=10pt,
                      },
        nodesq/.style = {draw, fill=blue!20, minimum size=7mm, inner sep=1pt},
        nodesqg/.style = {draw, fill=green!20, minimum size=7mm, inner sep=1pt},
        >=stealth, shorten >=1pt, auto,
        rotate=-90,xscale=0.9,yscale=1.2,every node/.style={scale=0.95}
      ]
      \node[node] (1) at (0, 0)   {$\Tnurmiev_1$};
      \node[node] (2) at (0, 1)   {$\Tnurmiev_2$};
      \node[node] (3) at (-1, 2)  {$\Tnurmiev_3$};
      \node[node] (4) at (1, 2)   {$\Tnurmiev_4$};
      \node[node] (5) at (-1, 3)  {$\Tnurmiev_5$};
      \node[node] (6) at (1, 3)   {$\Tnurmiev_6$};
      \node[node] (7) at (-2, 4)  {$\Tnurmiev_7$};
      \node[node] (8) at (0, 4)   {$\Tnurmiev_8$};
      \node[node] (9) at (2, 4)   {$\Tnurmiev_9$};
      \node[node] (10) at (-2, 5) {$\Tnurmiev_{10}$};
      \node[node] (11) at (0, 5)  {$\Tnurmiev_{11}$};
      \node[node] (12) at (2, 5)  {$\Tnurmiev_{12}$};
      \node[node] (13) at (-2, 6) {$\Tnurmiev_{13}$};
      \node[node] (14) at (0, 6)  {$\Tnurmiev_{14}$};
      \node[node] (15) at (2, 6)  {$\Tnurmiev_{15}$};
      \node[node] (16) at (0, 7)  {$\Tnurmiev_{16}$};
      \node[node] (17) at (-2, 7) {$\Tnurmiev_{17}$};
      \node[node] (18) at (2, 7)  {$\Tnurmiev_{18}$};
      \node[node] (19) at (-1, 8) {$\Tnurmiev_{19}$};
      \node[node] (20) at (0, 9)  {$\Tnurmiev_{20}$};
      \node[node] (21) at (0, 10) {$\Tnurmiev_{21}$};
      \node[node] (22) at (1, 8)  {$\Tnurmiev_{22}$};
      \node[node] (23) at (0, 11) {$\Tnurmiev_{23}$};
      \node[node] (24) at (0, 12) {$\Tnurmiev_{24}$};
      \node[node] (25) at (0, 13) {$\Tnurmiev_{25}$};
      \node[nodesq] (gen) at (-3.3, 0) {$\,\unit{3}\,$};
      \node[nodesq] (DW) at (-3.3, 2.3) {$\,\Tdet+\TW\,$};
      \node[nodesq] (De111) at (-3.3, 4.4) {$\,\Tdet+e_{111}\,$};
      \node[nodesq] (D) at (-3.3, 6) {$\Tdet$};
      \path[->] (1) edge (2);
      \path[->] (2) edge (3) edge (4);
      \path[->] (3) edge (5) edge (6);
      \path[->] (4) edge (5) edge (6);
      \path[->] (5) edge (7) edge (8) edge (9);
      \path[->] (6) edge (7);
      \path[->,dashed] (6) edge (8);
      \path[->] (7) edge (10) edge (11) edge (12);
      \path[->] (8) edge (10) edge (12);
      \path[->,dashed] (8) edge (11);
      \path[->] (9) edge (12);
      \path[->] (10) edge (13) edge (14);
      \path[->] (11) edge (14) edge (15);
      \path[->] (12) edge (13) edge (15);
      \path[->] (13) edge (16) edge (17) edge (18);
      \path[->] (14) edge (16) edge (17);
      \path[->] (15) edge (16);
      \path[->] (16) edge (19) edge (22);
      \path[->] (17) edge (19);
      \path[->] (18) edge (19);
      \path[->] (19) edge (20);
      \path[->] (20) edge (21);
      \path[->] (21) edge (23);
      \path[->, bend right=10] (22) edge (23);
      \path[->] (23) edge (24);
      \path[->] (24) edge (25);
      \path[->]  (DW) edge (De111);
      \path[->]  (De111) edge (D);
      \draw[arrows={->[scale=1,black,width=1pt]},dash pattern=on 1pt off 1pt]  (DW) -- (7);
      \draw[arrows={->[scale=1,black,width=1pt]},dash pattern=on 1pt off 1pt]  (DW) -- (8);
      \draw[arrows={->[scale=1,black,width=1pt]},dash pattern=on 1pt off 1pt]  (De111) -- (10);
      \path[->]  (D) edge (17);
      \path[->,dashed] (gen) edge (1);
      \path[->,dashed] (gen) edge (DW);
    \end{tikzpicture}}
    \caption{Overview of inclusions among the moment polytopes of tensors in $\C^3 \ot \C^3 \ot \C^3$, up to cyclic permutations of the factors.
        Nodes are labeled by representative tensors with this moment polytope.
        The tensor $\Tnurmiev_i$ is tensor $i$ from \cref{table:c333 unstable tensors info}.
        The moment polytope of $\unit{3}$ is the Kronecker polytope.
        The same is true for any other tensor not equivalent to the ones in the diagram.
        Square and circular nodes contain $\Sl$-semistable and $\Sl$-unstable tensors respectively.
        An arrow is drawn from polytope $P$ to polytope $Q$ if $P \supseteq Q'$ for a $Q'$ that can be obtained from $Q$ by a permutation of the three factors.
        It is dashed if there is an inclusion of moment polytopes (as above) but no degeneration between the corresponding tensors (for all permutations of the factors).
        It is dotted if we do not know whether such a degeneration exists.
        See \cref{remark:inclusion relations} for more details.
    }
    \label{fig:c333-unstable-moment-polytope-inclusions}
\end{figure}

\begin{table}[H]
\centering
\renewcommand{\arraystretch}{1.2}
\footnotesize
\makebox[\linewidth]{
\begin{tabular}{rCCc@{}c@{}c@{}c@{\ }c@{\ }c@{}c@{}c@{\ }c@{\ }c@{}c@{}c@{}clllll}\toprule
No. & vertices & ineqs. & \multicolumn{13}{c}{Minimum-norm point} & $F_{(\frac13,\frac13,\frac13)}$ & $F_{(\frac12,\frac12,0)}$ & $F_{(\frac12,0,\frac12)}$ & $F_{(0,\frac12,\frac12)}$ & $\min_{\theta} F_{\theta}$ \\\cmidrule(lr){1-21}
1   & 38       & 46           & $($ & $\frac{5}{13},$  & $ \frac{9}{26},$ & $ \frac{7}{26}$ & $\sep$ & $\frac{29}{78},$ & $ \frac{1}{3},$  & $ \frac{23}{78}$ & $\sep$ & $\frac{29}{78},$ & $ \frac{1}{3},$ & $ \frac{23}{78}$ & $)$  & $2.9806$ & $3{\color{gray}.0000}$ & $3{\color{gray}.0000}$ & $3{\color{gray}.0000}$ & $2.9798$   \\
2   & 36       & 46           & $($ & $\frac{17}{42},$ & $\frac{1}{3},$   & $\frac{11}{42}$ & $\sep$ & $\frac{17}{42},$ & $\frac{1}{3},$   & $\frac{11}{42}$  & $\sep$ & $\frac{5}{14},$  & $\frac{5}{14},$ & $\frac{2}{7}$    & $)$  & $2.9643$ & $3{\color{gray}.0000}$ & $3{\color{gray}.0000}$ & $3{\color{gray}.0000}$ & $2.9622$  \\
3   & 43       & 49           & $($ & $\frac{7}{18},$  & $\frac{7}{18},$  & $\frac{2}{9}$   & $\sep$ & $\frac{7}{18},$  & $\frac{7}{18},$  & $\frac{2}{9}$    & $\sep$ & $\frac{4}{9},$   & $\frac{5}{18},$ & $\frac{5}{18}$   & $)$  & $2.9156$ & $3{\color{gray}.0000}$ & $3{\color{gray}.0000}$ & $3{\color{gray}.0000}$ & $2.9154$  \\
4   & 53       & 52           & $($ & $\frac{2}{5},$   & $\frac{3}{10},$  & $\frac{3}{10}$  & $\sep$ & $\frac{2}{5},$   & $\frac{3}{10},$  & $\frac{3}{10}$   & $\sep$ & $\frac{13}{30},$ & $\frac{1}{3},$  & $\frac{7}{30}$   & $)$  & $2.9508$ & $3{\color{gray}.0000}$ & $3{\color{gray}.0000}$ & $3{\color{gray}.0000}$ & $2.9476$ \\
5   & 47       & 50           & $($ & $\frac{13}{30},$ & $\frac{1}{3},$   & $\frac{7}{30}$  & $\sep$ & $\frac{13}{30},$ & $\frac{1}{3},$   & $\frac{7}{30}$   & $\sep$ & $\frac{2}{5},$   & $\frac{2}{5},$  & $\frac{1}{5}$    & $)$  & $2.8980$ & $3{\color{gray}.0000}$ & $3{\color{gray}.0000}$ & $3{\color{gray}.0000}$ & $2.8979$  \\
6   & 57       & 54           & $($ & $\frac{2}{5},$   & $\frac{13}{35},$ & $\frac{8}{35}$  & $\sep$ & $\frac{2}{5},$   & $\frac{13}{35},$ & $\frac{8}{35}$   & $\sep$ & $\frac{16}{35},$ & $\frac{2}{7},$  & $\frac{9}{35}$   & $)$  & $2.9143$ & $3{\color{gray}.0000}$ & $3{\color{gray}.0000}$ & $3{\color{gray}.0000}$ & $2.9130$ \\
7   & 47       & 51           & $($ & $\frac{3}{7},$   & $\frac{2}{7},$   & $\frac{2}{7}$   & $\sep$ & $\frac{10}{21},$ & $\frac{1}{3},$   & $\frac{4}{21}$   & $\sep$ & $\frac{10}{21},$ & $\frac{1}{3},$  & $\frac{4}{21}$   & $)$ & $2.8595$ & $3{\color{gray}.0000}$ & $3{\color{gray}.0000}$ & $3{\color{gray}.0000}$ & $2.8536$  \\
8   & 52       & 51           & $($ & $\frac{3}{7},$   & $\frac{5}{14},$  & $\frac{3}{14}$  & $\sep$ & $\frac{3}{7},$   & $\frac{5}{14},$  & $\frac{3}{14}$   & $\sep$ & $\frac{3}{7},$   & $\frac{5}{14},$ & $\frac{3}{14}$   & $)$   & $2.8899$ & $3{\color{gray}.0000}$ & $3{\color{gray}.0000}$ & $3{\color{gray}.0000}$ & $2.8899$ \\
9   & 18       & 25           & $($ & $\frac{1}{2},$   & $\frac{1}{2},$   & $0$             & $\sep$ & $\frac{1}{3},$   & $\frac{1}{3},$   & $\frac{1}{3}$    & $\sep$ & $\frac{1}{3},$   & $\frac{1}{3},$  & $\frac{1}{3}$    & $)$ & $2.6207$ & $2.4495$ & $2.4495$ & $3{\color{gray}.0000}$ & $2{\color{gray}.0000}$ \\
10  & 29       & 46           & $($ & $\frac{1}{2},$   & $\frac{1}{3},$   & $\frac{1}{6}$   & $\sep$ & $\frac{1}{2},$   & $\frac{1}{3},$   & $\frac{1}{6}$    & $\sep$ & $\frac{1}{2},$   & $\frac{1}{3},$  & $\frac{1}{6}$    & $)$  & $2.7551$ & $3{\color{gray}.0000}$ & $3{\color{gray}.0000}$ & $3{\color{gray}.0000}$ & $2.7551$ \\
11  & 31       & 33           & $($ & $\frac{2}{5},$   & $\frac{3}{10},$  & $\frac{3}{10}$  & $\sep$ & $\frac{1}{2},$   & $\frac{3}{10},$  & $\frac{1}{5}$    & $\sep$ & $\frac{1}{2},$   & $\frac{3}{10},$ & $\frac{1}{5}$    & $)$ & $2.8567$ & $3{\color{gray}.0000}$ & $3{\color{gray}.0000}$ & $2.8284$ & $2.8284$ \\
12  & 23       & 28           & $($ & $\frac{6}{11},$  & $\frac{5}{11},$  & $0$             & $\sep$ & $\frac{4}{11},$  & $\frac{4}{11},$  & $\frac{3}{11}$   & $\sep$ & $\frac{4}{11},$  & $\frac{4}{11},$ & $\frac{3}{11}$   & $)$ & $2.6030$ & $2.4495$ & $2.4495$ & $3{\color{gray}.0000}$ & $2{\color{gray}.0000}$  \\
13  & 17       & 26           & $($ & $\frac{5}{9},$   & $\frac{4}{9},$   & $0$             & $\sep$ & $\frac{4}{9},$   & $\frac{1}{3},$   & $\frac{2}{9}$    & $\sep$ & $\frac{4}{9},$   & $\frac{1}{3},$  & $\frac{2}{9}$    & $)$  & $2.5522$ & $2.4495$ & $2.4495$ & $3{\color{gray}.0000}$ & $2{\color{gray}.0000}$  \\
14  & 20       & 25           & $($ & $\frac{5}{9},$   & $\frac{2}{9},$   & $\frac{2}{9}$   & $\sep$ & $\frac{5}{9},$   & $\frac{2}{9},$   & $\frac{2}{9}$    & $\sep$ & $\frac{4}{9},$   & $\frac{4}{9},$  & $\frac{1}{9}$    & $)$ & $2.6866$ & $2.8284$ & $3{\color{gray}.0000}$ & $3{\color{gray}.0000}$ & $2.6834$ \\
15  & 15       & 23           & $($ & $\frac{3}{5},$   & $\frac{2}{5},$   & $0$             & $\sep$ & $\frac{2}{5},$   & $\frac{3}{10},$  & $\frac{3}{10}$   & $\sep$ & $\frac{2}{5},$   & $\frac{3}{10},$ & $\frac{3}{10}$   & $)$ & $2.5874$ & $2.4142$ & $2.4142$ & $3{\color{gray}.0000}$ & $2{\color{gray}.0000}$ \\
16  & 13       & 21           & $($ & $\frac{2}{3},$   & $\frac{1}{3},$   & $0$             & $\sep$ & $\frac{1}{2},$   & $\frac{1}{3},$   & $\frac{1}{6}$    & $\sep$ & $\frac{1}{2},$   & $\frac{1}{3},$  & $\frac{1}{6}$    & $)$  & $2.4361$ & $2.4142$ & $2.4142$ & $3{\color{gray}.0000}$ & $2{\color{gray}.0000}$ \\
17  & 13       & 18           & $($ & $\frac{1}{2},$   & $\frac{1}{2},$   & $0$             & $\sep$ & $\frac{1}{2},$   & $\frac{1}{4},$   & $\frac{1}{4}$    & $\sep$ & $\frac{1}{2},$   & $\frac{1}{4},$  & $\frac{1}{4}$    & $)$  & $2.5198$ & $2.4495$ & $2.4495$ & $2.8284$ & $2{\color{gray}.0000}$ \\
18  & 9        & 12           & $($ & $\frac{1}{2},$   & $\frac{1}{2},$   & $0$             & $\sep$ & $\frac{1}{2},$   & $\frac{1}{2},$   & $0$              & $\sep$ & $\frac{1}{3},$   & $\frac{1}{3},$  & $\frac{1}{3}$    & $)$  & $2.2894$ & $2{\color{gray}.0000}$ & $2.4495$ & $2.4495$ & $2{\color{gray}.0000}$  \\
19  & 8        & 13           & $($ & $\frac{3}{5},$   & $\frac{2}{5},$   & $0$             & $\sep$ & $\frac{3}{5},$   & $\frac{2}{5},$   & $0$              & $\sep$ & $\frac{2}{5},$   & $\frac{2}{5},$  & $\frac{1}{5}$    & $)$  & $2.2300$ & $2{\color{gray}.0000}$ & $2.4142$ & $2.4142$ & $2{\color{gray}.0000}$ \\
20  & 5        & 6            & $($ & $\frac12,$       & $ \frac12,$      & $ 0$            & $\sep$ & $ \frac12,$      & $ \frac12,$      & $ 0$             & $\sep$ & $ \frac12,$      & $ \frac12,$     & $ 0$             & $)$  & $2{\color{gray}.0000}$ & $2{\color{gray}.0000}$ & $2{\color{gray}.0000}$ & $2{\color{gray}.0000}$ & $2{\color{gray}.0000}$ \\
21  & 4        & 4            & $($ & $\frac23,$       & $ \frac13,$      & $ 0$            & $\sep$ & $ \frac23,$      & $ \frac13,$      & $ 0$             & $\sep$ & $ \frac23,$      & $ \frac13,$     & $ 0$             & $)$  & $1.8899$ & $2{\color{gray}.0000}$ & $2{\color{gray}.0000}$ & $2{\color{gray}.0000}$ & $1.8899$ \\
22  & 3        & 3            & $($ & $1,$             & $ 0,$            & $ 0$            & $\sep$ & $ \frac13,$      & $\frac13,$       & $\frac13$        & $\sep$ & $ \frac13,$      & $\frac13,$      & $\frac13$        & $)$ & $2.0801$ & $1.7321$ & $1.7321$ & $3{\color{gray}.0000}$ & $1{\color{gray}.0000}$ \\
23  & 2        & 2            & $($ & $1,$             & $ 0,$            & $ 0$            & $\sep$ & $ \frac12,$      & $ \frac12,$      & $ 0$             & $\sep$ & $ \frac12,$      & $ \frac12,$     & $ 0$             & $)$ & $1.5874$ & $1.4142$ & $1.4142$ & $2{\color{gray}.0000}$ & $1{\color{gray}.0000}$ \\
24  & 1        & 0            & $($ & $1,$             & $ 0,$            & $ 0$            & $\sep$ & $ 1,$            & $ 0,$            & $ 0$             & $\sep$ & $ 1,$            & $ 0,$           & $ 0$             & $)$ & $1{\color{gray}.0000}$ & $1{\color{gray}.0000}$ & $1{\color{gray}.0000}$ & $1{\color{gray}.0000}$ & $1{\color{gray}.0000}$
\\\bottomrule
\end{tabular}
}
\caption{Data on the moment polytopes of the non-zero $\Sl$-unstable tensors in $\C^3 \ot \C^3 \ot \C^3$ (see \cref{table:c333 unstable tensors info}). We list the minimum-norm point in the polytope measured with respect to the $\ell^2$-norm (equivalently, the point closest to the uniform point in the respective format), which is unique. We also give numerical evaluations of the quantum functionals $F_\theta$~\cite{christandlUniversalPointsAsymptotic2021} for certain parameters $\theta$, see
 \cref{remark:quantum functionals}.}
\label{table:c333 unstable tensors polytope data}
\end{table}

\newpage

\section{A probabilistic approach for \texorpdfstring{$4\times4\times4$}{4 x 4 x 4} tensors}
\label{section:algorithms for 4x4x4}

At this point we are able to enumerate a sufficient set of inequalities $\ineqs$ as in \cref{lemma:tensor attainability} for $\C^4 \ot \C^4 \ot \C^4$ using our optimized enumeration algorithm (\cref{algorithm:tensor weight-matrix-enumeration-no-symmetries}).
However, running the many required Gr\"obner basis computations (typically around one million of them per tensor) to determine attainability as in \cref{section:algorithms for computing tensor moment polytopes} presents an obstruction for computing moment polytopes for tensors of this format.
One sizable difficulty that arises when computing Gr\"obner bases is potential coefficient swell of the intermediate polynomials, which leads to memory and performance issues.
This problem can occur even when the starting polynomials and resulting Gr\"obner basis are simple. See \cite{arnoldModularAlgorithmsComputing2003} for an example with Buchberger's algorithm.

In \cref{subsection:groebner finite fields} we present a heuristic solution to this problem, namely computing the Gr\"obner bases over finite fields.
In \cref{subsection:verification}, we state a Borel polytope computation algorithm using the above heuristic.
As a consequence of errors introduced by this heuristic, we require ways to verify the correctness of computed moment polytopes.
To this end, we provide two verification algorithms in \cref{subsection:verification}.
Analogously to the probabilistic algorithm (\cref{theorem:probabilistic algorithm}) and the randomized algorithm (\cref{theorem:symbolic algorithm randomized}), one is correct with bounded probability while the second will be always correct when successful, but it may fail with some probability.
These algorithms require the use of tensor scaling algorithms \cite{burgisserTheoryNoncommutativeOptimization2019,burgisser2018tensorScaling} to verify whether some vertices are in fact elements of the moment polytope in question, which we will explain in \cref{subsection:tensor scaling}.
As for the randomized algorithm for computing moment polytopes, the randomized verification algorithm will require computation of Gr\"obner bases over function fields, which is slow.
In \cref{example:symbolic groebner derandomization} we showed a way to circumvent these issues.
In \cref{subsection:groebner reconstruction} we describe a more general procedure to achieve this goal.
We end the section in \cref{subsection:4x4x4} by applying our techniques to compute the moment polytopes of certain tensors in $\C^4\ot\C^4\ot\C^4$, most notably the $2\times2$ matrix multiplication tensor.

\subsection{Computing Gr\"obner bases over finite fields}
\label{subsection:groebner finite fields}

A possible solution to coefficient swell in Gr\"obner basis computations is to replace the field $\Q$ with the finite field $\F_q$ for some large prime $q$, and compute the Gr\"obner basis over $\F_q$ instead.
More precisely, fix some tensor $T$ and inequality $h$ and consider the polynomial system $F = \tensorpolysystem^T\!(h)$ from \cref{definition:tensor polysystem}. If $F \subseteq \Z[x_1,\ldots,x_m]$ (as is the case for our tensors of interest, also after randomization as in \cref{theorem:probabilistic algorithm}), we may interpret the coefficients as elements of $\F_q$ to obtain the polynomial system $F_{\F_q} \subseteq \F_q[x_1,\ldots,x_m]$.\footnote{Even if the entries of $T$ do not lie in $\Z$, we may pass to some suitable field extension of $\F_q$, but for simplicity we do not discuss this case.}
For computation of the reduced Gr\"obner basis $G_{\F_q}$ of $\<F_{\F_q}>_{\F_q}$ coefficients are of course elements of $\F_q$, and therefore coefficient swell is not a problem.

The primary question now is whether $q$ has the desired property that $\<F>_\C = \<1>_\C$ precisely when $\<F_{\F_q}>_{\F_q} = \<1>_{\F_q}$. %
Equivalently, by \cref{proposition:reduced groebner basis trivial ideal,corollary:groebner field extension}, we want to understand the logical relation between $G_\Q = \{1\}$ and $G_{\F_q} = \{1\}$, where $G_\Q$ denotes the Gr\"obner basis of $\<F>_\C$.
In general, these two statements are not equivalent.
However, it is known that this is only so for a finite number of primes, which illustrates the feasibility of this approach.
The first direction is straightforward.
Assuming $G_{\Q} = \{1\}$ implies that $1 \in \<F>_\Q$ which gives that $\sum_{f\in F} f g_f = 1$ for some rational polynomials $g_f$.
Let $D \in \N$ be the minimal scalar that satisfies $D \cdot g_f \in \Z[x_1,\ldots,x_m]$ for all $f \in F$.
Then $D\sum_{f\in F} f g_f = D$ is an equation in $\Z[x_1,\ldots,x_m]$, and we can pass to $\F_q$.
If $q$ does not divide $D$, then clearly $1 \in \<F_{\F_q}>_{\F_q}$.
This is false only for the finite number of divisors of $D$.
These divisors are of course upper bounded by $D$, which motivates us to pick $q$ to be large (even though $D$ can in principle be huge).

The other direction is more involved, and we refer to the literature on Gr\"obner bases \cite{arnoldModularAlgorithmsComputing2003}.
Primes for which passing to $\F_q$ does not ``lose much algebraic information'' are typically called \emph{lucky primes} for $F$.
We refer to \cite[Definition~5.1]{arnoldModularAlgorithmsComputing2003} for a precise definition.
For all lucky primes $q$ as defined in \cite{arnoldModularAlgorithmsComputing2003}, we have $G_\Q = \{1\}$ if and only if $G_{\F_q} = \{1\}$.
Then \cite[Theorem~5.13]{arnoldModularAlgorithmsComputing2003} states $q$ is lucky whenever $q$ does not divide the leading coefficients of a certain set of polynomials in $\Z[x_1,\ldots,x_m]$.
This set of polynomials is what is called a \emph{strong Gr\"obner basis} of $F \subseteq \Z[x_1,\ldots,x_m]$ computed over the ring $\Z$, which we do not define here.
The take-away message is, $q$ is not lucky only for the finite set of divisors. In all other cases $G_\Q = \{1\}$ if and only if $G_{\F_q} = \{1\}$.

Assuming the generalized Riemann hypothesis, there exists bounds on the number of primes such that
$G_{\F_q} \neq \{1\}$ (e.g.\ a solution exists in the algebraic closure of $\F_q$) whenever $G_\Q = \{1\}$ (e.g.\ no solution exists over $\C$) \cite[Theorem~5, see also Section~7]{koiran1996hilbertNullstellensatzPolynomialHierachy}.
Then there are $d^{\OO(m^2)}$ such primes, where $d$ is a the largest degree among the polynomials of $F$ and $m$ is the number of variables. In our case, $d$ equals the number of tensor factors $d = 3$, and the number of variables equals $m = a(a-1)/2 + b(b-1)/2 + c(c-1)/2$.
In the case $G_\Q \neq \{1\}$ the situation is not so clear, and we are not aware of explicit bounds on the number of primes for which $G_{\F_q} \neq \{1\}$.

We obtain the following alternative algorithm to \Cref{algorithm:tensor checking-inequalities}.

\begin{algorithmbreak}{Determine attainability for a tensor $T$ using the finite field heuristic.}
\label{algorithm:tensor checking-inequalities Fq}%
\textbf{Input:} A tensor $T \in \C^a\ot\C^b\ot\C^c$ and an inequality $h \in \R^{a+b+c}$. \\
\textbf{Output:} \texttt{True} or \texttt{False} \\
\textbf{Algorithm:}
\begin{algorithmic}[1]
    \State Generate a large random prime $q$, say $q \geq 2^{30}$.
    \State Create symbolic lower triangular matrices $A,B,C$ with ones on the diagonal.
    \State Setup the polynomial system $\tensorpolysystem^T_{\F_q}\!(h) \coloneqq \big\{ \big((A\ot B\ot C)T\big){}_{i,j,k} \mid \langle (\,e_i \sep e_j \sep e_k\,),h \rangle < 0 \big\}$.
    \State Compute the reduced Gr\"obner basis $G_{\F_q}$ of $\smash{\tensorpolysystem^T_{\F_q}\!(h)}$ over $\F_q$.
    \State \Return $G_{\F_q} \neq \{1\}$
\end{algorithmic}
\end{algorithmbreak}

\subsection{Inclusion verification via tensor scaling}
\label{subsection:tensor scaling}

Recall the randomized algorithm from \cref{section:algorithms for computing tensor moment polytopes}, \cref{algorithm:tensor algorithm randomized}.
We used there that for any matrices $A,B$ and $C$, the candidate polytope $\DeltaB((A \ot B \ot C)T)$ is always contained in $\Delta(T)$.
Hence after computing $\DeltaB((A \ot B \ot C)T)$ for random $A,B$ and $C$, it only remains to prove that $\Delta(T) \subseteq \DeltaB((A \ot B \ot C)T)$, which we could do by computing Gr\"obner bases over an appropriate function field.
When moving to finite fields, we have no guarantees at all whether attainability will be correctly determined. As a result the candidate polytope $P$ we obtain need not be contained in $\Delta(T)$.
Indeed, even if we prove that the inequalities of $P$ hold for $\Delta(T)$ (using Gr\"obner bases over the function field as before), we obtain only that $\Delta(T) \subseteq P$.
Thus, an extra ingredient is required for verification.
This ingredient is provided by tensor scaling algorithms \cite{burgisser2018tensorScaling,burgisserTheoryNoncommutativeOptimization2019}, which can be used to verify whether a point $p$ is an element of $\Delta(T)$.
If we let $p$ range over the vertices of $P$ and use these algorithms to prove $p \in \Delta(T)$ for all $p$, we obtain $P \subseteq \Delta(T)$ by convexity, as desired.

More precisely, tensor scaling is a procedure that attempts the following: given a tensor $T \in \C^{a}\ot\C^b\ot\C^c \setminus \{0\}$ and a point~$p \in \weylchamber$ assumed to be in $\Delta(T)$, find an element $T' \in \overline{\G \cdot T}$ such that $\mu(T') = \diag(p)$ (recall that $\mu$ is the moment map defined in \cref{equation:tensor moment map}).
Methods for this problem, sometimes referred to as the \emph{moment polytope membership} problem, were developed in a long series of works \cite{burgisser2018alternatingMinimization,allenzhu2018operatorScaling,cole2018operatorScaling,garg2019operatorScaling,burgisser2018tensorScaling,burgisserTheoryNoncommutativeOptimization2019}.

A triple of matrices $(A,B,C) \in \G$ such that $T' \approx (A\ot B \ot C) T$ can be viewed as a certificate for $p \in \Delta(T)$, provided the approximation error is small enough, as we explain below.
Tensor scaling is an iterative and numerical algorithm, and it will find $T' = (A \ot B \ot C)T$ such that $\mu(T')$ and $\diag(p)$ are close, say $\norm{\mu(T') - \diag(p)} \leq \eps$, where we can set $\eps > 0$ to a value of our choice.

Recall that $\weylchamber_\Q \subseteq \weylchamber$ are the set of triples of non-decreasing vectors entries from $\Q$.
Let $\weylchamber_\Z \subseteq \weylchamber$ be the subset consisting of vectors with entries from $\Z$.

\begin{theorem}[Tensor scaling, {\cite[Theorem~1.13]{burgisser2018tensorScaling}}]
    \label{thm:scaling via gradient descent}
    There exists an algorithm with input $T \in \C^a\ot\C^b\ot\C^c$, $p \in \weylchamber_\Q$ and $\eps > 0$ that either outputs correctly that $p \notin \Delta(T)$ with probability at least 1/2,
    or it outputs $(A,B,C) \in \G$ such that
    \begin{equation}
        \label{eq:tensor scaling}
        \big\| \mu\bigl( (A \ot B \ot C) T\bigr) - \diag(p) \big\| \leq \eps.
    \end{equation}
    This algorithm runs in time polynomial in $1/\eps$ and the bitsize of the inputs $T$ and $p$.
\end{theorem}

Crucially, there is a choice of $\eps$ such that inclusion of a point $p$ is guaranteed if $\norm{ \mu\big( (A\ot B \ot C) T \big) - p} \leq \eps$ for some $(A,B,C) \in \G$. This has been observed in
\cite[Lemma~7]{burgisser2015membershipCoNP} (for the Kronecker polytope), \cite[Theorem~1.12]{burgisser2018tensorScaling}, \cite[Corollary~3.35]{burgisserTheoryNoncommutativeOptimization2019}.
We repeat the proof here for convenience of the reader.
We denote by $\norm{h}_\infty \coloneqq \max_{i=1}^{a+b+c} |h_i|$ the infinity norm of $h \in \R^{a+b+c} \cong \R^a \times \R^b \times \R^c$.

\begin{theorem}[{\cite{burgisser2015membershipCoNP}}]
\label{prop:scaling necessary epsilon}
Let $(\lambda,\mu,\nu) \in \weylchamber_\Z$ with $\lambda,\mu,\nu$ each summing to $\ell \in \N$.
Let $\ineqs$ be as in \cref{lemma:weight enumeration algorithm}.
Set $C \coloneqq \max\{\norm{h}_\infty \mid h \in \ineqs\}$.
Write $n \coloneqq a+b+c$.
If $0 \leq \eps < (\sqrt{n}C\ell)^{-1}$ and $T' \in \G \cdot T$, then
\begin{align}
\label{equation:polytope inclusion bound}
    \big\| \mu\bigl( T' \bigr) - \diag(p) \big\| \leq \eps
\end{align}
implies $p \in \Delta(T)$.
Moreover, $C \leq n^{n-1}$.
\end{theorem}
\begin{proof}
Suppose the assumptions of the lemma hold but $p \notin \Delta(T)$.
Then there exists an inequality $h$ that is valid for $\Delta(T)$ while $\<p,h> < 0$.
Again let $\ineqs$ be as in \cref{lemma:weight enumeration algorithm}.
By \cref{lemma:tensor attainability} and because $p \in \weylchamber$, we may without loss of generality pick $h \in \ineqs_S \subseteq \ineqs$ for some $S \subseteq \Omega$.
We observe that $h$ and $(\lambda,\mu,\nu)$ have integer entries, so we can bound the distance from $p$ to $\Delta(T)$ as
\begin{align*}
\frac{\left|\<p, h>\right|}{\norm{h}} = \frac{1}{\ell} \frac{\left|\<(\lambda,\mu,\nu),h>\right|}{\norm{h}} \geq \frac{1}{\ell \norm{h}}.
\end{align*}
Now we use that the Euclidean norm $\norm{\cdot}$ on $\R^{n}$ is upper bounded by $\sqrt{n}\norm{\cdot}_\infty$.
By definition of $C$ we know $\norm{h}_\infty \leq C$.
Thus the distance from $p$ to $\Delta(T)$ is lower bounded by $(\sqrt{n}C\ell)^{-1}$.
Finally, we use the Wielandt-Hoffman theorem (see e.g.\ \cite[III.5.15]{bhatia1997matrixAnalysis}) to bound
\begin{align*}
    \norm{\mu(T') - \diag(p)}
    \geq \bigl\| \spec\bigl( \mu(T') \bigr) - p \bigr\|
    \geq \frac{1}{\sqrt{n}C\ell}
\end{align*}
and obtain a contradiction.

It remains to upper bound $C$.
Recall that any $h \in \ineqs$ is defined as the solution to $M_\textnormal{e} = 0$, where $M$ is some weight matrix as computed in \cref{algorithm:tensor weight-matrix-enumeration}.
Note that $M_\textnormal{e}$ has only integer entries equal to $1$, $-1$ and $0$. Then Siegel's lemma \cite[Lemma~D.4.1.]{hindry2000diophantineGeometry} tells us that
$\norm{h}_\infty \leq n^{n-1}$.
\end{proof}

For every instance of tensor scaling, to some vertex $p$, we set $\eps$ such that the above property holds.
In our applications, $n$, $\ell$ and $C$ are all small, and the procedure is practically feasible and fast. For example, for $\C^3\ot\C^3\ot\C^3$ the value of $\ell$ for any vertex is bounded by 15 (see \cref{table:333 vertex data}) and
$C$ is equal to 16. So we need only take $\varepsilon < 1/720$.
For $\C^4\ot\C^4\ot\C^4$ we have $C = 43$. Our candidate for the $2\times2$ matrix multiplication polytope satisfies $\ell \leq 12$ for all its vertices, so
we can take $\varepsilon < 1/(516\sqrt{n}) < 1/2064$.
For the Kronecker polytope, $\ell \leq 24$, and we can take $\varepsilon < 1/4128$.

Tensor scaling algorithms produce numerical results, so some care should be taken to ensure no numerical issues arise when checking \cref{equation:polytope inclusion bound}.
Standard interval arithmetic packages can do the job.
Alternatively, the produced matrices $A,B,C$ can potentially be rounded to an algebraic representation which satisfies \cref{equation:polytope inclusion bound} exactly.

\begin{remark}
In \cite{burgisserTheoryNoncommutativeOptimization2019} the above results are stated for arbitrary reductive groups and their representations.
Hence the procedure can also be used in the general setting.
\end{remark}

\subsection{Verification algorithms}
\label{subsection:verification}

From \cref{algorithm:tensor checking-inequalities Fq} we obtain a new algorithm for computing Borel polytopes.%

\begin{algorithmbreak}{Borel polytope computation using the finite field heuristic.}
\label{algorithm:tensor algorithm borel Fq}%
 \textbf{Input:} A tensor $T \in \Z^a \ot \Z^b \ot \Z^c$. \\
 \textbf{Output:} A set of inequalities defining a polytope $P$, potentially equal to the Borel polytope $\DeltaB(T)$ of $T$. \\
 \textbf{Algorithm:}
\begin{algorithmic}[1]
    \State Let $\ineqs$ be a set of inequalities as in \cref{lemma:tensor attainability}.
    \Statex \vspace{-0.12em}
    \State Initialize $\widehat{\ineqs}_{T} \gets \varnothing$.
    \For{$h \in \ineqs$}
        \State Determine attainability of $\Omega_h$ for $(A \ot B \ot C)T$ using the finite field heuristic (\cref{algorithm:tensor checking-inequalities Fq}).
        \If{attainable}
            \State Add $h$ to $\widehat{\ineqs}_{T}$.
        \EndIf
    \EndFor
    \Statex  \vspace{-0.12em}
    \State \Return The set $\widehat{\ineqs}_{T}$ along with a set of inequalities defining $\weylchamber$.
\end{algorithmic}
\end{algorithmbreak}

As mentioned before, the above algorithm may be extended to other $T \in \C^a\ot\C^c\ot\C^c$ by taking suitable field extensions.

Typically we run \cref{algorithm:tensor algorithm borel Fq} on input $(A\ot B\ot C)T$ for randomly generated $(A,B,C) \in \G_{\uppertriangular}$ with integer entries.
As we have no guarantee that the polytope $P$ defined by the output is correct,
we need to verify $\Delta(T) \subseteq P$ and $P \subseteq \Delta(T)$.
This we do in the following algorithm.

\begin{algorithmbreak}{Moment polytope verification.}
\label{algorithm:tensor algorithm verification}%
 \textbf{Input:} A tensor $T \in \C^a \ot \C^b \ot \C^c$ and a candidate set of rational inequalities $\widehat{\ineqs}_{\Tgeneric}$ for $\Delta(T)$. \\
 \textbf{Output:} \textnormal{\texttt{Correct}}, \textnormal{\texttt{Incorrect}} or \textnormal{\texttt{Failure}}. \\
 \textbf{Algorithm:}
\begin{algorithmic}[1]
    \State Replace $\widehat{\ineqs}_{\Tgeneric}$ by an irredundant subset.
    \State Enumerate the vertices of the polytope $P \subseteq \weylchamber$ defined by $\widehat{\ineqs}_{\Tgeneric}$.
    \State Compute $C \coloneqq \max \{ \norm{h}_\infty \mid h \in \ineqs\}$.
    \Statex \vspace{-0.12em}
        \For{all vertices $p$ of $P$}
        \State Determine $\ell$ such that $p\ell \in \N^{a+b+c}$ and compute $\eps = (\sqrt{a+b+c}\,\ell\, C + 1)^{-1}$.
        \If{tensor scaling (\cref{thm:scaling via gradient descent}) with inputs $T, p$ and precision $\eps$ claims $p \notin \Delta(T)$}
            \State \Return \texttt{Incorrect}
        \EndIf
    \EndFor
    \Statex \vspace{-0.12em}
    \For{$h \in \widehat{\ineqs}_{\Tgeneric}$}
        \State Determine attainability of $\Omega_h$ for all $T_h \in U_h \cdot T$, with $U_h \subseteq \G_{\uppertriangular}$ non-empty Zariski-open (\cref{algorithm:tensor checking-inequalities symbolic}).
        \If{not attainable}
            \State \Return \textnormal{\texttt{Failure}}.
        \EndIf
    \EndFor
    \Statex \vspace{-0.12em}
    \State \Return \texttt{Correct}.
\end{algorithmic}
\end{algorithmbreak}

The algorithm may output \texttt{Failure}, which indicates that the inequalities may be correct, but the verification algorithm cannot prove this. This is for the same reason as discussed below \cref{theorem:symbolic algorithm randomized}: an inequality may be valid while $\Omega_h$ is not (generically) attainable.

\begin{theorem}
    \label{theorem:verification}
    Let $T \in \C^a \ot \C^b \ot \C^c$ be a tensor and $\widehat{\ineqs}_{\Tgeneric}$  a set of rational inequalities.
    If \cref{algorithm:tensor algorithm verification} with input $T$ and $\widehat{\ineqs}_{\Tgeneric}$ outputs \textnormal{\texttt{Correct}}, then
    \begin{align}
    \label{equation:verification output}
        \Delta(T) = \bigcap_{\widehat{\ineqs}_{\Tgeneric}} H_h \cap \weylchamber.
    \end{align}
    If the output is instead \textnormal{\texttt{Incorrect}}, \cref{equation:verification output} is false with probability at least 1/2.
\end{theorem}
\begin{proof}
    Let $P = \bigcap{}_{\widehat{\ineqs}_{\Tgeneric}} H_h \cap \weylchamber$.
    If the output is \texttt{Correct}, we find that for generic $\Tgeneric \in \G_{\uppertriangular}$ we have that $\Delta(T) = \Delta(\Tgeneric)$ and for each
    $h \in \smash{\widehat{\ineqs}_{\Tgeneric}}$, $\Omega_h$ is attainable for $\Tgeneric$ (this is the same argument as in the proof of \cref{theorem:symbolic algorithm randomized}).
    This proves $\Delta(T) = \DeltaB(\Tgeneric) \subseteq P$.

    If tensor scaling succeeds, by \cref{thm:scaling via gradient descent,prop:scaling necessary epsilon} we obtain that all vertices of $P$ lie in $\Delta(T)$. We conclude that $\Delta(T) = P$ by convexity, as desired.

    If the output is \texttt{Incorrect}, tensor scaling failed for a vertex of $P$, and \cref{thm:scaling via gradient descent} tells us that $p \notin \Delta(T)$ with probability at least $1/2$. This implies $P \neq \Delta(T)$.
\end{proof}

\begin{remark}
\label{remark:probabilistic outer verification}
Alternatively to using \cref{algorithm:tensor checking-inequalities symbolic} in \cref{algorithm:tensor algorithm verification}, attainability of $h$ for $T' = (A \ot B \ot C)T$ can be verified for each inequality $h \in \widehat{\ineqs}_{\Tgeneric}$ using \cref{algorithm:tensor checking-inequalities}.
If this succeeds for all $h \in \widehat{\ineqs}_{\Tgeneric}$ we obtain $\DeltaB(T') \subseteq P$, where $P$ is the polytope in $\weylchamber$ defined by $\widehat{\ineqs}_{\Tgeneric}$.
If tensor scaling succeeds for all vertices of $P$, we obtain $P \subseteq \Delta(T)$.
If $(A,B,C)$ is chosen as in \cref{theorem:probabilistic algorithm}, with probability at least $1/2$, we find that $\Delta(T) = \DeltaB(T')$ and hence $\Delta(T) = P$.
If the procedure does not fail, this allows us to conclude $P$ is equal to the moment polytope of $T$ with arbitrary probability less than 1.
\end{remark}

\subsection{Gr\"obner basis reconstruction}
\label{subsection:groebner reconstruction}

In practice, for many inequalities computing the Gr\"obner basis over the function field in \Cref{algorithm:tensor checking-inequalities symbolic} takes too long to be executable directly.
In \cref{example:symbolic groebner derandomization} we saw a way we could still run perform this algorithm, which was sufficient for all $3\times3\times3$ tensors. However, this trick is not applicable in general.
In this section we present a more systematic way to use Gr\"obner basis over $\Q$ (obtained in \cref{algorithm:tensor checking-inequalities}) to construct Gr\"obner bases over the function field.
We apply our techniques to prove an inequality of the $2\times2$ matrix multiplication tensor.

This construction proceeds via (symbolic) linear system solving.
Let $e \in G_\Q(T') \subseteq \Q[x_1,\ldots,x_m]$.
For generic randomization $(A,B,C) \in \G_{\uppertriangular}$, we saw before that the coefficients of $e$ can be deduced from one of the elements of the symbolic Gr\"obner basis by filling in $(A,B,C)$ into the coefficients (\cref{proposition:groebner function field}).
Hence, if we know $e$ it is likely the symbolic Gr\"obner basis contains an element $e_{\textnormal{s}}$ which equals $e$ except with all coefficients in $\Q$ replaced with some rational functions in $\Q(\G_{\uppertriangular})$.

We can write $e = \sum_{i} f_i g_i$ where $\{f_1,\ldots,f_q\} = \tensorpolysystem^{T'}\!(h) \subseteq \Q[x_1,\ldots,x_m]$ and $g_i \in \Q[x_1,\ldots,x_m]$.
The coefficients of $e$ are then specified by a linear system (see also \cite{deLoeraComputingInfeasabilityCertificates2011}).
More precisely, let $d$ be the maximum degree among $\{g_i\}$ and $D$ the maximum degree among $\{f_ig_i\}$.
Then order all monomials in $\Q[x_1,\ldots,x_m]$ up to degree $d$ by $h_1,\ldots,h_k$ and up to degree $D$ by $h_1,\ldots,h_K$. One may take only the monomials that contain variables used in $f_i$ and $g_i$.
Write $v_{f} \in \C^K$ for the column vector of coefficients of a polynomial $f$ of degree less than $D$ with respect to this monomial ordering.
Then this linear system is given by
\begin{align}
    \label{eq:rational to symbolic system}
    \underbrace{\bigg[\begin{array}{c|c|c}
        F_1 & \cdots & F_q
    \end{array}\bigg]}_{\eqqcolon\, F}
    {
    \renewcommand\arraystretch{0.6}
    \setlength\extrarowheight{3pt}
    \left[\begin{array}{c}
        w_1\\\text{---}\\
        \raisebox{0.1em}{$\vdots$}\\\text{---}\\
        w_q
    \end{array}\right]}
    =
    v_e
    \qquad
    \text{where}\quad
    w_i \in \C^k \text{ and }
    F_i \coloneqq
    \Big[\begin{array}{c|c|c|c}
        v_{f_ih_1} & v_{f_ih_2} & \cdots & v_{f_ih_k}
    \end{array}\Big].
\end{align}
Observe that when taking $w_i$ to equal the vector of coefficients defining $g_i$ as a linear combination of $h_1,\ldots,h_k$, the above linear system is indeed satisfied.

Now we replace $\{f_1,\ldots,f_q\}$ by $\{f_{\textnormal{s},1},\ldots,f_{\textnormal{s},q}\} = \tensorpolysystem^{T_{\textnormal{s}}}(h) \subseteq \Q(\G_{\uppertriangular})[x_1,\ldots,x_m]$, and set up the symbolic matrix $F_{\textnormal{s}}$ defined analogously to $F$ above.
Remove the rows in the system belonging to non-zero rows of $v_e$.
Then we can compute the kernel $w_{\textnormal{s}}$ of the resulting matrix $F_{\textnormal{s}}'$.
If all goes well, $F_{\textnormal{s}} w_{\textnormal{s}}$ will equal $v_{e_{\textnormal{s}}}$.
If all goes well for all elements, the resulting Gr\"obner basis $\hat G_{\Q(\G_{\uppertriangular})}$ is the symbolic Gr\"obner basis, which one may in practice check quickly by computing the Gr\"obner basis of $\tensorpolysystem^{T_{\textnormal{s}}}(h) \cup \hat G_{\Q(\G_{\uppertriangular})}$.

We may further optimize this strategy.
We can assume the symbolic solution $w_{\textnormal{s}}$ will have the same entries non-zero as the rational solution, and by this assumption we can remove all columns in $F_{\textnormal{s}}'$ corresponding to the entries where $w_i$ is zero.
Moreover, we can assume the row pivots (i.e.\ the rows which cause a increase in rank when going from top to bottom) of $F_{\textnormal{s}}$ are equal to those of $F$.
We can further restrict $F_{\textnormal{s}}$ to these rows.

\begin{example}
\label{example:mm2 derandomization}
We put this idea into practice for an inequality which holds for the moment polytope of the $2 \times 2$ matrix multiplication tensor $\MM_2$, but not for~$\kronpol{4}{4}{4}$.%
\footnote{One vertex of~$\kronpol{4}{4}{4}$ that is excluded by~$h$ is~$p = (\frac12, \frac16, \frac16, \frac16 \sep \frac12, \frac12, 0, 0 \sep \frac14, \frac14, \frac14, \frac14)$. To see that~$p \in \kronpol{4}{4}{4}$, observe that~$S = \sqrt{\sfrac{1}{4}} e_{1,1,4} + \sqrt{\sfrac{1}{4}} e_{1,2,3} + \sqrt{\sfrac{1}{6}} e_{2,2,2} + \sqrt{\sfrac{1}{12}} e_{3,1,2} + \sqrt{\sfrac{1}{12}} e_{3,2,1} + \sqrt{\sfrac{1}{6}} e_{4,1,1}$ satisfies~$\mu(S) = \diag(p)$.}
The tensor is specified by the slices
\begin{align}
    \label{equation:MM2}
    \MM_2 \sim T =
    \left[
    \begin{array}{cccc|cccc|cccc|cccc}
    \grayzero & \grayzero & \grayzero & \grayzero & \grayzero & \grayzero & \grayzero & \grayzero & \grayzero & \grayzero & 1 & \grayzero & 1 & \grayzero & \grayzero & \grayzero \\
    \grayzero & \grayzero & \grayzero & \grayzero & \grayzero & \grayzero & \grayzero & \grayzero & \grayzero & \grayzero & \grayzero & 1 & \grayzero & 1 & \grayzero & \grayzero \\
    \grayzero & \grayzero & 1 & \grayzero & 1 & \grayzero & \grayzero & \grayzero & \grayzero & \grayzero & \grayzero & \grayzero & \grayzero & \grayzero & \grayzero & \grayzero \\
    \grayzero & \grayzero & \grayzero & 1 & \grayzero & 1 & \grayzero & \grayzero & \grayzero & \grayzero & \grayzero & \grayzero & \grayzero & \grayzero & \grayzero & \grayzero
    \end{array}
    \right].
\end{align}
We prove the inequality $h = (\,-1, 0, 0, -1 \sep 0, 0, 1, 1 \sep 1, 1, 0, 0\,)$ holds for $\Delta(\MM_2)$.
Let $T'$ equal $\MM_2$ with some randomization $(A,B,C) \in \G_{\uppertriangular}$ applied.
The polynomial system $\tensorpolysystem^{T'}(h)$ contains 8 polynomials of degree at most 2 using the 7 variables $\{x_{18}, x_{17}, x_{16}, x_{15}, x_{14}, x_6, x_7\}$.
For some randomization, the Gr\"obner basis over $\Q$ equals
\begin{align*}
    \Big\{
    &x_5 + \tfrac{268593690625}{3339944740372}, x_6 + \tfrac{742225980905}{834986185093}, x_{14} - \tfrac{17694295006140}{9661977296489},
    x_{15} + \tfrac{102169377609156}{9661977296489},
    \\&x_{16} - \tfrac{17694295006140}{9661977296489} x_{18} - \tfrac{519372632850}{9661977296489},
    x_{17} + \tfrac{102169377609156}{9661977296489}x_{18} + \tfrac{3949646632800}{9661977296489}
     \Big\}.
\end{align*}
Notice that we cannot perform the trick from \cref{example:symbolic groebner derandomization}.
Let $e = x_5 + {\scriptstyle 268593690625/3339944740372}$.
We can solve the rational system \cref{eq:rational to symbolic system} for $d = 2$.
There are $k = 29$ monomials of degree $2$ or lower, and $K = 323$ monomials of degree $D = 2 \cdot 2 = 4$ or lower.
Since $8 \cdot 29 = 232$, $F_{\textnormal{s}}$ is a symbolic matrix of size $323 \times 232$.
Restricting to rows where $v_e$ is non-zero and columns where the solution $w$ of the rational system (containing the stacked $v_{g_i}$) is non-zero, the resulting matrix has size $321 \times 21$.
Restricting to the row pivots of $F$, it has size $20 \times 21$.
We find a kernel element and construct
\begin{align*}
    e_{\textnormal{s}}
    =
    &\ x_5 +
    \scalebox{0.75}{$
    \big(-z_{1}z_{4}z_{8}z_{10}z_{11}z_{15} + z_{2}z_{3}z_{8}z_{10}z_{11}z_{15} \big)/\big(z_{1}^2z_{5}z_{9}z_{13}z_{17} - z_{1}^2z_{5}z_{9}z_{14}z_{16} - z_{1}z_{2}z_{5}z_{8}z_{13}z_{17} + z_{1}z_{2}z_{5}z_{8}z_{14}z_{16}
    $}\\&\scalebox{0.85}{$
    +\ z_{1}z_{3}z_{5}z_{9}z_{11}z_{17} - z_{1}z_{3}z_{5}z_{9}z_{12}z_{16} + z_{1}z_{3}z_{5}z_{9}z_{13}z_{15} - z_{1}z_{4}z_{5}z_{8}z_{11}z_{17} + z_{1}z_{4}z_{5}z_{8}z_{12}z_{16} - z_{1}z_{4}z_{5}z_{8}z_{13}z_{15}
    $}\\&\scalebox{0.85}{$
    +\ z_{1}z_{4}z_{6}z_{9}z_{11}z_{15} - z_{1}z_{4}z_{7}z_{8}z_{11}z_{15} - z_{2}z_{3}z_{6}z_{9}z_{11}z_{15} + z_{2}z_{3}z_{7}z_{8}z_{11}z_{15} + z_{3}^2z_{5}z_{9}z_{11}z_{15} - z_{3}z_{4}z_{5}z_{8}z_{11}z_{15} \big)$}
\end{align*}
This approach works for every element in the Gr\"obner basis, and we verify the resulting set equals the symbolic Gr\"obner basis.
Hence we prove the inequality $h$ holds for $\Delta(\MM_2)$.
\end{example}

\subsection{Moment polytopes of selected \texorpdfstring{$4\times4\times4$}{4 x 4 x 4}  tensors.}
\label{subsection:4x4x4}

We end this section with computational results.
We computed the moment polytopes of several tensors $T \in \C^4\ot\C^4\ot\C^4$ by applying \cref{algorithm:tensor algorithm borel} to $(A \ot B \ot C)T$ for integer randomization of $A,B,C$ (we provide more implementation details in \cref{subsection:tensor implementation}).
The results are listed in \cref{table:444 vertex data}.
All results were verified following \cref{remark:probabilistic outer verification}.
That is, inclusion of all points is verified with certainty, but there is a non-zero probability too little points have been listed.
For each vertex, exact algebraic rational expressions for $(A,B,C)$ satisfying property \cref{equation:polytope inclusion bound} were reconstructed from the numeric matrices obtained from running the tensor scaling algorithm (\cref{thm:scaling via gradient descent}), via rounding.

In the table, $\unit{k}$ denotes the unit tensor, $\MM_2$ denotes the $2 \times 2$ matrix multiplication tensor (\cref{equation:MM2}), $v_i$ are defined as in \cref{equation:333 vs},
$W \coloneqq e_{112}+e_{121}+e_{211}$
and $\Tdet \coloneqq e_1 \wedge e_2 \wedge e_3$.
The symbols $\ot$ and $\oplus$ denote the Kronecker product and the direct sum respectively.

For the $2\times2$ matrix multiplication tensor, exclusion of the point $(\sfrac14,\sfrac14,\sfrac14,\sfrac14 \sep \sfrac13,\sfrac13,\sfrac13,0 \sep \sfrac12, \sfrac12,0,0)$ is proved in \cite{vandenBerg2025mmPolytope}, where a generalization to $n \times n$ matrix multiplication tensors is presented as well.
These results where inspired by the computational results obtained here.

\section*{Acknowledgements}
MvdB, VL, and MW acknowledge support by the European Research Council (ERC Grant Agreement No.~101040907).
MvdB also acknowledges financial support by the Dutch National Growth Fund (NGF), as part of the Quantum Delta NL visitor programme.
MC and HN acknowledge financial support from the European Research Council (ERC Grant Agreement No.~818761), VILLUM FONDEN via the QMATH Centre of Excellence (Grant No.~10059) and the Novo Nordisk Foundation (grant NNF20OC0059939 `Quantum for Life'). MC also thanks the National Center for Competence in Research SwissMAP of the Swiss National Science Foundation and the Section of Mathematics at the University of Geneva for their hospitality. Part of this work was completed while MC was Turing Chair for Quantum Software, associated to the QuSoft research center in Amsterdam, acknowledging financial support by the Dutch National Growth Fund (NGF), as part of the Quantum Delta NL visitor programme.
HN also acknowledges support by the European Union via a ERC grant (QInteract, Grant No.~101078107).
MW also acknowledges the Deutsche Forschungsgemeinschaft (DFG, German Research Foundation) under Germany's Excellence Strategy - EXC\ 2092\ CASA - 390781972, the BMBF (QuBRA, 13N16135; QuSol, 13N17173) and the Dutch Research Council (NWO grant OCENW.KLEIN.267).
JZ was supported by NWO Veni grant VI.Veni.212.284. JZ and MvdB are grateful for support by Qusoft and CWI.
Views and opinions expressed are those of the author(s) only and do not necessarily reflect those of the European Union or the European Research Council Executive Agency. Neither the European Union nor the granting authority can be held responsible for them.

\newpage

\appendix

\vspace*{-5em}
\section{Vertices of moment polytopes of all \texorpdfstring{$3\times3\times3$}{3 x 3 x 3} tensors}

\makebox[\textwidth][c]{\includegraphics{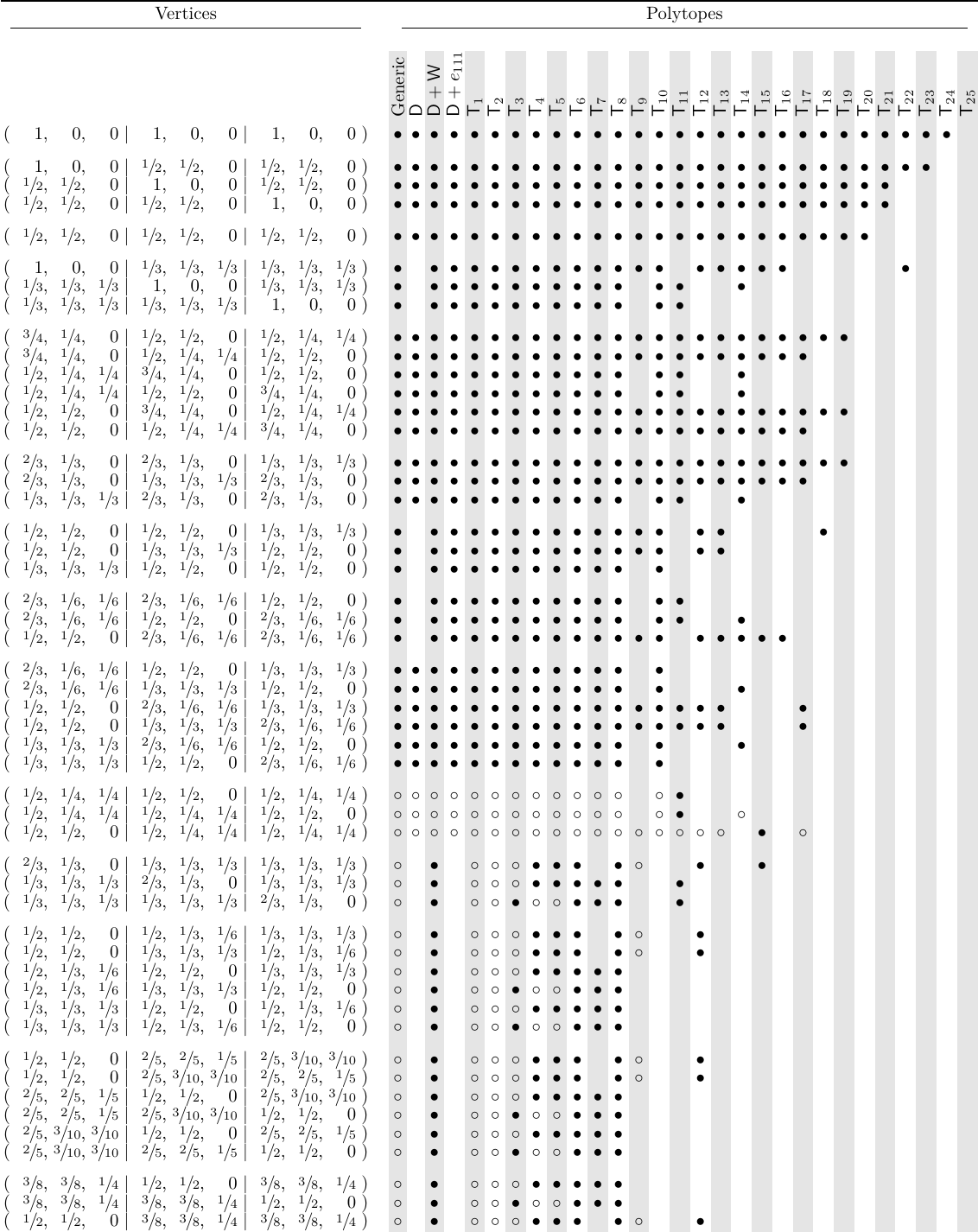}}
\begin{table}[H]
\makebox[\textwidth][c]{\includegraphics{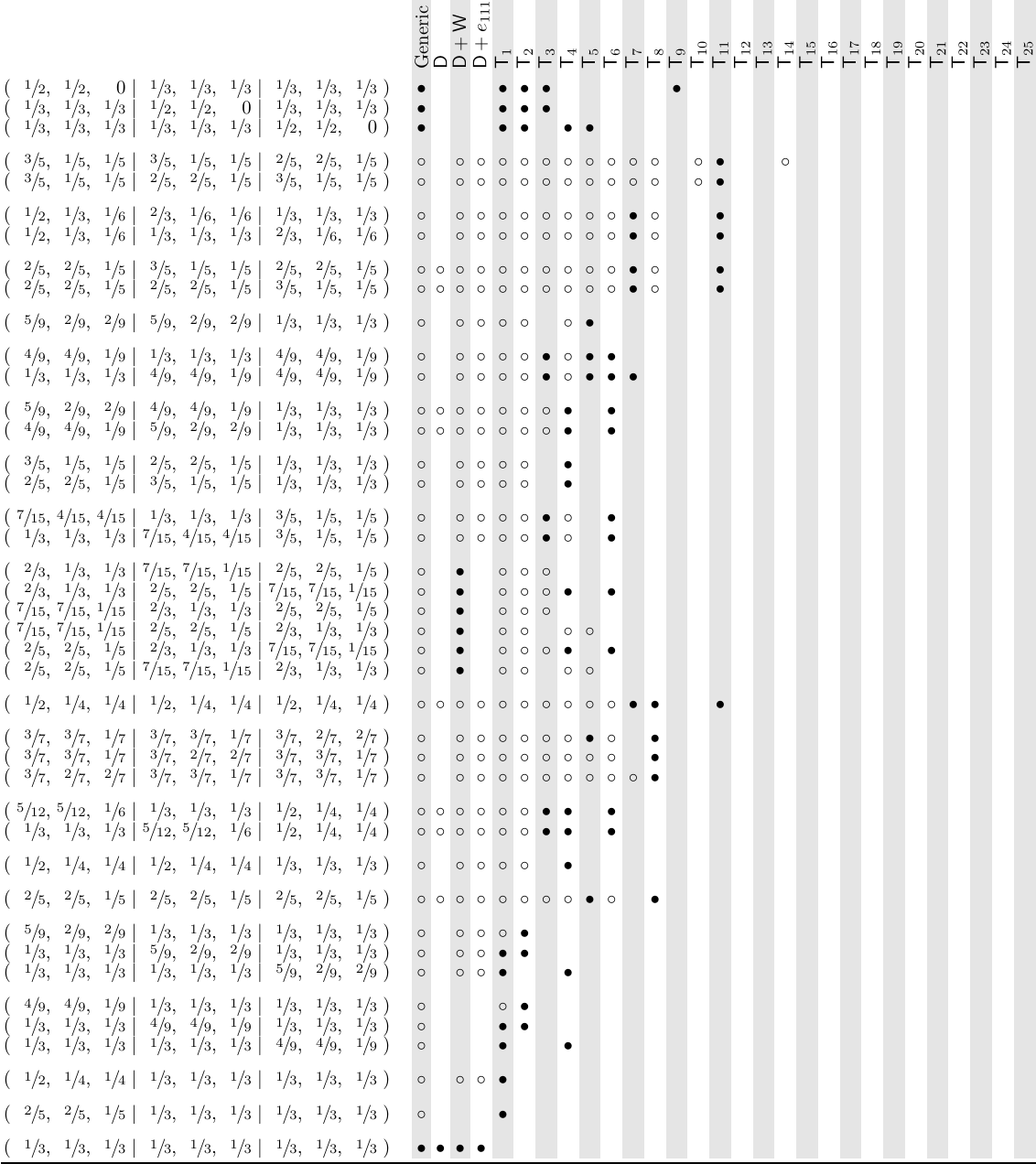}}
\caption{Complete vertex data of the moment polytopes in $\C^3 \ot \C^3 \ot \C^3$, see \cref{section:C333}. A $\cm$ indicates the point is a vertex of the moment polytope of the indicated tensor. A $\cl$ indicates it is only an element.
The column labeled \emph{Generic} concerns the Kronecker polytope.
This is the moment polytope of all tensors in family 1, 2 and 3, and of some tensors in family 4.
The tensor $\Tnurmiev_i$ refers to the $i$-th tensor in \cref{table:c333 unstable tensors info}.
Up to permutations of the subsystems, all moment polytopes are included.
Vertices are grouped by their orbits under system permutations.
This data is also available online in \cite{vandenBerg2025momentPolytopesGithub}.}
\label{table:333 vertex data}
\end{table}

\newpage

\vspace*{-5em}
\section{Vertices of moment polytopes of selected \texorpdfstring{$4\times4\times4$}{4 x 4 x 4} tensors}

\raisebox{0pt}[63em]{\makebox[\textwidth][c]{\includegraphics[scale=1.00]{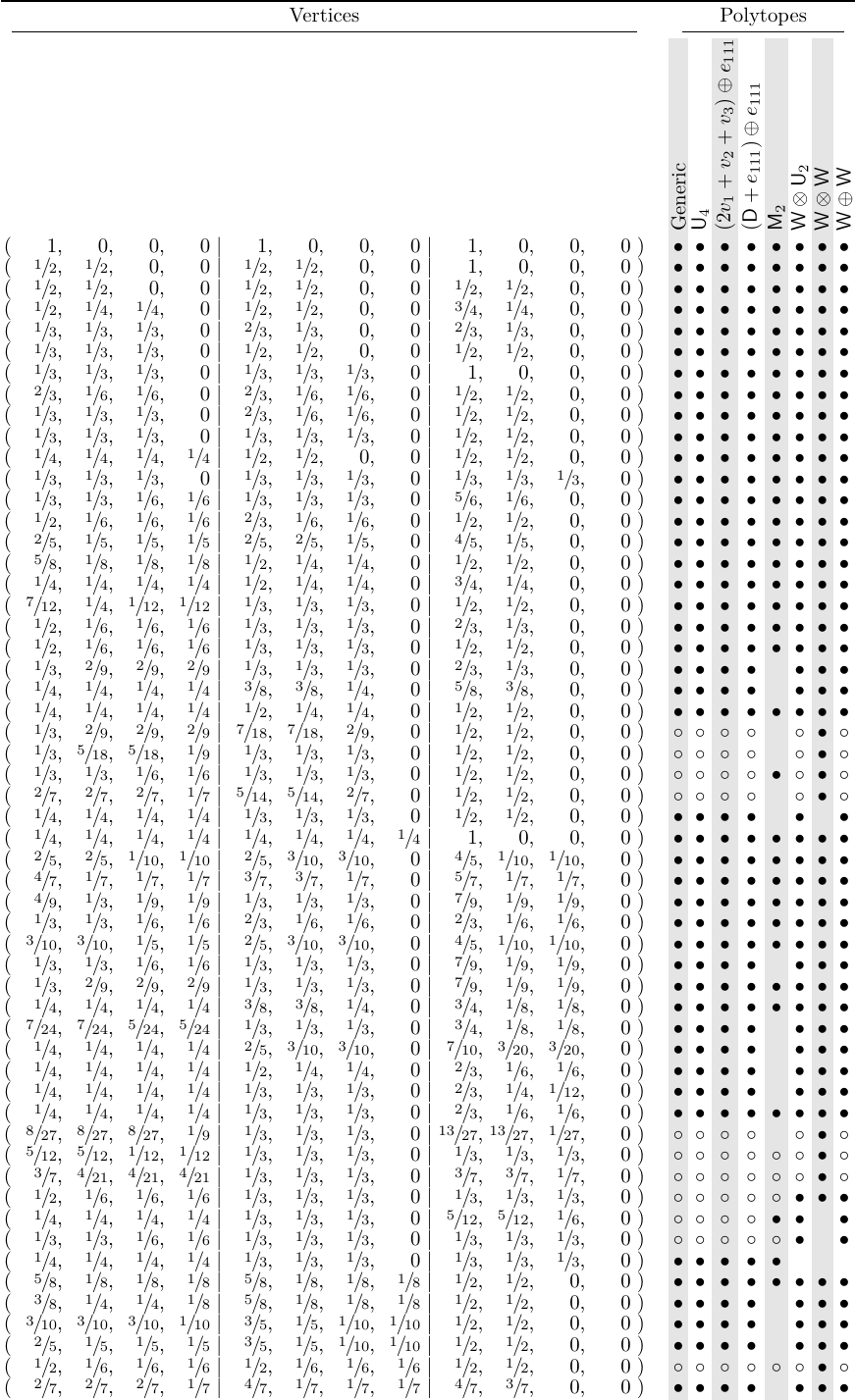}}}
\newpage
\begin{table}[H]
\makebox[\textwidth][c]{\includegraphics{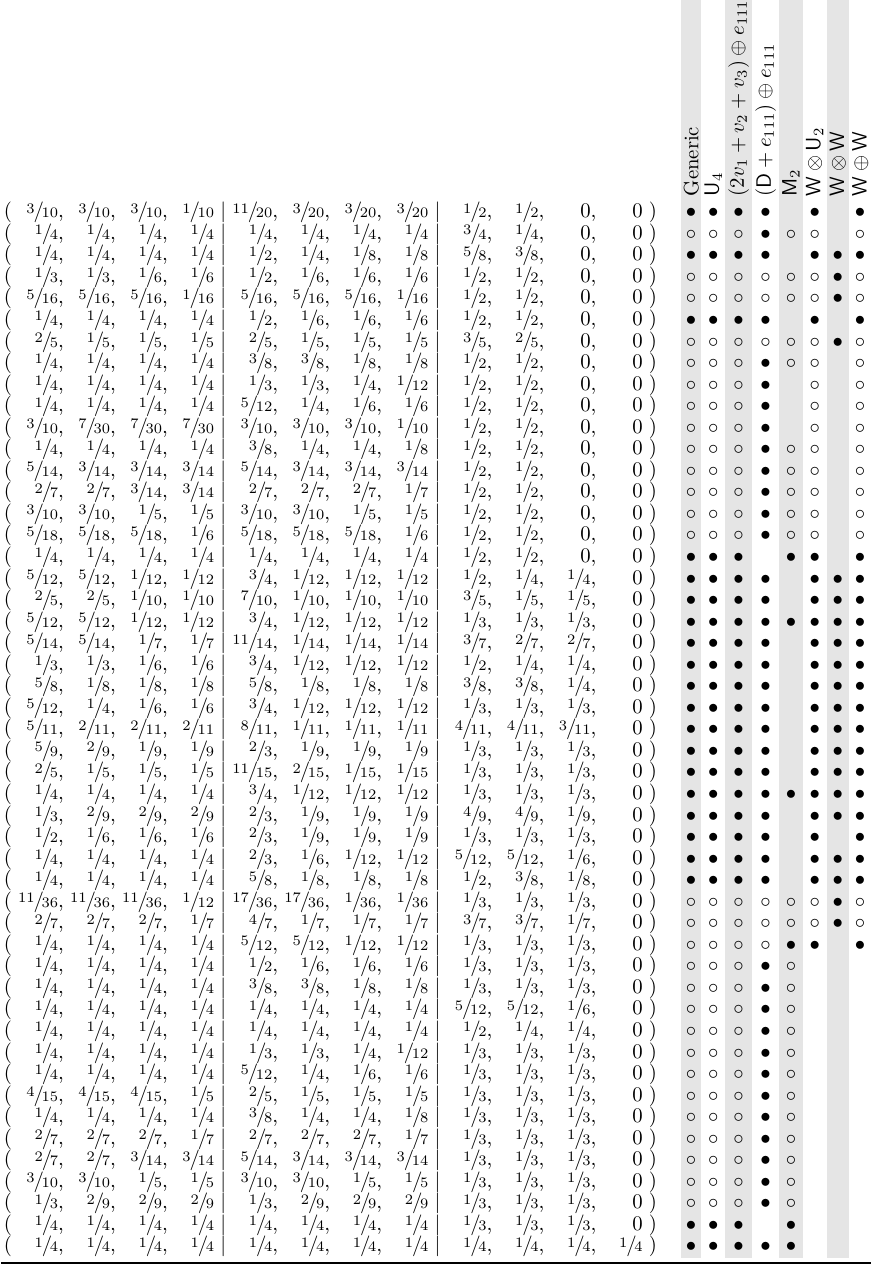}}
\caption{Vertex data of several moment polytopes in $\C^4 \ot \C^4 \ot \C^4$, see \cref{section:algorithms for 4x4x4}.
    Vertices are listed up to permutation of the systems (the polytopes are symmetric under system permutations).
A $\cm$ indicates the point is a vertex of the moment polytope of the indicated tensor. A $\cl$ indicates it is only an element.
Inclusion of all points has been verified following \cref{prop:scaling necessary epsilon}.
Exclusion of other points has been checked probabilistically following \cref{remark:probabilistic outer verification}.
The column labeled \emph{Generic} concerns the Kronecker polytope, which is the moment polytope of a generic tensor.
This data is available online in \cite{vandenBerg2025momentPolytopesGithub}.
}
\label{table:444 vertex data}
\end{table}

\section{Our implementation, number of inequalities and runtimes}
\label{subsection:tensor implementation}

Our reference implementation has been written in Sage \cite{sagemath}.
For working with polytopes (in particular for vertex enumeration) we used the Normaliz library \cite{normaliz}.
\Cref{algorithm:tensor weight-matrix-enumeration} and variations use FLINT \cite{flint} for computing integer kernels.
\Cref{algorithm:tensor checking-inequalities,algorithm:tensor checking-inequalities Fq} use Singular \cite{singular} to compute the Gr\"obner bases.
\Cref{algorithm:tensor checking-inequalities symbolic} uses the implementation of Buchbergers algorithm from Sage.
We use the \emph{Total degree reverse lexicographical ordering} (often called \emph{degrevlex}) for the total monomial order used when computing Gr\"obner bases (see \cref{definition:groebner-basis}). This is generally considered to be the optimal choice for most applications \cite{AdamsLoustaunauGroebner}, and we found this to be true for our applications as well.

\begin{remark}[Homotopy continuation]
\label{remark:homotopy continuation}
    We have briefly investigated homotopy continuation \cite{burgisser2013condition, bates2024numericalnonlinearalgebra} as an alternative to the Gr\"obner basis computations, via the implementation in \cite{breiding2018homotopyContinuationjl}.
    Although we decided against using homotopy continuation at all, future work may benefit from this technique and because of this we include this discussion here.

    Homotopy continuation presents some overhead for each computation,
    while the majority of Gr\"obner basis computations in Singular take very little time.
    Considering the large number of inequalities (see \cref{tab:inequality counts} and \cref{tab:attainability}), computing the Gr\"obner basis is preferred on average.
    For a few specific computations, Gr\"obner basis computations may take an unpractical number of time and memory resources.
    Here homotopy continuation could present an attractive numerical alternative to computing Gr\"obner bases.
    However, for the several cases we investigated,\footnote{These cases all concern polynomial systems corresponding to generic inequalities of the tensors $\MM_2$ and $\unit{4}$ in $\C^4\ot\C^4\ot\C^4$, which we can circumvent as described in \cref{subsection:filtering inequalities}.}
    where the Gr\"obner basis computation took an unpractical number of time or was terminated because of an out of memory error, also the homotopy continuation computation ran out of memory.
\end{remark}

In $\C^4\ot\C^4\ot\C^4$ there are sometimes a small number of Gr\"obner basis computations that take too long even over $\F_q$.
For such cases, we simply assume they are not equal to $\{1\}$. (This is not completely unreasonable, because longer runtimes mean the Gr\"obner basis is likely complicated and hence not trivial.)
This is harmless if our verification algorithm (\cref{subsection:verification}) still succeeds in the end.
Such cases typically compromised less than 1 per 100,000 of inequalities.

\subsection{Randomization}

For the modular arithmetic optimization (\cref{algorithm:tensor checking-inequalities Fq}), we generate primes in the range $[2^{30},2^{31})$. These are the largest primes supported by Singular.

Following \cref{theorem:probabilistic algorithm}, we use integer randomness to generate a generic element $T' = (A\ot B\ot C)T \in \G \cdot T$ with $(A,B,C) \in \G$.
As mentioned, our bounds on the required randomness are much larger than what is required in practice.
Since large numbers can present memory issues when computing Gr\"obner bases of $\Q$, we take relatively small values in this case.
The exact choice is then based on heuristics, and determined by observing when results became consistent across repeated executions of the algorithm.
Our reference implementation over $\Q$ uses random integers from the range $\{1,\ldots,1000\}$.
When working a finite fields $\F_q$ (as in \cref{subsection:4x4x4}), we instead generate each entry uniformly random from $\{0,\ldots,q-1\}$.
In both cases we use the build-in random number generation of Sage.

As we will discuss in \cref{section:C333}, we can use the randomized algorithm (\cref{algorithm:tensor algorithm randomized}) to determine the moment polytopes of all elements of $\C^3\ot\C^3\ot\C^3$ with certainty.
Hence we can experimentally validate our choices of randomization in this case.
We ran \cref{algorithm:tensor algorithm} 1000 times, using both Gr\"obner basis computations over both $\Q$ and $\F_q$ and the randomization of the prime $q$ and generic element $T'$ as described above, on the tensors in $\C^3\ot\C^3\ot\C^3$ listed in \cref{table:c333 unstable tensors info}.
We obtained the same results every iteration, which suggests effectiveness of the two heuristics.

\subsection{Inequality enumeration}

The following two tables show the number of inequalities obtained during each step in the inequality enumeration algorithms \cref{algorithm:tensor weight-matrix-enumeration-no-weyl} and \cref{algorithm:tensor weight-matrix-enumeration-no-symmetries}, and corresponding runtimes.

\begin{table}[H]
    \centering
    \renewcommand{\arraystretch}{1.1}
\begin{tabular}{crrrrrr}\toprule
           &  \multicolumn{2}{>{\centering\arraybackslash}p{10em}}{\textbf{Full rank weight matrices}} & \multicolumn{2}{>{\centering\arraybackslash}p{10em}}{\textbf{Inequalities, no symmetries}} & \multicolumn{2}{>{\centering\arraybackslash}p{10em}}{\textbf{Inequalities}}\\\cmidrule(lr){2-3}\cmidrule(lr){4-5}\cmidrule(lr){6-7}
    \textbf{Shape}  & Number & Time (s) & Number & Time (s) & Number & Time (s) \\
    \cmidrule(lr){1-7}
 $(2,2,2)$ & 8  & 0.0009%
           & 6  & 0.0005%
           & 29 & 0.0007%
           \\
 $(2,2,3)$ & 89 & 0.0011%
           & 65 & 0.0022%
           & 81 & 0.0008%
           \\
 $(2,3,3)$ & 798 & 0.0026%
           & 528 & 0.0170%
           & 345 & 0.0023%
           \\
 $(3,3,3)$ & 5499 & 0.0247%
           & 3196 & 0.1187%
           & 2845 & 0.0338%
           \\
 $(3,3,4)$ & 155145 & 0.3296%
           & 90366 & 3.7259%
           & 22867 & 0.3648%
           \\
 $(3,4,4)$ & 2838740 & 2.5581%
           & 1563246 & 14.3456%
           & 315305 & 7.3850%
           \\
 $(4,4,4)$ & 40456223 & 52.0675%
           & 20950156 & 166.2963%
           & 8109383 & 252.8900%
           \\\bottomrule
\end{tabular}
    \caption{Data on the inequality enumeration algorithms, \cref{algorithm:tensor weight-matrix-enumeration-no-weyl} and \cref{algorithm:tensor weight-matrix-enumeration-no-symmetries} (when all dimensions are equal), for different tensor shapes.
    Lists the number of output objects in every step of the algorithm and their execution times in seconds using 15 cores on a 4.90 GHz CPU.
    These steps are: enumeration of full rank standard weight matrices, determining inequalities by computing integer kernels, and restoring the symmetries.
    In the last step, the optimization ``Highest weight inclusion'' from subsection \cref{subsection:filtering inequalities}
    is applied.
    }
    \label{tab:inequality counts}
\end{table}

\begin{table}[H]
    \centering
    \renewcommand{\arraystretch}{1.1}
\begin{tabular}{crrrrrr}\toprule
           &  \multicolumn{2}{>{\centering\arraybackslash}p{10em}}{\textbf{Full rank weight matrices}} & \multicolumn{2}{>{\centering\arraybackslash}p{10em}}{\textbf{Inequalities, no symmetries}} & \multicolumn{2}{>{\centering\arraybackslash}p{10em}}{\textbf{Inequalities}}\\\cmidrule(lr){2-3}\cmidrule(lr){4-5}\cmidrule(lr){6-7}
    \textbf{Shape}  & Number & Time (s) & Number & Time (s) & Number & Time (s) \\
    \cmidrule(lr){1-7}
 $(2,2)$ & 2    & 0.0003%
           & 2  & 0.0004%
           & 9  & 0.0006%
           \\
 $(2,2,2)$ & 8  & 0.0009%
           & 6  & 0.0005%
           & 29 & 0.0007%
           \\
 $(2,2,2,2)$ & 73  & 0.0017%
             & 48  & 0.0019%
             & 185 & 0.0016%
           \\
 $(2,2,2,2,2)$ & 1810 & 0.0318%
               & 1119 & 0.0469%
               & 3879 & 0.0461%
           \\
 $(2,2,2,2,2,2)$ & 131226 & 3.8156%
                 & 81305  & 4.2534%
                 & 281309 & 6.4527%
           \\\bottomrule
\end{tabular}
    \caption{Data on the inequality enumeration algorithms, as in \cref{tab:inequality counts}.}
    \label{tab:inequality counts qubits}
\end{table}

\subsection{Attainability verification}

We provide data on the runtime of attainability verification using \cref{algorithm:tensor checking-inequalities} and \cref{algorithm:tensor checking-inequalities Fq} as used in \cref{algorithm:tensor algorithm borel} and \cref{algorithm:tensor algorithm borel Fq}.

\begin{table}[H]
    \centering
    \setlength{\tabcolsep}{5.5pt}
    \renewcommand{\arraystretch}{1.1}
    \begin{tabular}{lccccccrr}\toprule
                   & \multicolumn{5}{c}{\textbf{Inequalities}}     & \textbf{Vertices} & \multicolumn{2}{c}{\textbf{Runtimes}} \\\cmidrule(lr){2-6}\cmidrule(lr){7-7}\cmidrule(lr){8-9}
        \textbf{Tensor}     & All      & Not generic & Maxranks     & Attainable & Final & Final & $\Q$      &  $\F_q$ \\
        \cmidrule(lr){1-9}
        $\unit{3}$     & 2845     & 2187        & 355           & 0          & 45    & 33       & 0.254     & 0.239       \\
        $\nurmievT{4}$      & 2845     & 2187        & 355           & 20         & 52    & 53       & 0.264     & 0.237       \\
        $\nurmievT{9}$      & 2845     & 2187        & 736           & 292        & 25    & 18       & 0.293     & 0.263       \\
        $\unit{4}$     & 8109383  & 7139405     & 1102518       & 0          & 270   & 328      & - & 20:10            \\
        $\MM_2$ & 8109383  & 7139405     & 1102518       & 1227       & 129   & 181      & - & 3:06                \\\bottomrule
    \end{tabular}
    \caption{
    Data concerning attainability.
    The tensors $\unit{r}$ are the unit tensors of rank $r$.
    Tensor $\nurmievT{4}$ and $\nurmievT{9}$ are taken from \cref{table:c333 unstable tensors info}. $\MM_2$ is the $2 \times 2$ matrix multiplication tensor (see \cref{equation:MM2}).
    The table contains the number of inequalities at every step of the algorithm, including many optimizations from \cref{subsection:filtering inequalities}.
    The columns correspond to:
    all possible inequalities of the shape of the tensor, the inequalities that are not valid for the generic polytope, the inequalities that contain the maxrank point corresponding to the tensor (\cref{lemma:maxranks}), the inequalities that are attainable according to \cref{algorithm:tensor checking-inequalities Fq}, and the final number of irredundant inequalities after intersecting with the generic moment polytope. Additionally, the number of vertices is listed, as well as the runtimes in seconds or minutes, using 15 cores on a 4.90 GHz CPU. These specify the runtimes for determining the attainable inequalities both over $\Q$ and over $\F_q$, for some large randomly generated prime $q$.
    A dash indicates computations could not be completed within 10 hours.
    For all tensors, the respective computations over $\Q$ and $\F_q$ resulted in the same sets of inequalities.
}
    \label{tab:attainability}
\end{table}

\bibliographystyle{alphaurl}
\bibliography{references}

\end{document}